\documentclass[10pt,oneside,reqno]{amsart}
\usepackage{hyperref}
\usepackage[nonewpage]{imakeidx}
\makeindex[title=Index of symbols]
\usepackage[all,pdf]{xy}
\usepackage{amsmath}
\usepackage{amssymb}
\usepackage{amsxtra}
\usepackage{amsthm}

\allowdisplaybreaks[4]
\usepackage{stmaryrd}
\usepackage{bbm}
\usepackage{upgreek}
\usepackage{mathtools}
\usepackage{thmtools}
\usepackage{geometry}
\usepackage{enumitem}
\usepackage{comment}
\usepackage{microtype}
\usepackage{mathpazo}
\usepackage{domitian}
\usepackage[T1]{fontenc}

\AtBeginDocument{
 \DeclareSymbolFont{AMSb}{U}{msb}{m}{n}
 \DeclareSymbolFontAlphabet{\mathbb}{AMSb}
 }
 
\let\SSec\S

\newcommand\lopen{\mathopen{]}}
\newcommand\ropen{\mathclose{[}}

\newcommand\rmd{\mathrm{d}}
\newcommand\rme{\mathrm{e}}
\newcommand\rmi{\mathrm{i}}

\newcommand\cA{\mathcal{A}}

\newcommand\cC{\mathcal{C}}

\newcommand\cH{\mathcal{H}}
\newcommand\cI{\mathcal{I}}
\newcommand\cK{\mathcal{K}}
\newcommand\cL{\mathcal{L}}

\newcommand\cO{\mathcal{O}}

\newcommand\cX{\mathcal{X}}

\renewcommand\AA{\mathbb{A}}

\newcommand\CC{\mathbb{C}}

\newcommand\FF{\mathbb{F}}

\newcommand\QQ{\mathbb{Q}}
\newcommand\RR{\mathbb{R}}
\renewcommand\SS{\mathbb{S}}

\newcommand\ZZ{\mathbb{Z}}

\newcommand\mf[1]{\mathfrak{#1}}

\newcommand\A{\mathsf{A}}
\newcommand\B{\mathsf{B}}

\newcommand\G{\mathsf{G}}
\renewcommand\H{\mathsf{H}}

\newcommand\M{\mathsf{M}}
\newcommand\N{\mathsf{N}}

\renewcommand\S{\mathsf{S}}
\newcommand\T{\mathsf{T}}

\newcommand\Z{\mathsf{Z}}

\newcommand\SL{\mathsf{SL}}
\newcommand\GL{\mathsf{GL}}

\newcommand\triv{\mathbbm{1}}
\newcommand\dpii{2\uppi\rmi}
\newcommand\bs{\backslash}

\renewcommand\Re{\mathop{\mathrm{Re}}}

\DeclareMathOperator\Tr{Tr}
\DeclareMathOperator\sgn{sgn}
\DeclareMathOperator\supp{supp}

\DeclareMathOperator\res{res}
\DeclareMathOperator\ord{ord}
\DeclareMathOperator\Res{Res}
\DeclareMathOperator\Gal{Gal}
\newcommand\fin{\mathrm{fin}}
\newcommand\el{\mathrm{ell}}
\newcommand\Kl{\mathrm{Kl}}
\DeclareMathOperator\vol{vol}
\DeclareMathOperator\orb{orb}
\renewcommand\leq{\leqslant}
\renewcommand\geq{\geqslant}

\newcommand\legendresymbol[2]{\genfrac{(}{)}{}{}{#1}{#2}}
\newcommand\expcharacter[2][]{\rme_{#1}\!\left(#2\right)}

\makeatletter
\def\subsection{\@startsection{subsection}{2}
  \z@{3pt\@plus0pt}{-.5em}%
  {\normalfont\bfseries}}
\makeatother

\makeatletter
\def\@seccntformat#1{
  \protect\textup{\protect\@secnumfont
    \ifnum\pdfstrcmp{subsection}{#1}=0 \bfseries\fi
    \csname the#1\endcsname
    \protect\@secnumpunct
  }%
}  
\makeatother

\newtheoremstyle{THEOREM}
{2.5pt}
{2pt}
{\itshape}
{}
{\bfseries}
{.}
{.5em}
{\thmname{#1}\thmnumber{ #2}\thmnote{ (#3)}}

\newtheoremstyle{DEFINITION}
{2.5pt}
{2pt}
{}
{}
{\bfseries}
{.}
{.5em}
{\thmname{#1}\thmnumber{ #2}\thmnote{ (#3)}}

\newtheoremstyle{EXERCISE}
{2pt}
{2pt}
{}
{}
{\scshape}
{.}
{.5em}
{\thmname{#1}\thmnumber{ #2}\thmnote{ (#3)}}

\theoremstyle{THEOREM}
\newtheorem{theorem}{Theorem}[section]
\newtheorem{lemma}[theorem]{Lemma}
\newtheorem{proposition}[theorem]{Proposition}
\newtheorem{corollary}[theorem]{Corollary}

\theoremstyle{DEFINITION}
\newtheorem{definition}[theorem]{Definition}

\theoremstyle{EXERCISE}
\newtheorem{remark}[theorem]{Remark}

\numberwithin{equation}{section}

\makeatletter
\renewenvironment{proof}[1][\proofname]{\par
  \vspace{-6pt}
  \pushQED{\qed}
  \normalfont \topsep6\p@\@plus6\p@\relax
  \trivlist
  \item[\hskip\labelsep\rmfamily\bfseries
    #1\@addpunct{:}]\ignorespaces
}{
  \popQED\endtrivlist\@endpefalse
  \vspace{-6pt}
}
\makeatother

\makeatletter
\newenvironment{insertproof}[1][\proofname]{
  \par
  \vspace{-6pt}
  \pushQED{\qed}
  \normalfont \topsep6\p@\@plus6\p@\relax
  \trivlist
  \item[\hskip\labelsep\rmfamily\bfseries
    #1\@addpunct{:}]\ignorespaces
}{
  \popQED\endtrivlist\@endpefalse
  \vspace{-6pt}
}
\makeatother

\begin{document}
\setlength \lineskip{3pt}
\setlength \lineskiplimit{3pt}
\setlength \parskip{1pt}
\setlength \partopsep{0pt}

\title{Beyond endoscopy for $\GL_2$ over $\QQ$ with ramification 1: Poisson summation}
\author{Yuhao Cheng}
\address{Qiuzhen College, Tsinghua University, 100084, Beijing, China}
\email{chengyuhaomath@gmail.com}
\date{25 July, 2026}
\begin{abstract}
At the beginning of this century, Langlands introduced a strategy known as \emph{Beyond Endoscopy} to attack the principle of functoriality. Altu\u{g} studied $\GL_2$ over $\QQ$ in the unramified setting. The first step involves isolating specific representations, especially the residual part of the spectral side, in the elliptic part of the geometric side of the trace formula. We generalize this step to the case with ramification at $S=\{\infty,q_1,\dots,q_r\}$ with $2\in S$, thereby fully resolving the problem of isolating these representations over $\QQ$ that had remained unresolved for over a decade. Such a formula that isolates the specific representations is derived by modifying Altu\u{g}'s approach. We use the approximate functional equation to ensure the validity of the Poisson summation formula. Then, we compute the residues of specific functions to isolate the desired representations.
\end{abstract}

\maketitle
 
\tableofcontents

\section{Introduction}
\subsection{Background}
Let $\G$ be a connected reductive algebraic group over the field $\QQ$ of rational numbers and let $\AA$ denote the ad\`ele ring of $\QQ$. Let $\prescript{L}{}\G$ be the $L$-group of $\G$. Let $\pi$ be a cuspidal automorphic representation of $\G$ over $\QQ$ and $\rho$ be a finite dimensional complex representation of $^L\G$. For a finite set $S$ of places of $\QQ$ containing the archimedean place 
$\infty$, the partial $L$-function $L^{S}(s,\pi,\rho)$ can be defined (cf. \cite{getz2024} for details) and expressed as a Dirichlet series
\[
L^{S}(s,\pi,\rho)=\sum_{\substack{n=1\\ \gcd(n,S)=1}}^{+\infty}\frac{a_{\pi,\rho}(n)}{n^s}
\]
for $\Re s$ sufficiently large.

In general, Langlands' \emph{Beyond Endoscopy} proposal \cite{langlands2004} aims to provide a formula for
\begin{equation}\label{eq:beyondendoscopy}
\sum_{\pi}m_{\pi}\prod_{v\in S}\Tr(\pi_v(f_v))\ord_{s=1}L^{S}(s,\pi,\rho)
\end{equation}
for cuspidal representations $\pi$ and suitable functions $f_v$ on $\G(\QQ_v)$ for each $v\in S$, where $m_\pi$ denotes the multiplicity of $\pi$ in $L^2_{\mathrm{cusp}}(\G(\QQ)\bs \G(\AA)^1)$. Using this, we can analyze $\ord_{s=1}L(s,\pi,\rho)$ and determine whether $\pi$ arises from an automorphic representation of a ``smaller" group $\H$. More precisely, given an automorphic representation $\pi$ of $\G$ (corresponding to a map $\cL\to \prescript{L}{}\G$ by expected Langlands correspondence, where $\cL$ denotes the $L$-packet) and a morphism $\phi\colon\prescript{L}{}\H\to \prescript{L}{}\G$,  we want to determine whether there is a map $\cL\to \prescript{L}{}\H$ making the diagram
\[
      \xymatrix
      {
         & \ar@{-->}_{?}[ld]\mathcal{L}\ar^{\pi}[rd] &  \\
      \prescript{L}{}\H \ar_{\phi}[rr]&            & \prescript{L}{}\G 
      }
\]
commute.

By applying the trace formula, we are able to evaluate
\[
\sum_{\pi}m_\pi a_{\pi,\rho}(p)\prod_{v\in S}\Tr(\pi_v(f_v))=I_{\mathrm{cusp}}(f^{p,\rho}),
\]
where $f^{p,\rho}=\bigotimes_{v\in S}f_v\otimes\bigotimes_{v\notin S}'f_v^{p,\rho}$, with $f_v^{p,\rho}$ spherical for $v\notin S$, and $\Tr(\pi_p(f_p^{p,\rho}))=a_{\pi,\rho}(p)$ and $\Tr(\pi_\ell(f_\ell^{p,\rho}))=1$ if $\ell\notin S\cup\{p\}$.

If $L^S(s,\pi,\rho)$ admits meromorphic continuation on $\Re s>1-\delta$ with neither zeros nor poles on the vertical line $s=1+\rmi t$ except at $s=1$, then by the consequence of the Ikehara theorem (see \cite[Proposition 9.15]{marius2014} for example), we expect the following asymptotic formula:
\[
\lim_{X\to +\infty}\frac{1}{X}\sum_{\substack{p<X\\ p\notin S}}a_{\pi,\rho}(p)\log p=\ord_{s=1}L^S(s,\pi,\rho),
\]
where $p$ runs over all primes less than $X$.
For example, for $L^S(s,\pi,\rho)=\zeta(s)$, the above formula reduces to
\[
\lim_{X\to +\infty}\frac{1}{X}\vartheta(X)=\lim_{X\to +\infty}\frac{1}{X}\sum_{p<X}\log p =1,
\]
which is equivalent to the prime number theorem.

If such asymptotic formula holds, \eqref{eq:beyondendoscopy} can be rewritten as
\[
\frac{1}{X}\sum_{\pi}m_\pi\prod_{v\in S}\Tr(\pi_v(f_v))\sum_{\substack{p<X\\ p \notin S}}a_{\pi,\rho}(p)\log p .
\] 

For functions $f^{p,\rho}$ above, we have
\[
\frac{1}{X}\sum_{\pi}m_\pi\prod_{v\in S}\Tr(\pi_v(f_v))\sum_{\substack{p<X\\ p \notin S}}a_{\pi,\rho}(p)\log p 
=\frac{1}{X}\sum_{\substack{p<X\\ p \notin S}}\log pI_{\mathrm{cusp}}(f^{p,\rho}).
\] 
Therefore, we expect that 
\[
\lim_{X\to +\infty}\frac{1}{X}\sum_{\substack{p<X\\ p \notin S}}\log pI_{\mathrm{cusp}}(f^{p,\rho})=\sum_{\pi}m_\pi \prod_{v\in S}\Tr(\pi_v(f_v))\ord_{s=1}L^S(s,\pi,\rho),
\]
where $\pi$ runs over all cuspidal representations of $\G$ over $\QQ$. 

We can consider a more general setting, which was proposed by Sarnak \cite{sarnak2001}. For $\gcd(n,S)=1$, we have
\[
\sum_{\pi}m_\pi a_{\pi,\rho}(n)\prod_{v\in S}\Tr(\pi_v(f_v))=I_{\mathrm{cusp}}(f^{n,\rho}),
\]
where $f^{n,\rho}=\bigotimes_{v\in S}f_v\otimes\bigotimes_{v\notin S}'f_v^{n,\rho}$, with $f_v^{n,\rho}$ spherical for $v\notin S$, and $\Tr(\pi_p(f_p^{n,\rho}))=a_{\pi,\rho}(p^{v_p(n)})$ for all $p\notin S$. Using the Ikehara theorem, we expect to have
\[
\lim_{X\to +\infty}\frac{1}{X}\sum_{\substack{n<X\\ \gcd(n,S)=1}}a_{\pi,\rho}(n)=\res_{s=1}L^S(s,\pi,\rho)
\]
if $L^S(s,\pi,\rho)$ has at most a simple pole at $s=1$, and we expect that
\[
\lim_{X\to +\infty}\frac{1}{X}\sum_{\substack{n<X\\ \gcd(n,S)=1}}I_{\mathrm{cusp}}(f^{n,\rho})=\sum_{\pi}m_\pi \prod_{v\in S}\Tr(\pi_v(f_v))\res_{s=1}L^S(s,\pi,\rho)
\]
in this case, where $\pi$ runs over all cuspidal representations of $\G(\AA)$. 

If $L^S(s,\pi,\rho)$ has a pole of higher order at $s=1$, we can also detect them by considering the asymptotic formula when $X\to +\infty$. For example, if $L^S(s,\pi,\rho)$ has a pole of order $k$ for some $\pi$, the asymptotic formula would involve $X\log^{k-1}X$.

For more background and problems of Beyond Endoscopy, one may refer \cite{langlands2004,langlands2010,arthur2017,espinosa2022,altug2024}.

We now assume that $\G=\GL_2$, $\rho=\mathrm{Sym}^d$. For any prime number $p$ and $m\in \ZZ_{\geq 0}$, we define \index{xpm@$\cX_p^{m}$}
\[
\cX_p^{m}=\{X\in \M_2(\ZZ_p)\,|\, \mathopen{|}\det X\mathclose{|}_p = p^{-m}\}.
\]
For example, if $m=0$, $\cX_p^{m}$ is just $\cK_p=\GL_2(\ZZ_p)$. By Hecke operator theory, we can choose $f^{p,d}$ to be $\bigotimes_{v\in \mf{S}}f^{p,d}_v$, where $\mf{S}$ \index{smf@$\mf{S}$} denotes the set of places of $\QQ$, and
\begin{enumerate}[itemsep=0pt,parsep=0pt,topsep=2pt,leftmargin=0pt,labelsep=3pt,itemindent=9pt,label=\textbullet]
  \item If $v=p$, $f^{p,d}_p=p^{-d/2}\triv_{\cX_p^{d}}$.
  \item If $v=\ell$ is a prime different from $p$, $f^{p,d}_\ell=\triv_{\cK_\ell}$.
  \item If $v=\infty$,  $f^{p,d}_\infty=f_\infty\in C^\infty(Z_+\bs \G(\RR))$ such that the orbital integral is compactly supported modulo $Z_+$, and other than this condition it is arbitrary, where $Z_+$ \index{zplus@$Z_+$} denotes the identity component of the center of $\G(\RR)$.
\end{enumerate}
Note that $f_\infty$ is not assumed to be compactly supported modulo $Z_+$ to ensure some generality. For example, one may take $f_\infty$ to be a matrix coefficient of some discrete series representation with weight $>2$.

In this case, $\pi$ runs over all unramified cuspidal representations. Venkatesh \cite{venkatesh2004} established an asymptotic formula for summing over $n<X$ and $d\leq 2$, using the Petersson-Kuznetsov trace formula. As a new approach, Altu\u{g} \cite{altug2015,altug2017,altug2020} used the Arthur-Selberg trace formula to prove an asymptotic formula for summing over $n<X$, $d=1$ and with $f_\infty$ being the matrix coefficient of the weight $k$ discrete series for $k>2$ even. More precisely, Altu\u{g} proved that
\[
\sum_{n<X}\Tr(T_k(n))\ll_{k,\varepsilon} X^{\frac{31}{32}+\varepsilon},
\]
where $T_k(n)$ denotes the (normalized) $n^{\mathrm{th}}$ Hecke operator acting on $S_k(\SL_2(\ZZ))$, the space of holomorphic cusp forms of weight $k$ for the modular group $\Gamma=\SL_2(\ZZ)$.

The first step in proving this theorem is to isolate the specific representations, especially the residual part of the spectral side (in this case, there is only one, namely the trivial representation), in the elliptic part of the trace formula for the unramified case, which is precisely the content of \cite{altug2015}:
\begin{theorem}[Altu\u{g}, \cite{altug2015}]\label{thm:maintheoremunramify}
We have
\[
I_{\el}(f^{p,d})=\Tr(\triv(f^{p,d}))-\frac{1}{2}\Tr(\xi_0(f^{p,d}))-S(\square)+S(0)+S(\xi\neq 0).
\]
See \cite[Theorem 1.1]{altug2015} for the precise definitions of these terms.
\end{theorem}

Recently, this theorem was generalized to $\GL_3$ \cite{deng2026,lee2026}.
Moreover, for $\GL_2$ over a number field case, \cite{emory2024} generalized this theorem under some technical assumptions.

\subsection{The main result in this paper and proof strategy}
In this paper, we will generalize the main result in \cite{altug2015} by considering the case over $\QQ$ with ramification and generalize the function $f^{n,1}$ for any $n$ such that $\gcd(n,S)=1$, answering the question raised in \cite[Section 4.1]{espinosa2022}. In this paper, we only consider the case $d=1$ so that $\rho$ is the standard representation. For higher symmetric powers we have $a_{\pi,\mathrm{Sym}^j}(p)=a_{\pi}(p^j)$ for primes $p$. Hence we can also obtain the result for the $p^{\mathrm{th}}$ Fourier coefficient of higher symmetric powers. 

We consider $S=\{\infty,q_1,\dots,q_r\}$ for primes $q_1,\dots,q_r$ such that $2\in S$ (the reason why we assume that $2\in S$ will be clear later, especially in \autoref{rem:kloosterman}), and a corresponding function $f^{n}=\bigotimes_{v\in \mf{S}}'f^n_v$ such that the local components at nonarchimedean places in $S$ are arbitrary smooth and compactly supported functions. 
Specifically, for $\gcd(n,S)=1$, the function $f^n$ is defined as follows:
\begin{enumerate}[itemsep=0pt,parsep=0pt,topsep=2pt,leftmargin=0pt,labelsep=3pt,itemindent=9pt,label=\textbullet]
  \item If $v=p\notin S$, we choose $f^n_v=p^{-n_p/2}\triv_{\cX_p^{n_p}}$, where $n_p=v_p(n)$.
  \item If $v=q_i\in S$ and is a prime, we choose $f^n_v=f_{q_i}$ to be an arbitrary function in $C_c^\infty(\G(\QQ_{q_i}))$.
  \item If $v=\infty$, we choose $f^n_v=f_\infty\in C^\infty(Z_+\bs \G(\RR))$ such that its orbital integrals are compactly supported modulo $Z_+$.
\end{enumerate}
In this case, $f_v^{n}$ is spherical for $v\notin S$, and $\Tr(\pi_p(f_p^{n}))=a_{\pi}(p^{n_p})$ for all $p\notin S$.

With this setup, we will prove the following theorem in this paper:
\begin{theorem}\label{thm:maintheoremramify}
We have
\[
I_{\el}(f^n)=\sum_{\mu}\Tr(\mu(f^n))-\frac{1}{2}\sum_{\mu}\Tr((\xi_0\otimes\mu)(f^n))-\Sigma(\square)+\Sigma(0)+\Sigma(\xi\neq 0),
\]
where $\mu$ runs over all $1$-dimensional representations of $\G(\QQ)\bs \G(\AA)^1=Z_+\G(\QQ)\bs \G(\AA)$. See \autoref{thm:finaltrace} for the precise definitions of these terms.
\end{theorem}
Thus, we have isolated the \emph{$1$-dimensional term} $\sum_{\mu}\Tr(\mu(f^n))$ and the \emph{Eisenstein term} $\sum_{\mu}\Tr((\xi_0\otimes\mu)(f^n))$ in the elliptic part of the trace formula. (The names of these terms will be explained at the end of \autoref{sec:isolation}.) The isolated representations do \emph{not} merely include the trivial representation and the Eisenstein series as in \cite{altug2015}, which will be explained in \autoref{rem:1dim}.

The proof 
is divided into the following steps:
\begin{enumerate}[itemsep=0pt,parsep=0pt,topsep=0pt,leftmargin=0pt,labelsep=3pt,itemindent=9pt,label=\textbullet]
  \item \textbf{Elementary computation.} We explicitly compute the local orbital integral in the unramified case and the germ expansion in the ramified case of the elliptic part of the trace formula. Then, we take the product of all the local data and perform some simplifications.
  \item \textbf{Poisson summation.}   We apply the Poisson summation formula to the derived expression. The main difference is that we use "semilocal" analysis: the analysis on $\QQ_S=\RR\times \QQ_{q_1}\times \dots \times \QQ_{q_r}$, rather than just $\RR$ as in the unramified case. As in \cite{altug2015}, some smoothing issues arise. We use the approximate functional equation for the partial Zagier $L$-function to smooth the function at the archimedean place and conduct a technical analysis to handle the functions at the nonarchimedean places.
  \item \textbf{Isolation of special terms.} After applying the Poisson summation formula, we isolate the $\xi=0$ term. Then, we apply the contour shifting method. To do this, we verify the analytic continuation of a Kloosterman-type Dirichlet series, which can be computed explicitly. Finally, we compute the residues of specific functions to isolate the desired terms.
  \item \textbf{Computation of special terms.} We use the Weyl integration formula to compute the local terms of the traces of $1$-dimensional representations and their twists with the Eisenstein series. Next, we take the product of all local terms and use the Fourier inversion formula to show that the special terms are precisely the $1$-dimensional term and the Eisenstein term, respectively.
\end{enumerate}

Throughout the paper, we encounter numerous nonarchimedean and global analogs of Altu\u{g}'s approach, which addressed archimedean or unramified cases. Readers familiar with \cite{altug2015} will notice many similarities. A comparison is provided in the following table, where the italicized texts refer to \cite{altug2015}.

\begin{center}
\begin{tabular}{|c|c|c|}
\hline
\textbf{Archimedean/Unramified Case} & \textbf{Nonarchimedean Case} & \textbf{Global/Ramified Case}\\
\hline
\emph{Theorem 1.1} (\autoref{thm:maintheoremunramify})& N/A & \autoref{thm:maintheoremramify} (\autoref{thm:finaltrace})\\
\hline
\emph{Section 2.2.3} (\autoref{thm:archimedeanintegral})& \autoref{cor:shalika} & N/A\\
\hline
\emph{Section 2.2.4} & N/A & \autoref{thm:ellipticpart}\\
\hline
\emph{Proposition 3.1} (\autoref{thm:funceqndeltaunramify}) & N/A & \autoref{thm:funceqndelta}\\
\hline
\emph{Proposition 3.4} & N/A & \autoref{thm:afe}\\
\hline
\emph{Lemma 4.1} (\autoref{lem:smooth}) & Proof in \autoref{lem:rapiddecay} & \autoref{lem:rapiddecay}\\
\hline
\emph{Theorem 4.2} & N/A & \autoref{thm:ellipticpoisson}\\
\hline
\emph{Corollary 5.4} & N/A & \autoref{cor:dirichletglobal}\\
\hline
\emph{Theorem 6.1} & N/A & \autoref{thm:tracexi0}\\
\hline
\emph{1$^{st}$ equality in Lemma 6.2} (\autoref{prop:archimedeantrace}) & \autoref{prop:ramifiedtrace} & \autoref{thm:maintermcompute}\\
\hline
\emph{2$^{nd}$ equality in Lemma 6.2} (\autoref{prop:archimedeantraceeisenstein}) & \autoref{prop:ramifiedtraceeisenstein} & \autoref{thm:eisensteincompute}\\
\hline
\end{tabular}
\end{center}

\subsection{Notations}
\begin{enumerate}[itemsep=0pt,parsep=0pt,topsep=0pt,leftmargin=0pt,labelsep=3pt,itemindent=9pt,label=\textbullet]
  \item $\# X$ denotes the number of elements in a set $X$.
  \item For $A\subseteq X$, $\triv_A$ denotes the characteristic function of $A$, defined by $\triv_A(x)=1$ for $x\in A$ and $\triv_A(x)=0$ for $x\notin A$.
  \item $\triv$ also denotes the trivial character or the trivial representation.
  \item For $x\in \RR$, $\lfloor x\rfloor$ denotes the greatest integer that is less than or equal to $x$.
  \item We often use the notation $a\equiv b\,(n)$ to denote $a\equiv b\pmod n$.
  \item For $D\equiv 0,1\pmod 4$, $\legendresymbol{D}{\cdot}$ denotes the Kronecker symbol.
  \item If $R$ is a ring (which we \emph{always} assume to be commutative with $1$), $R^\times$ denotes its group of units.
  \item $\SS^1$ denotes the group of complex numbers with absolute value $1$. 
  \item For any locally compact abelian group $G$, $\widehat{G}$ or $G\sphat\ $ denotes its Pontryagin dual, which is the set of continuous homomorphisms from $G$ to $\SS^1$, equipped with the compact-open topology.
  \item If $E$ is a number field or a nonarchimedean local field, $\cO_E$ denotes its ring of integers.
  \item For any prime $p$, \index{kp@$\cK_p$} $\cK_p=\G(\ZZ_p)$ denotes the standard maximal compact subgroup of $\G(\QQ_p)$, and \index{ip@$\cI_p$}
  \[
  \cI_p=\left\{\begin{pmatrix}
               a & b \\
               c & d 
             \end{pmatrix}\in \G(\ZZ_p)\,\middle|\, p\mid c\right\}
  \]
  denotes the Iwahori subgroup of $\G(\QQ_p)$.
  \item $\mf{S}$ denotes the set of places of $\QQ$.
  \item For $S=\{\infty, q_1,\dots,q_r\}$, $\ZZ^S$ denotes the ring of $S$-integers in $\QQ$, that is, \index{zuppers@$\ZZ^S$}
    \[
    \ZZ^S=\{\alpha\in \QQ\ |\ v_\ell(\alpha)\geq 0\ \text{for all}\ \ell\notin S\}.
    \]
    Additionally, we define \index{qs@$\QQ_S$}\index{qsfin@$\QQ_{S_\fin}$}
    \[
    \QQ_S=\prod_{v\in S}\QQ_v=\RR\times \QQ_{q_1}\times\dots\times\QQ_{q_r}\quad\text{and}\quad\QQ_{S_\fin}=\prod_{v\in S_\fin}\QQ_v=\QQ_{q_1}\times\dots\times\QQ_{q_r},
    \]
    and \index{zsfin@$\ZZ_{S_\fin}$}
    \[
    \ZZ_{S_\fin}=\prod_{v\in S_\fin}\ZZ_v=\ZZ_{q_1}\times\dots\times\ZZ_{q_r}.
    \]
  \item For $n\in \ZZ$, we write $\gcd(n,S)=1$ or $n\in \ZZ_{(S)}$ \index{zlowers@$\ZZ_{(S)}$} if 
  $p\nmid n$ for all $p\in S$. 
  We write $n\in \ZZ_{(S)}^{>0}$ \index{zlowersg@$\ZZ_{(S)}^{>0}$} if additionally $n>0$.
  \item For any prime $p$, $v_p$ denotes the $p$-adic valuation. For $a\in\QQ$, we define 
    $a_{(p)}=p^{v_p(a)}$ and
    $a^{(p)}=a/a_{(p)}$. Moreover, we define \index{alowerq@$a_{(q)}$} \index{aupperq@$a^{(q)}$} 
    \[
    a_{(q)}=\prod_{i=1}^{r}q_i^{v_{q_i}(a)}\quad\text{and}\quad a^{(q)}=\frac{a}{a_{(q)}}=\prod_{p\notin S}p^{v_{p}(a)}.
    \]
  \item $\Gamma(s)$ denotes the gamma function, defined by
  \[
  \Gamma(s)=\int_{0}^{+\infty}\rme^{-x}x^s\frac{\rmd x}{x}
  \]
  for $\Re s>0$, analytically continued to $\CC$. $\zeta(s)$ denotes the Riemann zeta function, defined by
  \[
    \zeta(s)=\sum_{n=1}^{+\infty}\frac{1}{n^s}
  \]
  for $\Re s>1$, analytically continued to $\CC$.
  \item $(\sigma)$ denotes the vertical contour from $\sigma-\rmi\infty$ to $\sigma+\rmi\infty$.
  \item We define \index{ex@$\rme(x)$} \index{einftyx@$\rme_\infty(x)$} $\rme(x)=\rme_\infty(x)=\rme^{\dpii x}$. For a prime $p$ and $x\in \QQ_p$, we define $\langle x\rangle_p$ to be the "fractional part" of $x$. Namely, if $x=\sum_{n\geq -N}a_np^n\in \QQ_p$, then 
      \[
      \langle x\rangle_p=\sum_{n=-N}^{-1}a_np^n\in \QQ.
      \]
      We then define \index{epy@$\rme_p(y)$}$\rme_p(x)=\rme(-\langle x\rangle_p)$ for $x\in \QQ_p$. \emph{Note the minus sign}.
  \item We use $f(x)=O(g(x))$ or $f(x)\ll g(x)$ to denote that there exists a constant $C$ such that $|f(x)|\leq C|g(x)|$ for all $x$ in a specified set. If the constant depends on other variables, they will be subscripted under $O$ or $\ll$.
  \item The notation $f(x)\asymp g(x)$ indicates that $f(x)\ll g(x)$ and $g(x)\ll f(x)$. If the constant depends on other variables, they will be subscripted under $\asymp$.
  \item We use $\square$ as the end of the proof and use $\blacksquare$ as the end of the proof of a lemma that is inserted in the middle of another proof.
\end{enumerate}

\section{Elementary computations}\label{sec:preliminary}
\subsection{Elliptic part of the trace formula}\label{subsec:testfunction}
Let $\G=\GL_2$. An element $\gamma\in \G(\QQ)$ is called \emph{elliptic} if its characteristic polynomial is irreducible over $\QQ$. The centralizer of $\gamma$ in $\G$ is denoted by $\G_\gamma$, that is, $\G_\gamma(A)=\{g\in \G(A)\,|\, g\gamma=\gamma g\}$ for any $\QQ$-algebra $A$. 

\begin{proposition}\label{prop:resgroup}
Let $\gamma\in \G(\QQ)$ be an elliptic element. Then there exists a quadratic extension $E/\QQ$ such that $\G_\gamma=\Res_{E/\QQ} \GL_1$, where $\Res$ denotes the Weil restriction of algebraic groups.
\end{proposition}
\begin{proof}
See \cite[Section 5.9]{kottwitz2005}. We just remark that $E$ is the splitting field of the characteristic polynomial of $\gamma$. 
\end{proof}

We denote by $E_\gamma$ the quadratic field determined by $\gamma$ in this proposition and call $E_\gamma$ the \emph{associated quadratic field} of $\gamma$.

Let $\mf{S}$ \index{smf@$\mf{S}$} be the set of places of $\QQ$ and let $\AA$ be the ad\`ele ring of $\QQ$. Recall from \cite[Section 2.2]{altug2015} that the elliptic part of the trace formula for functions of the form $f=\bigotimes_{v\in \mf{S}}' f_v$ on $\G(\AA)$ is
\begin{equation}\label{eq:elliptictrace}
I_\el(f)=\sum_{\gamma\in \G(\QQ)^\#_\el}\vol(\gamma)\prod_{v\in \mathfrak{S}}\orb(f_v;\gamma),
\end{equation}
where $ \G(\QQ)^\#_\el$ denotes the set of elliptic conjugacy classes in $\G(\QQ)$, and
\[
\orb(f_v;\gamma)=\int_{\G_\gamma(\QQ_v)\bs \G(\QQ_v)} f_v(g^{-1}\gamma g)\rmd g,
\qquad
\vol(\gamma)=\int_{Z_+\G_\gamma(\QQ)\bs \G_\gamma(\AA)} \rmd g,
\]
where $Z_+$ \index{zplus@$Z_+$} denotes the connected component of the identity matrix in the center $\Z(\RR)$ of $\G(\RR)$. The measures of $\G_\gamma(\QQ_v)$ and $\G(\QQ_v)$ are normalized as follows. For $\G(\QQ_v)$, 
\begin{enumerate}[itemsep=0pt,parsep=0pt,topsep=0pt,leftmargin=0pt,labelsep=3pt,itemindent=9pt,label=\textbullet]
  \item If $v=p$ is a prime, we normalize the Haar measure on $\G(\QQ_p)$ such that the volume of $\G(\ZZ_p)$ is $1$.
  \item If $v=\infty$, we choose an arbitrary Haar measure (which does not affect the following discussions).
\end{enumerate} 
For $\G_\gamma(\QQ_v)\cong E_v^\times$ (by \autoref{prop:resgroup}),
\begin{enumerate}[itemsep=0pt,parsep=0pt,topsep=0pt,leftmargin=0pt,labelsep=3pt,itemindent=9pt,label=\textbullet]
  \item If $v=p$ is a prime, we know that $E_p$ is either $\QQ_p^2$ or a quadratic extension of $\QQ_p$. In the first case, we let $\G_\gamma(\ZZ_p)=\ZZ_p^\times \times \ZZ_p^\times$. In the second case, we let $\G_\gamma(\ZZ_p)=\cO_{E_p}^\times$.
  In each case, we normalize the Haar measure on $\G_\gamma(\QQ_p)=E_p^\times$ such that the volume of $\G_\gamma(\ZZ_p)$ is $1$.
  \item If $v=\infty$, we know that $E_v$ is either $\RR^2$ or $\CC$. If $E_v=\RR^2$, we choose the measure on $E_v^\times=\RR^\times \times \RR^\times$ to be $\rmd x\rmd y/|xy|$, where $(x,y)$ is the coordinate of $\RR^2$. The action of $Z_+$ on $E_v^\times$ is $a\cdot(x,y)=(ax,ay)$, and we define the measure on $Z_+\bs E_v^\times$ to be the measure of the quotient of $E_v^\times$ by the measure $\rmd a/a$ on $Z_+$. If $E_v=\CC$, we choose the  measure on $E_v^\times=\CC^\times$ to be $2\rmd r\rmd \theta/r$, where we use the polar coordinate $(r,\theta)$ on $\CC^\times$. The action of $Z_+$ on $E_v^\times$ is $a\cdot z=az$, and we define the measure on $Z_+\bs E_v^\times$ to be the measure of the quotient of $E_v^\times$ by the measure $\rmd a/a$ on $Z_+$.
\end{enumerate} 
\begin{remark}
The measure normalization is taken in \cite{finis2011}, which is \emph{different} from that in \cite{langlands2004} used by Altu\u{g}. If we followed Langlands' measure normalization, the action of $Z_+$ on  $\G_\gamma(\RR)$ is multiplying $\sqrt{a}$ rather than by $a$. We remark that this will double the elliptic part considered by Langlands and Altu\u{g}.
\end{remark}

From the definition of $f^n$ in the introductory section, the elliptic part of the trace formula for $f^n$ is
\[
\sum_{\gamma\in \G(\QQ)^\#_\el}\vol(\gamma)n^{-\frac12}\orb(f_\infty;\gamma)\prod_{i=1}^{r}\orb(f_{q_i};\gamma) \prod_{p\notin S}\orb(\triv_{\cX_p^{n_p}};\gamma).
\]

Now we compute the volume of $\gamma$ using these measure normalizations.
\begin{theorem}\label{thm:volumegamma}
Let $\gamma$ be an elliptic element in $\G(\QQ)$ and let $E=E_\gamma$ be the associated quadratic field. Then $\vol(\gamma)=2\rho\sqrt{|D|}$, where $\rho=\rho_E$ denotes the residue of the Dedekind zeta function $\zeta_E(s)$ at $s=1$ and $D=D_E$ is the discriminant of $E$.
\end{theorem}
\begin{proof}
The volume of $\gamma$ is the volume of $Z_+\G_\gamma(\QQ)\bs \G_\gamma(\AA)=Z_+E^\times\bs \AA_E^\times$ by \autoref{prop:resgroup}, where $\AA_E$ denotes the ad\`ele ring of $E$, and the volume equals $2\rho\sqrt{|D|}$. (See \cite[Theorem 4.3.2]{tate1950}. Note that we chose a different normalization of measure, so the factor $2$ and $\sqrt{|D|}$ should be multiplied.)
\end{proof}

\begin{corollary}\label{cor:volumegamma}
Let $\gamma$ be an elliptic element of $\G(\QQ)$ and let $E=E_\gamma$ be the associated quadratic field. Let $D$ be the discriminant of $E$ and $\chi$ be the quadratic character associated with $E$. Then $\vol(\gamma)=2\sqrt{|D|}L(1,\chi)$.
\end{corollary}
\begin{proof}
Since $\zeta_E(s)=\zeta(s)L(s,\chi)$, we get $L(1,\chi)=\rho_E$. The corollary now follows from \autoref{thm:volumegamma}.
\end{proof}

\subsection{Local computation}
\begin{definition}\label{def:defk}
Suppose that $\gamma$ is a regular semisimple element in $\G(\QQ_p)$ and let $E_p$ be the centralizer of $\gamma$ in $\M_2(\QQ_p)$ so that $E_p$ is either $\QQ_p^2$ or a quadratic extension of $\QQ_p$. There exists $\Delta\in E_p$ such that the integral closure of $\ZZ_p$ in $E_p$ is $\ZZ_p\oplus\ZZ_p\Delta$. Let $\cdot\mapsto\overline{\cdot}$ denote the nontrivial element in $\Gal(E_p/\QQ_p)$. Let $\gamma_1,\gamma_2\in E_p$ be the eigenvalues of $\gamma$. Suppose that $\gamma_1-\gamma_2=\kappa(\Delta-\overline{\Delta})$ for some $\kappa\in \QQ_p^\times$, we define $k=v_p(\kappa)\in \ZZ$.
\end{definition}
Clearly such a $\kappa$ always exists, and $k$ is independent of the choice of $\Delta$. 


We provide an alternative characterization of $k$. For any $\gamma\in \G(\QQ_p)$, let $T=\Tr \gamma$ and $N=\det\gamma$. (Note that in \cite{langlands2004} and \cite{altug2015}, they both set $N=4\det \gamma$, but the notation $N=\det\gamma$ is more convenient for our computation in the last two sections.) Write $T^2-4N=\sigma^2 D$, where $v_p(D)\in\{0,1\}$ for $p\neq 2$, and $D\in \ZZ_p$ with $v_p(D)\in\{0,2,3\}$ and $D\equiv 0,1\,(4)$ for $p=2$.

\begin{proposition}\label{prop:kell}
Under the above assumptions, $k_p=v_p(\sigma)$.
\end{proposition}
\begin{proof}
For $p\neq 2$, we have $\cO_{E_p}=\ZZ_p\oplus \ZZ_p\sqrt{D}$, where $\sqrt{D}$ is understood as usual if $E_p$ is a field, and as $(\sqrt{D},-\sqrt{D})$ if $E_p=\QQ_p^2$. Hence we can choose $\Delta=\sqrt{D}$. Since
\[
\gamma_{1,2}=\frac{T\pm\sqrt{T^2-4N}}{2},
\]
we have $\gamma_1-\gamma_2=\pm \sqrt{T^2-4N}=\pm \sigma\sqrt{D}$. Since $\Delta-\overline{\Delta}=2\sqrt{D}$, it follows that $k_p=v_p(\sigma/2)=v_p(\sigma)$.

For $p=2$ and $D\equiv 1\pmod 4$ we have 
\[
\cO_{E_p}=\ZZ_p\oplus \ZZ_p\frac{1+\sqrt{D}}{2},
\]
where $\sqrt{D}$ is considered as in the $p\neq 2$ case. Hence we can choose $\Delta=(1+\sqrt{D})/2$. Since $\Delta-\overline{\Delta}=\sqrt{D}$ and $\gamma_1-\gamma_2=\pm \sigma\sqrt{D}$, we get $k_p=v_p(\sigma)$. For $D\equiv 0\pmod 4$ we have 
\[
\cO_{E_p}=\ZZ_p\oplus \ZZ_p\frac{\sqrt{D}}{2},
\]
and the conclusion follows similarly.
\end{proof}

For any elliptic element $\gamma\in \G(\QQ)$, we denote $T=\Tr \gamma$ and $N=\det\gamma$ and write $T^2-4N=\sigma^2 D$, where $D\equiv 0,1\pmod 4$ is the fundamental discriminant of $E_\gamma$. By \autoref{prop:kell}, $k_\gamma=v_p(\sigma)$. From this we know that for $p\neq 2$, 
\[
k_\gamma=\left\lfloor\frac{v_p(T^2-4N)}{2}\right\rfloor.
\]
For $p=2$, if $v_2(T^2-4N)$ is odd, then $k_\gamma=(v_2(T^2-4N)-3)/2$. If $v_2(T^2-4N)$ is even, writing $T^2-4N=2^{v_2(T^2-4N)}y_0$, then
\[
k_\gamma=\begin{cases}
  v_2(T^2-4N)/2, & \text{if $y_0\equiv 1\pmod 4$}, \\
  (v_2(T^2-4N)-2)/2, & \text{if $y_0\equiv 3\pmod 4$}.
\end{cases}
\] 

The local orbital integrals can be computed for easy functions.
\begin{theorem}\label{thm:maximalcompact}
Let $\gamma$ be an elliptic element in $\G(\QQ)$ and let $E=E_\gamma$ be the associated quadratic field. Let $\chi$ be the character associated with $E$. Let $p$ be a prime number. If $\gamma$ is integral in $E_p$ and $\mathopen{|}\det\gamma\mathclose{|}_p=p^{-m}$, then
\begin{equation}\label{eq:maximallocal}
\orb(\triv_{\cX_p^{m}};\gamma)=1+\sum_{j=1}^{k}p^j\left(1-\frac{\chi(p)}{p}\right),
\end{equation}
and $\orb(\triv_{\cX_p^{m}};\gamma)=0$ otherwise.
\end{theorem}

\begin{theorem}\label{thm:iwahori}
Let $\gamma$ be an elliptic element in $\G(\QQ)$ and let $E=E_\gamma$ be the associated quadratic field. Let $\chi$ be the character associated with $E$. Let $p$ be a prime number. If $\gamma$ is integral in $E_p$ and $\mathopen{|}\det\gamma\mathclose{|}_p=1$, then
\begin{equation}\label{eq:iwahori}
   \orb(\triv_{\cI_p};\gamma) =\frac{2}{p+1}\left(1+\sum_{j=1}^{k}p^j\left(1-\frac{\chi(p)}{p}\right)\right)+\frac{-1+\chi(p)}{1+p},
\end{equation}
and $\orb(\triv_{\cI_p};\gamma)=0$ otherwise.
\end{theorem}
The proofs of the above theorems, though elementary, are somewhat involved. We will prove them in \autoref{sec:prooflocalorbit}.
\begin{remark}
The above two theorems also hold if $\gamma$ is hyperbolic regular if we set $\chi$ to be the trivial character, and the proof is essentially the same.
\end{remark}

\begin{remark}\label{rem:shalika}
We now recall the \emph{Shalika germ}. Let $\G$ be a connected reductive algebraic group over $\QQ$. For any maximal torus $\T(\QQ_p)$ in $\G(\QQ_p)$, the germ of $\orb(f;t)$ at $I$ is finite dimensional \cite{shalika1972} and we have
\[
\orb(f;t)=\sum_{\mf{o}}\orb(f;\mf{o})\Gamma_\mf{o}(t)
\]
near $I$, where $\mf{o}$ runs over all unipotent conjugacy classes. 

For $\G=\GL_2$, we can use the above two theorems to derive an explicit formula for $\Gamma_{\mf{o}}(t)$. For $\G=\GL_m$, the unipotent conjugate classes correspond to partitions of $m$. In the case $m=2$, there are exactly $2$ partitions, namely $(1,1)$ and $(2)$, corresponding to
\[
I=\begin{pmatrix}
  1 & 0 \\
  0 & 1 
\end{pmatrix}\quad\text{and}\quad U=\begin{pmatrix}
  1 & 1 \\
  0 & 1 
\end{pmatrix}.
\]
Using the measure normalization of unipotent orbits as in \cite{kivinen2023}, we have (see (2.1) in  loc. cit.)
\[
\orb(f;I)=f(I)\qquad\text{and}\qquad \orb(f;U)=\int_{\G(\ZZ_p)\times \N(\QQ_p)}f(k^{-1}nk)\rmd n\rmd k.
\]
From this, it follows that
\[
\orb(\triv_{\cK_p};I)=1,\quad \orb(\triv_{\cI_p};I)=1,
\quad
\orb(\triv_{\cK_p};U)=1,\quad\text{and}\quad \orb(\triv_{\cI_p};U)=\frac{2}{p+1}.
\]
Indeed, the first three identities are obvious. For the fourth identity, \autoref{lem:iwahoriindex} implies that $k^{-1}nk\in \cI_p$ if and only if $k\in \cI_p$ and $a\in \ZZ_p$, or $k\in \cK_p-\cI_p$ and $a\in p\ZZ_p$, where $n=(\begin{smallmatrix}1 & a\\ 0 & 1\end{smallmatrix})$.

Hence by \autoref{thm:maximalcompact} and \autoref{thm:iwahori}, we find that
\[
\Gamma_{I}(t)=\frac{1-\chi(p)}{1-p}\qquad\text{and}\qquad\Gamma_{U}(t)=\frac{1}{1-p^{-1}}\left(1-\frac{\chi(p)}{p}\right)p^k
\] 
for $t$ that is sufficiently close to $I$. (Note that if $t$ is sufficiently close to $I$, then $t$ is integral and $\mathopen{|}\det t\mathclose{|}_p=1$.) Also we find that $\Gamma_I(t)=0$ if $\T$ is split in $\QQ_p$ for such $t$.
\end{remark}

We now analyze the ramified case using the above remark.
\begin{theorem}\label{thm:qintegral}
For any maximal torus $\T(\QQ_p)$ in $\G(\QQ_p)$ and any $f\in C_c^\infty(\G(\QQ_p))$, there exists a neighborhood $N$ of $I$ such that
\begin{equation}\label{eq:qintegral}
\orb(f;\gamma)=\frac{1-\chi(p)}{1-p}f(I)+\frac{1}{1-p^{-1}}\left(1-\frac{\chi(p)}{p}\right)p^k\orb(f;U)
\end{equation}
for any $\gamma\in N$, and $k=k_\gamma$ is as in \autoref{def:defk}. $\chi(p)=1,-1$ or $0$ if $\T(\QQ_p)$ is split, inert or ramified, respectively.
\end{theorem}
\begin{proof}
By \autoref{rem:shalika} we know that there exist functions $\Gamma_I(\gamma)$ and $\Gamma_U(\gamma)$ such that
\[
\orb(f;\gamma)=\Gamma_I(\gamma)f(I)+\Gamma_U(\gamma)\orb(f;U),
\]
near $I$, and we have shown that
\[
\Gamma_{I}(\gamma)=\frac{1-\chi(p)}{1-p}\qquad\text{and} \qquad\Gamma_{U}(\gamma)=\frac{1}{1-p^{-1}}\left(1-\frac{\chi(p)}{p}\right)p^k.
\]
Hence \eqref{eq:qintegral} holds.
\end{proof}

\begin{corollary}\label{cor:shalika}
Let $\T(\QQ_p)$ be a maximal torus in $\G(\QQ_p)$ and let $f\in C_c^\infty(\G(\QQ_p))$. Let $z=aI$ be an element in the center of $\G(\QQ_p)$. Then there exists a neighborhood $N$ of $z$ and constants $\lambda_1$ and $\lambda_2$ (independent of the torus) such that for any $\gamma\in N$,
\begin{equation}\label{eq:qintegral2}
\orb(f;\gamma)=\lambda_1\frac{1-\chi(p)}{1-p}+\lambda_2\left(1-\frac{\chi(p)}{p}\right)p^{k_\gamma}.
\end{equation}
\end{corollary}

\begin{proof}
Applying the above theorem with $\widetilde{f}(\gamma)=f(a\gamma)$, we have $\lambda_1=f(z)$ and
\[
\lambda_2=\frac{1}{1-p^{-1}}\orb(\widetilde{f};U).\qedhere
\]
\end{proof}

For any prime $p\in S$, we define \index{thetap@$\theta_{p}(\gamma)$}
\[
\theta_{p}(\gamma)=\frac{1}{\mathopen{|}\det\gamma\mathclose{|}_p^{1/2}}\left(1-\frac{\chi(p)}{p}\right)^{-1}p^{-{k_\gamma}}\orb(f_{p};\gamma).
\]
Since $\det\gamma$ is smooth, by the above corollary, the local behavior of $\theta_{p}$ at $z=aI$ is
\begin{equation}\label{eq:shalikalocal}
\theta_{p}(\gamma)=\lambda_1\left(1-\frac{\chi(p)}{p}\right)^{-1}p^{-{k_\gamma}} \frac{1-\chi(p)}{1-p}+\lambda_2.
\end{equation}

Clearly $\theta_{p}(\gamma)$ is invariant under conjugation. Thus $\theta_{p}(\gamma)$ can be parametrized by $T_\gamma$ and $N_\gamma$, that is, $\theta_{p}(\gamma)=\theta_p(T_\gamma,N_\gamma)$. For any $\nu=(\nu_1,\nu_2,\cdots,\nu_r)\in \ZZ^r$ we write $q^\nu=q_1^{\nu_1}q_2^{\nu_2}\cdots q_r^{\nu_r}$. \index{qnu@$q^\nu$} For any $\nu\in \ZZ^r$ and any sign $\pm$, we define \index{thetap@$\theta_{p}^{\pm, \nu}(y)$}
\begin{equation}\label{eq:defthetaq}
\theta_{p}^{\pm, \nu}(T)=\theta_p(T,\pm nq^\nu)
\end{equation}
for $p\in S$.
Since $\theta_p(\gamma)$ is smooth away from the center, $\theta_p^{\pm,\nu}(T)$ is smooth except at $T^2=\pm 4nq^\nu$. 

We now consider the archimedean orbital integral. For $f_\infty\in C^\infty(\G(\RR))$, the archimedean part of the orbital integral is
\[
\orb(f_\infty;\gamma)=\int_{\G_\gamma(\RR)\bs\G(\RR)}f_\infty(g^{-1}\gamma g)\rmd g.
\]
if the integral makes sense.

\begin{theorem}\label{thm:archimedeanintegral}
For any $f_\infty\in C^\infty(\G(\RR))$ such that the orbital integrals are well-defined, any maximal torus $\T(\RR)$ in $\G(\RR)$ and any $z$ in the center of $\G(\RR)$, there exists a neighborhood $N$ in $\T(\RR)$ of $z$ and smooth functions $g_1,g_2\in C^\infty(N)$ (depending on $f_\infty$ and $z$) such that
\begin{equation}\label{eq:archimedeanintegral}
\orb(f_\infty;\gamma)=g_1(\gamma)+\frac{|\gamma_1\gamma_2|^{1/2}}{|\gamma_1-\gamma_2|}g_2(\gamma)
\end{equation}
for any $\gamma\in \T(\RR)$, where $\gamma_1$ and $\gamma_2$ are the eigenvalues of $\gamma$.  Moreover, $g_1$ and $g_2$ can be extended smoothly to all split and elliptic elements, remaining invariant under conjugation, with $g_1(\gamma)=0$ if $\T(\RR)$ is split, and $g_2$ can further be extended smoothly to the center. If $f_\infty$ is $Z_+$-invariant, then $g_1$ and $g_2$ are also $Z_+$-invariant.
\end{theorem}
\begin{proof}
See \cite{shelstad1979} or \cite[Chapter 6]{langlands1980}.
\end{proof}
Since $g_1$ and $g_2$ are invariant under conjugation, we can parametrize them by $T_\gamma$ and $N_\gamma$ as in the nonarchimedean case.

\subsection{Global computation}
Now we consider the global case. We first compute a single term for the elliptic part of the trace formula. Recall that the contribution of $\gamma$ to $I_\el(f^n)$ is
\begin{equation}\label{eq:ellipticsingle}
\vol(\gamma)n^{-1/2}\orb(f_\infty;\gamma)\prod_{i=1}^{r}\orb(f_{q_i};\gamma)\prod_{ p \notin S}\orb(\triv_{\cX_p^{n_p}};\gamma).
\end{equation}

If \eqref{eq:ellipticsingle} is nonzero, then by \autoref{thm:maximalcompact}, $\gamma\in \cO_{E_p}$ and $\mathopen{|}\det\gamma\mathclose{|}_p=p^{-n_p}$ for all $p\notin S$. Let $T=T_\gamma=\Tr\gamma$ and $N=N_\gamma=\det\gamma$. Then $T_\gamma\in \cO_{E_p}$ for all $p\notin S$ and $N_\gamma=\pm nq^\nu$ with $\nu\in \ZZ^r$. Thus, $T_\gamma\in \ZZ^S$ \index{zuppers@$\ZZ^S$} and $N_\gamma=\pm nq^\nu$ with $\nu\in \ZZ^r$, where $\ZZ^S$ denotes the ring of $S$-integers.

Now we assume that $N_\gamma=\pm nq^\nu$ with $\nu\in \ZZ^r$ and $T_\gamma\in \ZZ^S$. By \autoref{thm:maximalcompact} we have
\[
\eqref{eq:ellipticsingle}=2n^{-\frac12}\vol(\gamma)\orb(f_\infty;\gamma)\prod_{i=1}^{r}\left(q_i^{-\nu_i/2}\theta_{q_i}(\gamma)
\left(1-\frac{\chi(q_i)}{q_i}\right)q_i^{k_{q_i}}\right)\prod_{p\notin S}\left(1+\sum_{j=1}^{k_p}p^j\left(1-
\frac{\chi(p)}{p}\right)\right).
\]

Write $T^2-4N=\sigma^2D$, where $D\equiv 0,1\,(4)$ is the fundamental discriminant and $\sigma\in \ZZ^S$. By \autoref{prop:kell} we know that $k_p=v_p(\sigma)$ for any prime $p$. Define \index{alowerq@$a_{(q)}$}
\[
\sigma_{(q)}=\prod_{i=1}^{r}q_i^{v_{q_i}(\sigma)}=\prod_{i=1}^{r}q_i^{k_{q_i}}.
\]
We then have $\sigma=\sigma_{(q)}\sigma^{(q)}$ \index{aupperq@$a^{(q)}$}, where  $\sigma^{(q)}\in \ZZ_{(S)}$\index{zlowers@$\ZZ_{(S)}$}, which means that $\sigma^{(q)}\in \ZZ$ and is relatively prime to $q_i$ for all $i$.

We first consider the part
\begin{equation}\label{eq:finitepart}
\prod_{p\notin S}\left(1+\sum_{j=1}^{k_p}p^j\left(1-
\frac{\chi(p)}{p}\right)\right).
\end{equation}
By \autoref{thm:maximalcompact} we have
\[
    \eqref{eq:finitepart}
   = \prod_{p\notin S}\left(1+\left(1-\frac{\chi(p)}{p}\right)\sum_{j=1}^{v_p(\sigma)}p^j\right)= \sum_{f\mid \sigma^{(q)}}f\prod_{p\mid f}\left(1-\frac{\chi(p)}{p}\right)
   = \sum_{f\mid \sigma^{(q)}}\frac{\sigma^{(q)}}{f}\prod_{p\mid    \frac{\sigma^{(q)}}{f}}\left(1-\frac{\chi(p)}{p}\right),
\]
where in the last step we made the change of variable $f\mapsto \sigma^{(q)}/f$.

Since $\vol(\gamma)=2\sqrt{|D|}L(1,\chi)$ by \autoref{cor:volumegamma}, \eqref{eq:ellipticsingle} equals
\begin{align*}
    & 2n^{-\frac12}q^{-\frac\nu2}\sqrt{|D|}L(1,\chi)\orb(f_\infty;\gamma)\prod_{i=1}^{r}\left(\theta_{q_i}(\gamma)
\left(1-\frac{\chi(q_i)}{q_i}\right)\right)\sigma_{(q)}\sum_{f\mid \sigma^{(q)}}\frac{\sigma^{(q)}}{f}\prod_{p\mid \frac{\sigma^{(q)}}{f}}\left(1-\frac{\chi(p)}{p}\right) \\
    =& 2n^{-\frac12}q^{-\frac\nu2}\sqrt{|D|}\orb(f_\infty;\gamma)\prod_{i=1}^{r}\left(\theta_{q_i}(\gamma)
\left(1-\frac{\chi(q_i)}{q_i}\right)\right)\sigma_{(q)}\sum_{f\mid \sigma^{(q)}}\frac{\sigma^{(q)}}{f}\prod_{p\nmid \frac{\sigma^{(q)}}{f}}\left(1-\frac{\chi(p)}{p}\right)^{-1}\\
    = & 2\orb(f_\infty;\gamma)\frac{|T_\gamma^2-4N_\gamma|^{1/2}}{n^{1/2}q^{\nu/2}}\prod_{i=1}^{r}\left(\theta_{q_i}(\gamma) \left(1-\frac{1}{q_i}\legendresymbol{D}{q_i}\right)\right)\sum_{f\mid \sigma^{(q)}}\frac{1}{f}L\left(1,\legendresymbol{(\sigma^{(q)})^2D/f^2}{\cdot}\right).
\end{align*}

Next we consider $\orb(f_\infty;\gamma)$. 
Since $T_{z\gamma}=aT_\gamma$ and  $N_{z\gamma}=a^2N_\gamma$ for $z=aI$ with $a>0$, we have $g_i(T_\gamma,N_\gamma)=g_i(aT_\gamma,a^2N_\gamma)$
for $i=1,2$ and any $a>0$. If we let $a=1/(2\sqrt{|N_\gamma|})$, we know that $g_1$ and $g_2$ only depend on $T_\gamma/(2\sqrt{|N_\gamma|})$ and $\sgn N_\gamma$. Thus we can write
\begin{equation}\label{eq:archimedeanfinal}
\orb(f_\infty;\gamma)=g_1^{\sgn N_\gamma}\legendresymbol{T_\gamma}{2\sqrt{|N_\gamma|}}+\left|\frac{T_\gamma^2-4N_\gamma}{N_\gamma}\right|^{-1/2}g_2^{\sgn N_\gamma}\legendresymbol{T_\gamma}{2\sqrt{|N_\gamma|}}.
\end{equation}
Hence
\begin{equation}\label{eq:orbittheta}
\begin{split}
   &\orb(f_\infty;\gamma)\frac{|T_\gamma^2-4N_\gamma|^{1/2}}{n^{1/2}q^{\nu/2}} =\orb(f_\infty;\gamma)\left|\frac{T_\gamma^2-4N_\gamma}{N_\gamma}\right|^{1/2} \\
      =&2\left|\frac{T_\gamma^2}{4N_\gamma}-1\right|^{1/2}g_1^{\sgn N_\gamma}\legendresymbol{T_\gamma}{2\sqrt{|N_\gamma|}}+g_2^{\sgn N_\gamma}\legendresymbol{T_\gamma}{2\sqrt{|N_\gamma|}}= \theta_\infty^{\sgn N_\gamma}\legendresymbol{T_\gamma}{2\sqrt{|N_\gamma|}},
\end{split}
\end{equation}
where \index{thetainfty@$\theta_\infty(x)$}
\begin{equation}\label{eq:deftheta}
\theta_\infty^\pm(x)=2|x^2\mp 1|^{1/2}g_1^\pm(x)+g_2^\pm(x).
\end{equation}
From this and the properties of $g_i(\gamma)$ we know that $\theta_\infty^\pm(x)$ are smooth except for $\pm 1$. By assumption, $\orb(f_\infty;\gamma)$ is compactly supported modulo $Z_+$. Thus, by \eqref{eq:orbittheta}, the functions $\theta_\infty^\pm(x)$ are compactly supported.

Since $\theta_p(\gamma)=\theta_p^{\pm,\nu}(T_\gamma)$ for $N_\gamma=\pm n q^\nu$ and $p\in S$, we know that \eqref{eq:ellipticsingle} equals
\[
 \theta_\infty^{\sgn N_\gamma}\legendresymbol{T_\gamma}{2n^{1/2}q^{\nu/2}}\prod_{i=1}^{r}\left(\theta_{q_i}^{\sgn N_\gamma,\nu}(T_\gamma) \left(1-\frac{1}{q_i}\legendresymbol{D}{q_i}\right)\right)\sum_{f\mid \sigma^{(q)}}\frac{1}{f}L\left(1,\legendresymbol{(\sigma^{(q)})^2D/f^2}{\cdot}\right),
\]
where we naturally extend the Kronecker symbol $\legendresymbol{D}{k}$ to $D\in \ZZ^S$ with $D$ a discriminant and $k\in \ZZ_{(S)}$.

To simplify the notation, we introduce the \emph{Zagier $L$-function}. The \emph{classical Zagier $L$-function}\index{Lsdelta@$L(s,\delta)$} is defined by 
\begin{equation}\label{eq:zagier}
L(s,\delta):= \sum_{\substack{f^2\mid \delta\\\delta/f^2\equiv 0,1\,(4)}}\frac {1}{f^{2s-1}} L\left(s,\legendresymbol{\delta/f^2}{\cdot}\right)
\end{equation}
for $\delta\equiv 0,1\,(4)$ that is not a square. We will generalize this $L$-function to suit our setting. For $\delta\in \ZZ^S$ that is not a square, we define the \emph{partial Zagier $L$-function} \index{Lssdelta@$L^S(s,\delta)$} to be
\[
L^{S}(s,\delta)=\sum_{f^2\mid \delta}\frac{1}{f^{2s-1}}L^{S}\left(s,\legendresymbol{\delta/f^2}{\cdot}\right),
\]
where the sum of $f$ is over all $f\in \ZZ_{(S)}^{>0}$ such that $\delta/f^2\in \ZZ^S$, and
\[
L^{S}\left(s,\legendresymbol{\delta/f^2}{\cdot}\right)=\sum_{k\in \ZZ_{(S)}^{>0}}\frac{1}{k^s}\legendresymbol{\delta/f^2}{k}=\prod_{p\notin S}\left(1-\legendresymbol{\delta/f^2}{p}p^{-s}\right)^{-1}
\]
for $s$ such that the series converges absolutely, extended to $\CC$ by analytic continuation.
Since we have assumed that $2\in S$, we do not need the congruence modulo $4$.

From this we know that \eqref{eq:ellipticsingle} equals
\begin{equation}\label{eq:ellipticsinglefinal}
\theta_\infty^{\sgn N_\gamma}\legendresymbol{T_\gamma}{2\sqrt{|N_\gamma|}}\prod_{i=1}^{r}\theta_{q_i}^{\sgn N_\gamma,\nu}(T_\gamma)L^{S}(1,T_\gamma^2-4N_\gamma).
\end{equation}

Finally, we prove that
\begin{theorem}\label{thm:ellipticpart}
The elliptic part of the trace formula for $f^n$ is given by
\begin{equation}\label{eq:ellipticpart}
2\sum_{\pm}\sum_{\nu\in \ZZ^r}\sum_{\substack{T\in \ZZ^S\\T^2\mp 4nq^{\nu}\neq \square}}\theta_\infty^\pm\legendresymbol{T}{2n^{1/2}q^{\nu/2}}\prod_{i=1}^{r}\theta_{q_i}^{\pm ,\nu}(T)L^{S}(1,T^2\mp 4nq^\nu),
\end{equation}
where ``$\neq \square$" means that ``is not a square". The sums over $\nu$ and $T$ are finite.
\end{theorem}

To prove this theorem, we first prove the following two lemmas.

\begin{lemma}\label{lem:ctscpt}
The functions $\theta_\infty^\pm(x)$ and $\theta_{p}^{\pm,\nu}(y)$ (for any $p\in S$) are continuous and compactly supported.
\end{lemma}

\begin{proof}
Let $N=\pm nq^\nu$. We first prove the continuity of these functions.
By \eqref{eq:deftheta}, $\theta_\infty^\pm(x)$ is continuous. Since $\theta_{p}$ is smooth away from the center, it suffices to consider the behavior near the center, that is, for $y$ such that $y^2-4N$ is near $0$. 

If $N$ has no square roots in $\QQ_{p}$, then $\theta_{p}^{\pm,\nu}(T)=\theta_{p}(T,N)$ is smooth in $T$. Now we consider the case that $N$ has square roots in $\QQ_{p}$. Let $\sqrt{N}$ be any square root of $N$ and let $z=aI$ with $a=\sqrt{N}$. By \autoref{cor:shalika}, the local behavior of regular elements near $z$ is of the form
\[
\theta_{p}(\gamma)=\lambda_1\left(1-\frac{\chi(p)}{p}\right)^{-1}p^{-{k_\gamma}} \frac{1-\chi(p)}{1-p}+\lambda_2.
\]
Since 
\[
p^{-k_\gamma}=|T^2-4N|_{p}'^{1/2}
\]
by \autoref{prop:kell}, which tends to $0$ as $T\to \pm 2\sqrt{N}$ and 
\[
\lambda_1\left(1-\frac{\chi(p)}{p}\right)^{-1}\frac{1-\chi(p)}{1-p}
\]
is bounded, $\theta_{p}^{\pm,\nu}(T)=\theta_{p}(T,N)$ tends to a limit as $T\to \pm 2\sqrt{N}$. Hence  $\theta_{p}^{\pm,\nu}(T)$ is continuous.

We now prove that these functions are compactly supported.
By \eqref{eq:orbittheta} and \autoref{thm:archimedeanintegral}, $\theta_\infty^\pm(x)$ is compactly supported. 
We now prove that $\theta_{p}^{\pm,\nu}$ is compactly supported. Since $f_{p}$ is compactly supported, there exists $M$ such that $f_{p}(\gamma)=0$ for $\gamma$ with $\mathopen{|}\Tr \gamma\mathclose{|}_{p}\geq M$. For $|T|_{p}\geq M$, let $\gamma\in \G(\QQ_{p})$ with $\Tr \gamma=T$ and $\det \gamma=N$. Since $\mathopen{|}\Tr(g^{-1}\gamma g)\mathclose{|}_{p}=\mathopen{|}\Tr \gamma\mathclose{|}_{p}\geq M$, we obtain $f_{p}(g^{-1}\gamma g)=0$. Thus,
\[
\orb(f_{p};\gamma)=\int_{\G_\gamma(\QQ_{p})\bs \G(\QQ_{p})}f_{p}(g^{-1}\gamma g)\rmd g=0.
\]
Therefore $\theta_{p}(\gamma)=0$ and hence $\theta_{p}^{\pm,\nu}(T)=\theta_{p}(T,N)=0$. Hence $\theta_{p}^{\pm,\nu}(T)$ is compactly supported.
\end{proof}

\begin{lemma}\label{lem:nufinite}
For any $p\in S$, there exists $M>0$ such that $\theta_{p}^{\pm,\nu}(y)$ is identically $0$ for $|\nu|\geq M$.
\end{lemma}
\begin{proof}
Let 
\[
\cA^{m}=\{X\in \G(\QQ_{p})\,|\, \mathopen{|}\det X\mathclose{|}_{p} = p^{-m}\}.
\]
Since $f_{p}$ is compactly supported, the support of $f_{p}$ only meets finitely many $\cA^{m}$'s. Hence there exists $M>0$ such that $\supp(f_{p})\cap \cA^{m}=\varnothing$ for all $|m|\geq M$. For any $y\in \QQ_{p}$, let $\gamma\in \G(\QQ_{p})$ with $T_\gamma=y$ and $N_\gamma=\pm nq^\nu$, we have $\mathopen{|}\det(g^{-1}\gamma g)\mathopen{|}_{p}=\mathopen{|}\det\gamma\mathclose{|}_{p}=p^{-\nu}$. Hence for $|\nu|\geq M$ we have $g^{-1}\gamma g\notin \supp(f_{p})$ for all $g\in \G(\QQ_{p})$. Thus $\orb(f_{p};\gamma)=0$ and hence $\theta_{p}(\gamma)=0$. Therefore, $\theta_{p}^{\pm,\nu}(y)=\theta_{p}(\gamma)=0$.
\end{proof}

\begin{proof}[Proof of  \autoref{thm:ellipticpart}]
First we prove that the sums over $\nu$ and $T$ in \eqref{eq:ellipticpart} are finite.
By \autoref{lem:nufinite}, we know that the sum over $\nu$ is finite. Since $\theta_\infty^\pm(x)$ and $\theta_{q_i}^{\pm,\nu}(y_i)$ are compactly supported by \autoref{lem:ctscpt} and the sum over $\nu$ is finite, there exist $M_0,M_1,\dots,M_r>0$ such that the contribution of $T$ in \eqref{eq:ellipticpart} is zero unless $|T|_\infty\leq M_0$ and $|T|_{q_i}\leq M_i$ for all $i=1,\dots,r$. Since $\ZZ^S$ is discrete in $\QQ_S$ (\autoref{prop:cocompact}) we know that such $T$'s form a finite set, that is, the sum over $T$ is finite.

Recall that \eqref{eq:ellipticsingle} is zero unless $T_\gamma\in \ZZ^S$ and $N_\gamma=\pm nq^\nu$ for some $\nu\in \ZZ^r$, and $\gamma$ is elliptic if and only if $T_\gamma^2-4N_\gamma\neq \square$. Since the sums in \eqref{eq:ellipticpart} are finite, $I_\el(f^n)$ is precisely  \eqref{eq:ellipticpart} by \eqref{eq:ellipticsinglefinal}.
\end{proof}

\section{Approximate functional equation}\label{sec:afe}
In this section we derive an \emph{approximate functional equation} for the partial Zagier $L$-function. The goal of the approximate functional equation method is to give an absolutely convergent series for $s$ such that the original series defining the $L$-function does not converge. Such a method has been used in many directions in analytic number theory, such as the subconvexity problem and zero-density estimates \cite{iwaniec2004analytic}. 

In this section, we use an approximate functional equation to give an absolutely convergent series for $L^S(1,\delta)$, which is similar to the derivation of the classical Zagier $L$-function done by Altu\u{g}. The main theorems of this section are \autoref{thm:afe} and \autoref{cor:afe1}.

Recall that the partial Zagier $L$-function is defined as \index{Lssdelta@$L^S(s,\delta)$}
\[
L^{S}(s,\delta)=\sum_{f^2\mid \delta}\frac{1}{f^{2s-1}}L^{S}\left(s,\legendresymbol{\delta/f^2}{\cdot}\right)
\]
for $\delta\in \ZZ^S$ that is not a square. Write $\delta=\sigma^2D$, where $D\equiv 0,1\pmod 4$ is the fundamental discriminant. Further we write $\sigma=\sigma^{(q)}\sigma_{(q)}$. Then we have
\begin{align*}
L^{S}(s,\delta)=&\sum_{f\mid \sigma^{(q)}}\frac{1}{f^{2s-1}}L^{S}\left(s,\legendresymbol{(\sigma^{(q)})^2D/f^2}{\cdot}\right) \\
=&\prod_{i=1}^{r}\left(1-\legendresymbol{(\sigma^{(q)})^2D}{q_i}q_i^{-s}\right)\sum_{f\mid \sigma^{(q)}}\frac{1}{f^{2s-1}}L\left(s,\legendresymbol{(\sigma^{(q)})^2D/f^2}{\cdot}\right).
\end{align*}
If we define $\tau_\delta=(\sigma^{(q)})^2D=\delta/\sigma_{(q)}^2$, then 
\begin{proposition}
The partial Zagier $L$-function can be expressed in terms of the classical Zagier $L$-function $L(s,\delta)$\index{Lsdelta@$L(s,\delta)$} defined in \eqref{eq:zagier} as 
\begin{equation}\label{eq:qavoidedzagier}
L^{S}(s,\delta)=\prod_{i=1}^{r}\left(1-\epsilon_i q_i^{-s}\right)L(s,\tau_\delta),
\end{equation}
where \index{epsiloni@$\epsilon_i$} $\epsilon_i=\legendresymbol{\tau_\delta}{q_i}$ for $i=1,\dots,n$.
\end{proposition}

Now we prove a functional equation for $L^{S}(s,\delta)$, based on the functional equation for $L(s,\delta)$ .

\begin{theorem}[Functional equation of $L(s,\delta)$]\label{thm:funceqndeltaunramify}
We have
\begin{equation}\label{eq:funceqndelta0}
L(s,\delta)= \left(\frac{|\delta|}{\uppi}\right)^{\frac12-s}\frac{\Gamma((\iota+1-s)/2)}{\Gamma((\iota+s)/2)}L(1-s,\delta),
\end{equation}
where \index{iota@$\iota$} $\iota=\iota_\delta=0$ if $\delta>0$, and $\iota=\iota_\delta=1$ if $\delta<0$.
\end{theorem}

\begin{proof}
See the references in the proof of \cite[Proposition 3.1]{altug2015}.
\end{proof}

\begin{theorem}[Functional equation of $L^{S}(s,\delta)$]\label{thm:funceqndelta}
We have
\begin{equation}\label{eq:funceqndelta}
L^{S}(s,\delta)=\left(\frac{|\delta|_{\infty,q}'}{\uppi}\right)^{\frac12-s} \prod_{i=1}^{r}\frac{1-\epsilon_i q_i^{-s}}{1-\epsilon_i q_i^{s-1}} \frac{\Gamma((\iota+1-s)/2)}{\Gamma((\iota+s)/2)}L^{S}(1-s,\delta),
\end{equation}
where $|\delta|_{\infty,q}'=|\tau_\delta|$ \index{1deltainfinityq@$\vert\cdot\vert_{\infty,q}'$}, $\epsilon_i=\epsilon_{i,\delta}=\legendresymbol{\tau_\delta}{q_i}\in \{0,\pm 1\}$, 
$\iota=\iota_\delta=0$ if $\delta>0$, and $\iota=\iota_\delta=1$ if $\delta<0$.
\end{theorem}
\begin{proof}
Using \eqref{eq:qavoidedzagier} and the functional equation \eqref{eq:funceqndelta0} of $L(s,\delta)$.
\end{proof}
\begin{remark}\label{rem:absoluteprime}
We use the notation $|\delta|_{\infty,q}'$ because it is closely related to the product of the $\infty$-norm and the $q$-norm of $\delta$. We have 
\[
|\delta|_{\infty,q}'=|\tau_\delta|=\frac{|\delta|_\infty}{|\sigma_{(q)}^2|_\infty}=|\delta|_\infty \prod_{i=1}^{r}|\sigma|_{q_i}^2.
\]
For any prime $p$, we define the \emph{modified norm} $|\cdot|_p'$ \index{1ynorm@$\vert \cdot\vert_p'$} as follows. Let $y\in \QQ_p$. If $p\neq 2$, we define
\[
|y|_p'=p^{-2\lfloor v_p(y)/2\rfloor}.
\]
Hence $|y|_p'=|y|_p$ if $v_p(y)$ is even and $|y|_p'=p|y|_p$ if $v_p(y)$ is odd. If $p=2$, then for odd $v_p(y)$ we define $|y|_p'=p^{-v_p(y)+3}=p^3|y|_p$, while for even $v_p(y)$, if we write $y=p^{v_p(y)}y_0$, then we define
\[
|y|_p'=\begin{cases}
  p^{-v_p(y)}=|y|_p, & \text{if $y_0\equiv 1\pmod 4$}, \\
  p^{-v_p(y)+2}=p^2|y|_p, & \text{if $y_0\equiv 3\pmod 4$}.
\end{cases}
\]
In each case, we have $|y|_p\asymp_p |y|_p'$. Moreover, if $a\in 1+p^2\ZZ_p$, then $|ay|_p'=|y|_p'$.
We have $p^{k_\gamma}=|T_\gamma^2-4N_\gamma|_p'^{-1/2}$ and  $|\sigma|_{q_i}^2=|\delta|_{q_i}'$ by \autoref{prop:kell}. Hence we obtain
\[
|\delta|_{\infty,q}'=|\delta|_\infty\prod_{i=1}^{r}|\delta|_{q_i}'.
\]
Also, one can easily see that
\[
\epsilon_i=\legendresymbol{\delta|\delta|_{q_i}'}{q_i}.
\]
\end{remark}

Now we prove an \emph{approximate functional equation} for $L^{S}(s,\delta)$. The ingredients are the same in \cite[Section 3]{altug2015}. However, some propositions and proofs of loc. cit. were wrong. So we will give the correct proofs of them first.

To derive the approximate functional equation, we require an auxiliary function. Let $K_s(y)$ denote the \emph{modified Bessel function of the second kind}, which is defined by
\[
K_s(y)=\frac{1}{2}\int_{0}^{+\infty}\rme^{-y(t+1/t)/2}t^s\frac{\rmd t}{t}.
\]
One has (\cite[\SSec 10.27]{DLMF})
\begin{equation}\label{eq:besselmodify}
K_s(y)=\frac{1}{2}\frac{\uppi}{\sin(\uppi s)}(I_{-s}(y)-I_s(y)),
\end{equation}
where
\[
I_s(y)=\left(\frac{y}{2}\right)^s\sum_{k=0}^{+\infty}\frac{(y^2/4)^k}{k!\Gamma(s+k+1)}
\]
is the \emph{modified Bessel function of the first kind}. From this, we conclude that
\begin{equation}\label{eq:seriesk}
K_s(2)=\frac{1}{2}\frac{\uppi}{\sin(\uppi s)}\sum_{k=0}^{+\infty}\frac{1}{k!}\left(\frac{1}{\Gamma(k+1-s)}-\frac{1}{\Gamma(k+1+s)}\right).
\end{equation}

Next we estimate $K_s(2)$. We have the following \emph{Stirling estimates}. See \cite[Chapter II.0]{tenenbaum2015analytic} for a proof of the proposition below.

\begin{proposition}\label{prop:stirling}
We have
\begin{enumerate}[itemsep=0pt,parsep=0pt,topsep=0pt, leftmargin=0pt,labelsep=2.5pt,itemindent=15pt,label=\upshape{(\arabic*)}]
  \item $|\Gamma(s)|=\sqrt{2\uppi}\rme^{(s-1/2)\log s-s}(1+O_\delta(|s|^{-1/2}))$ for $\mathopen{|}\arg s\mathclose{|}<\uppi-\delta$.
\item $|\Gamma(s)|=\sqrt{2\uppi}|t|^{\sigma-1/2}\rme^{-\uppi|t|/2}(1+O_{\sigma_1,\sigma_2}(1/|t|))$ for $s=\sigma+\rmi t$ with $\sigma_1\leq \sigma\leq \sigma_2$ and $|t|\gg 1$.
\end{enumerate}
\end{proposition}

\begin{lemma}
We have
\begin{equation}\label{eq:estimatek}
K_s(2)\ll_{\sigma_1,\sigma_2} |s|^{|\sigma|}\rme^{-\uppi |t|/2}
\end{equation}
for $s=\sigma+\rmi t\in \CC$ and $\sigma_1\leq \sigma\leq \sigma_2$. (Here when $s\to 0$, $|s|^{|\sigma|}$ is defined by taking limit, which is $1$.)
\end{lemma}
\begin{proof}
Clearly $K_s(2)$ is even in $s$ (by \eqref{eq:besselmodify} for example).

For $t=0$ and $\sigma>0$, we have
\[
K_s(2)=\frac12\int_{0}^{+\infty}\rme^{-(t+1/t)}t^\sigma\frac{\rmd t}{t}\ll \int_{0}^{+\infty}\rme^{-t}t^{\sigma-1}\rmd t=\Gamma(\sigma)\ll \sigma^\sigma
\]
by \autoref{prop:stirling} (1). Since $K_s(2)$ is bounded near $0$, the estimate remains valid near $0$. Since $K_s(2)$ is even, the estimate also holds for $\sigma<0$. Thus \eqref{eq:estimatek} holds when $t=0$. Since $[\sigma_1,\sigma_2]$ is compact, this estimate also holds for $|t|\ll 1$.

Now consider the case $|t|\gg 1$. By \eqref{eq:seriesk} we have
\[
K_s(2)=\frac{1}{2}\sum_{k=0}^{+\infty}\frac{1}{k!}\left(\frac{\uppi}{\sin(\uppi s)(k-s)\cdots(1-s)\Gamma(1-s)}-\frac{\uppi}{\sin(\uppi s)(k+s)\cdots(1+s)s\Gamma(s)}\right).
\]
Applying Euler's reflection formula, we obtain
\[
K_s(2)=\frac{1}{2}\sum_{k=0}^{+\infty}\frac{1}{k!}\left(\frac{\Gamma(s)}{(k-s)\cdots(1-s)}-\frac{\Gamma(1-s)}{(k+s)\cdots(1+s)s}\right).
\]
Hence by \autoref{prop:stirling} (2) (here we use the compactness of $[\sigma_1,\sigma_2]$), we have
\[
\begin{split}
   |K_s(2)| & \ll_{\sigma_1,\sigma_2} \sum_{k=0}^{+\infty}\frac{1}{k!}\frac{1}{|t|^k}|\Gamma(s)|+\sum_{k=0}^{+\infty}\frac{1}{k!}\frac{1}{|t|^{k+1}}|\Gamma(1-s)|\ll_{\sigma_1,\sigma_2}|\Gamma(s)|+|\Gamma(1-s)| \\
     & \ll_{\sigma_1,\sigma_2}|t|^{\sigma-1/2}\rme^{-\uppi|t|/2}+|t|^{-1/2-\sigma}\rme^{-\uppi|t|/2}\ll_{\sigma_1,\sigma_2} |t|^{|\sigma|} \rme^{-\uppi|t|/2}.
\end{split}
\]
Hence the conclusion holds by combining the cases $|t|\gg 1$ and $|t|\ll 1$ together.
\end{proof}

We define \index{fx@$F(x)$}
\begin{equation}\label{eq:deff}
F(x)=\frac{1}{2K_0(2)}\int_{x}^{+\infty}\rme^{-t-1/t}\frac{\rmd t}{t}.
\end{equation}
Clearly, $F(0)=1$. Furthermore,
\[
0<\int_{x}^{+\infty}\rme^{-t-1/t}\frac{\rmd t}{t}<\int_{x}^{+\infty}\rme^{-t}\rmd t=\rme^{-x}.
\]
Thus,
\begin{equation}\label{eq:rapidf}
0<F(x)<\frac{\rme^{-x}}{2K_0(2)}.
\end{equation}

Let $\widetilde{F}$ be the Mellin transform of $F$, which is defined by
\begin{equation}\label{eq:defmellinf}
\widetilde{F}(s)=\int_{0}^{+\infty}F(t)t^s\frac{\rmd t}{t}.
\end{equation}
Since $F$ has rapid decay by \eqref{eq:rapidf}, $\widetilde{F}(s)$ converges for $\Re s>0$.

\begin{proposition}\label{prop:estimatemf}
We have
\[
\widetilde{F}(s)=\frac{1}{s}\frac{K_s(2)}{K_0(2)}.
\]
Thus $\widetilde{F}(s)$ admits an analytic continuation to the entire complex plane, except for a simple pole at $s=0$ with residue $1$. Moreover, $\widetilde{F}$ is an odd function, and for $s=\sigma+\rmi t\in \CC$ such that $\sigma_1\leq \sigma\leq \sigma_2$, we have
\[
\widetilde{F}(s)\ll_{\sigma_1,\sigma_2} |s|^{|\sigma|-1}\rme^{-\uppi|t|/2}.
\]
\end{proposition}

\begin{proof}
For $\Re s>0$ we have
\begin{align*}
   \widetilde{F}(s) & =\int_{0}^{+\infty}F(t)t^s\frac{\rmd t}{t} =\frac{1}{s}\int_{0}^{+\infty}F(t)\rmd(t^s)=\frac{1}{s}\left.F(t)t^s\right|_0^{+\infty}-\frac{1}{s}\int_{0}^{+\infty}F'(t)t^s\rmd t\\
     & =\frac{1}{s}\frac{1}{2K_0(2)}\int_{0}^{+\infty}\rme^{-t-1/t}t^s\frac{\rmd t}{t}=\frac{1}{s}\frac{K_s(2)}{K_0(2)}.
\end{align*}
Since $K_s(2)$ is entire for $s$, $\widetilde{F}(s)$ has an analytic continuation with only a simple pole at $s=0$, whose residue is $K_0(2)/K_0(2)=1$. 
Since $K_s(2)$ is an even function in $s$, $\widetilde{F}(s)=K_s(2)/(sK_0(2))$ is an odd function. Moreover, by the estimate \eqref{eq:estimatek} of $K_s(2)$, we get
\[
\widetilde{F}(s)\ll\frac{|K_s(2)|}{|s|}\ll_{\sigma_1,\sigma_2} |s|^{|\sigma|-1}\rme^{-\uppi|t|/2}. \qedhere
\]
\end{proof}

\begin{remark}
Actually, we only need the properties of $F(x)$ and $\widetilde{F}(s)$ but not the explicit expression. One can replace $F$ by any function satisfying the analogous properties.
\end{remark}

We now state and prove an approximate functional equation for $L^{S}(s,\delta)$. 
\begin{theorem}[Approximate functional equation for $L^{S}(s,\delta)$]\label{thm:afe}
For any $A>0$, we have
\begin{equation}\label{eq:approximatefe}
   L^{S}(s,\delta) =\sum_{f^2\mid \delta}\sum_{ k\in \ZZ_{(S)}^{>0}}\frac{1}{k^sf^{2s-1}}\legendresymbol{\delta/f^2}{k}F\legendresymbol{kf^2}{A}
   +|\delta|_{\infty,q}'^{\frac12-s}\sum_{f^2\mid \delta}\sum_{ k\in \ZZ_{(S)}^{>0}}\frac{1}{k^{1-s}f^{1-2s}}\legendresymbol{\delta/f^2}{k}V_{\iota,\epsilon,s}\legendresymbol{kf^2A}{|\delta|_{\infty,q}'},
\end{equation}
where \index{epsilon@$\epsilon$} $\epsilon=(\epsilon_1,\dots,\epsilon_r)\in \{-1,0,1\}^r$, and
\[
V_{\iota,\epsilon,s}(x)=\frac{\uppi^{s-1/2}}{2\uppi \rmi}\int_{(\sigma)}\widetilde{F}(u)\prod_{i=1}^{r}\frac{1-\epsilon_iq_i^{-s+u}}{1-\epsilon_iq_i^{s-u-1}}\frac{\Gamma(\frac{\iota+1-s+u}{2})}{\Gamma(\frac{\iota+s-u}{2})}(\uppi x)^{-u}\rmd u,
\]
$\sigma\in \RR$ is chosen such that $\sigma+\Re s>1$ and $\sigma+\Re (1-s)>1$, and $(\sigma)$ denotes the vertical contour from $\sigma-\rmi\infty$ to $\sigma+\rmi\infty$.
\end{theorem}
\begin{proof}
For any $A>0$, we define
\[
I(A,s)=\frac{1}{2\uppi\rmi}\int_{(\sigma)}\widetilde{F}(u)L^{S}(s+u,\delta)A^u\rmd u,
\]
where $\sigma$ is chosen such that $\sigma+\Re s>1$ and $\sigma+\Re (1-s)>1$. Thus the integral and the sum defining $L$-functions are absolutely convergent. Hence we can interchange the order of the sum and the integral to obtain 
\begin{align*}
   &I(A,s)  =\frac{1}{2\uppi\rmi}\int_{(\sigma)}\widetilde{F}(u)\sum_{f^2\mid \delta}\sum_{k\in \ZZ_{(S)}^{>0}}\frac{1}{k^{s+u}f^{2s+2u-1}}\legendresymbol{\delta/f^2}{k}A^u\rmd u \\
     & =\sum_{f^2\mid \delta}\sum_{k\in \ZZ_{(S)}^{>0}}\frac{1}{k^sf^{2s-1}}\legendresymbol{\delta/f^2}{k}\frac{1}{\dpii}\int_{(\sigma)} \widetilde{F}(u)\left(\frac{A}{kf^2}\right)^{u}\rmd u=\sum_{f^2\mid \delta}\sum_{k\in \ZZ_{(S)}^{>0}}\frac{1}{k^sf^{2s-1}}\legendresymbol{\delta/f^2}{k}F\legendresymbol{kf^2}{A},
\end{align*}
where in the last step we applied the Mellin inversion formula.

Now we move the contour to $\Re u=\sigma'<0$. Since $\delta$ is not a square, $L^{S}(s+u,\delta)$ is an entire function. Furthermore, since $\widetilde{F}(u)$ only has a simple pole at $u=0$ with residue $1$, we get
\[ 
I(A,s)=\frac{1}{\dpii}\int_{(\sigma')}\widetilde{F}(u)L^{S}(s+u,\delta)A^u\rmd u + L^{S}(s,\delta).
\]
Let $\sigma'=-\sigma$. Since $\widetilde{F}(u)=-\widetilde{F}(-u)$, we obtain
\[
\begin{split}
   \frac{1}{\dpii}\int_{(\sigma')}\widetilde{F}(u)L^{S}(s+u,\delta)A^u\rmd u& =-\frac{1}{\dpii}\int_{(\sigma)}\widetilde{F}(-u)L^{S}(s-u,\delta)A^{-u}\rmd (-u)\\
     & =-\frac{1}{\dpii}\int_{(\sigma)}\widetilde{F}(u)L^{S}(s-u,\delta)A^{-u}\rmd u.\\
\end{split}
\]
Using the functional equation \eqref{eq:funceqndelta}, we have 
\begin{align*}
 & \frac{1}{\dpii}\int_{(\sigma')}\widetilde{F}(u)L^{S}(s+u,\delta)A^u\rmd u\\
      =& -\frac{1}{\dpii} \int_{(\sigma)}\widetilde{F}(u)\legendresymbol{|\delta|_{\infty,q}'}{\uppi}^{\frac12+u-s} \prod_{i=1}^{r}\frac{1-\epsilon_iq_i^{-s+u}}{1-\epsilon_iq_i^{s-u-1}} \frac{\Gamma(\frac{\iota+1-s+u}{2})}{\Gamma(\frac{\iota+s-u}{2})}L^{S}(1-s+u,\delta)A^{-u}\rmd u\\
      = & -|\delta|_{\infty,q}'^{\frac12-s}\sum_{f^2\mid \delta}\sum_{k\in \ZZ_{(S)}^{>0}}\frac{1}{k^{1-s}f^{1-2s}}\legendresymbol{\delta/f^2}{k}\frac{\uppi^{s-\frac12}}{\dpii} \int_{(\sigma)}\widetilde{F}(u)\prod_{i=1}^{r}\frac{1-\epsilon_iq_i^{-s+u}}{1-\epsilon_iq_i^{s-u-1}} \frac{\Gamma(\frac{\iota+1-s+u}{2})}{\Gamma(\frac{\iota+s-u}{2})}\left(\frac{A\uppi kf^2}{|\delta|_{\infty,q}'}\right)^{-u}\rmd u\\
     = & -|\delta|_{\infty,q}'^{\frac12-s}\sum_{f^2\mid \delta}\sum_{k\in \ZZ_{(S)}^{>0}}\frac{1}{k^{1-s}f^{1-2s}}\legendresymbol{\delta/f^2}{k}V_{\iota,\epsilon,s}\legendresymbol{kf^2A}{|\delta|_{\infty,q}'}.
\end{align*}
Hence 
\[
\sum_{f^2\mid \delta}\sum_{k\in \ZZ_{(S)}^{>0}}\frac{1}{k^sf^{2s-1}}\legendresymbol{\delta/f^2}{k}F\legendresymbol{kf^2}{A}
     =L^{S}(s,\delta)-|\delta|_{\infty,q}'^{\frac12-s}\sum_{f^2\mid \delta}\sum_{k\in \ZZ_{(S)}^{>0}}\frac{1}{k^{1-s}f^{1-2s}}\legendresymbol{\delta/f^2}{k} V_{\iota,\epsilon,s}\legendresymbol{kf^2A}{|\delta|_{\infty,q}'},
\]
which reduces to \eqref{eq:approximatefe}.
\end{proof}

\begin{corollary}\label{cor:afe1}
For any $A>0$, we have
\begin{equation}\label{eq:approximatefeat1}
L^{S}(1,\delta)=\sum_{f^2\mid \delta}\sum_{k\in \ZZ_{(S)}^{>0}}\frac{1}{kf}\legendresymbol{\delta/f^2}{k} \left[F\legendresymbol{kf^2}{A}+\frac{kf^2}{\sqrt{|\delta|_{\infty,q}'}} V_{\iota,\epsilon}\legendresymbol{kf^2A}{|\delta|_{\infty,q}'}\right],
\end{equation}
where \index{viex@$V_{\iota,\epsilon}(x)$}
\[
V_{\iota,\epsilon}(x)=V_{\iota,\epsilon,1}(x)=\frac{\uppi^{1/2}}{2\uppi \rmi}\int_{(\sigma)}\widetilde{F}(s)\prod_{i=1}^{r}\frac{1-\epsilon_i q_i^{s-1}}{1-\epsilon_i q_i^{-s}}\frac{\Gamma((\iota+s)/2)}{\Gamma((\iota+1-s)/2)}(\uppi x)^{-s}\rmd s
\]
for any $\sigma>1$.
\end{corollary}

\section{Poisson summation}\label{sec:poisson}
Recall that the elliptic part of the trace formula for $f^n$ equals (cf. \eqref{eq:ellipticpart})
\[
2\sum_{\pm}\sum_{\nu\in \ZZ^r}\sum_{\substack{T\in \ZZ^S\\T^2\mp 4nq^{\nu}\neq \square}}\theta_\infty^\pm\legendresymbol{T}{2n^{1/2}q^{\nu/2}}\prod_{i=1}^{r}\theta_{q_i}^{ \pm,\nu}(T)L^{S}(1,T^2\mp 4nq^\nu).
\]
The main theorem in this section is \autoref{thm:ellipticpoisson}. As in \cite{altug2015}, we will use the approximate functional equation method to ensure the validity of the Poisson summation. However, we will use the Poisson summation on $\RR\times \QQ_{q_1}\times \dots\times \QQ_{q_r}$ instead of $\RR$. 

Using the approximate functional equation (\autoref{cor:afe1}) for $L^{S}(1,\delta)$ with $A=|T^2\mp 4nq^\nu|_{\infty,q}'^\vartheta$ for $\vartheta\in \lopen 0,1\ropen$, we have
\begin{align*}
 L^S(1,T^2\mp 4nq^\nu) 
  & =  \sum_{f^2\mid T^2\mp 4nq^\nu}\sum_{k\in \ZZ_{(S)}^{>0}}\frac{1}{kf}\legendresymbol{(T^2\mp 4nq^\nu)/f^2}{k}\\
   &\times\left[F\legendresymbol{kf^2}{|T^2\mp 4nq^\nu|_{\infty,q}'^\vartheta}+\frac{kf^2}{\sqrt{|T^2\mp 4nq^\nu|_{\infty,q}'}}V\legendresymbol{kf^2}{|T^2\mp 4nq^\nu|_{\infty,q}'^{1-\vartheta}}\right],
\end{align*}
where \index{vx@$V(x)$} $V=V_{\iota,\epsilon}$ with $\iota=\iota_{T^2\mp 4nq^\nu}$ and $\epsilon_i=\epsilon_{i,T^2\mp 4nq^\nu}$.

Therefore, we can write \eqref{eq:ellipticpart} as
\begin{equation}\label{eq:ellipticpartapproximate}
\begin{split}
&2\sum_{\pm}\sum_{\nu\in \ZZ^r}\sum_{\substack{T\in \ZZ^S\\ T^2\mp 4nq^\nu\neq \square}}\sum_{f^2\mid T^2\mp 4nq^\nu}  \sum_{k\in \ZZ_{(S)}^{>0}}\frac{1}{kf}\legendresymbol{(T^2\mp 4nq^\nu)/f^2}{k}\theta_\infty^\pm\legendresymbol{T}{2n^{1/2}q^{\nu/2}}\\
    \times &\prod_{i=1}^{r}\theta_{q_i}^{\pm,\nu}(T)\left[F\legendresymbol{kf^2}{|T^2\mp 4nq^\nu|_{\infty,q}'^\vartheta}+\frac{kf^2}{\sqrt{|T^2\mp 4nq^\nu|_{\infty,q}'}}V\legendresymbol{kf^2}{|T^2\mp 4nq^\nu|_{\infty,q}'^{1-\vartheta}}\right].
\end{split}
\end{equation}

\begin{theorem}\label{thm:ellipticpoisson}
We have
\begin{align*}
   &\eqref{eq:ellipticpartapproximate}+\Sigma(\square)  =4\sqrt{n}\sum_{\pm}\sum_{\nu\in \ZZ^r}q^{\nu/2}\sum_{k,f\in \ZZ_{(S)}^{>0}}\frac{1}{k^2f^3}\sum_{\xi\in \ZZ^S}\Kl_{k,f}^S(\xi,\pm nq^\nu)\\
   &\times\int_{x\in\RR}\int_{y\in\QQ_{S_\fin}}\theta_\infty^\pm(x) 
   \theta_{q}^{\pm,\nu}(y)\left[F\legendresymbol{kf^2(4nq^\nu)^{-\vartheta}}{|x^2\mp 1|_\infty^\vartheta|y^2\mp 4nq^\nu|_q'^\vartheta}+\frac{kf^2n^{-1/2}q^{-\nu/2}}{2\sqrt{|x^2\mp 1|_\infty|y^2\mp 4nq^\nu|_q'}}\right.\\
     &\times\left.V\legendresymbol{kf^2(4nq^\nu)^{\vartheta-1}}{|x^2\mp 1|_\infty^{1-\vartheta}|y^2\mp 4nq^\nu|_q'^{1-\vartheta}}\right]\rme\legendresymbol{-2x\xi n^{1/2}q^{\nu/2}}{kf^2}\rme_{q}\legendresymbol{-y\xi}{kf^2}\rmd x\rmd y,
\end{align*}
where $y=(y_1,\dots,y_r)\in \QQ_{S_\fin}=\QQ_{q_1}\times\dots\times\QQ_{q_r}$ with $y_i\in \QQ_{q_i}$ and $\QQ$ embedded diagonally, and for any $y=(y_1,\dots,y_r)\in \QQ_{S_\fin}$, \index{ thetaqy@$\theta_{q}^{\pm,\nu}(y)$} \index{1ynormq@$\vert \cdot\vert_q'$} \index{eqy@$\rme_q(y)$}
\[
\theta_{q}^{\pm,\nu}(y)=\prod_{i=1}^{r}\theta_{q_i}^{\pm,\nu}(y_i),\qquad |y|_q'=\prod_{i=1}^{r}|y_i|_{q_i}',\qquad \rme_q(y)=\prod_{i=1}^{r}\rme_{q_i}(y_i),
\]
and for $m\in \ZZ^S$ and $k,f\in \ZZ_{(S)}$, we define the partial generalized Kloosterman sum to be \index{klkfs@$\Kl_{k,f}^S(\xi,m)$}
\begin{equation}\label{eq:defkloosterman}
\Kl_{k,f}^S(\xi,m)=\sum_{\substack{a \bmod kf^2\\ a^2-4m\equiv 0\,(f^2)}}\legendresymbol{(a^2-4m)/f^2}{k} \rme\legendresymbol{a\xi}{kf^2}\rme_{q}\legendresymbol{a\xi}{kf^2},
\end{equation}
and we define
\begin{equation}\label{eq:defv}
V\legendresymbol{kf^2(4nq^\nu)^{\vartheta-1}}{|x^2\mp 1|_\infty^{1-\vartheta}|y^2\mp 4nq^\nu|_q'^{1-\vartheta}}=V_{\iota,\epsilon}\legendresymbol{kf^2(4nq^\nu)^{\vartheta-1}}{|x^2\mp 1|_\infty^{1-\vartheta}|y^2\mp 4nq^\nu|_q'^{1-\vartheta}}
\end{equation}
with $\iota=\iota_{x^2\mp 1}$ and $\epsilon_i=\epsilon_{i,y_i^2\mp 4nq^\nu}$, and finally
\begin{align*}
\Sigma(\square)=&2\sum_{\pm}\sum_{\nu\in \ZZ^r}\sum_{\substack{T\in \ZZ^S\\ T^2\mp 4nq^\nu= \square}}\sum_{f^2\mid T^2\mp 4nq^\nu} \sum_{k\in \ZZ_{(S)}^{>0}}\frac{1}{kf}\legendresymbol{(T^2\mp 4nq^\nu)/f^2}{k}\theta_\infty^\pm\legendresymbol{T}{2n^{1/2}q^{\nu/2}}\\
    \times &\prod_{i=1}^{r}\theta_{q_i}^{\pm,\nu}(T)\left[F\legendresymbol{kf^2}{|T^2\mp 4nq^\nu|_{\infty,q}'^\vartheta}+\frac{kf^2}{\sqrt{|T^2\mp 4nq^\nu|_{\infty,q}'}}V\legendresymbol{kf^2}{|T^2\mp 4nq^\nu|_{\infty,q}'^{1-\vartheta}}\right].
\end{align*}
\end{theorem}

\begin{remark}
$\Kl_{k,f}^{S}(\xi,m)$ is called the generalized Kloosterman sum since it reduces to the classical Kloosterman sum
\[
S(a,b;c)=\sum_{x\in (\ZZ/c\ZZ)^\times}\rme\legendresymbol{ax+bx^{-1}}{c}
\]
in some cases. See footnote 15 of \cite{altug2015} and \cite[Appendix B]{altug2017} for the relations.
\end{remark}

\begin{proof}[Proof of \autoref{thm:ellipticpoisson}]
$\eqref{eq:ellipticpartapproximate}+\Sigma(\square)$ equals
\begin{align*}
&\sum_{\pm}\sum_{\nu\in \ZZ^r}\sum_{T\in \ZZ^S}\sum_{f^2\mid T^2\mp 4nq^\nu} \sum_{k\in \ZZ_{(S)}^{>0}}\frac{1}{kf}\legendresymbol{(T^2\mp 4nq^\nu)/f^2}{k}\theta_\infty^\pm\legendresymbol{T}{2n^{1/2}q^{\nu/2}}\\
    \times &\prod_{i=1}^{r}\theta_{q_i}^{\pm,\nu}(T)\left[F\legendresymbol{kf^2}{|T^2\mp 4nq^\nu|_{\infty,q}'^\vartheta}+\frac{kf^2}{\sqrt{|T^2\mp 4nq^\nu|_{\infty,q}'}}V\legendresymbol{kf^2}{|T^2\mp 4nq^\nu|_{\infty,q}'^{1-\vartheta}}\right].
\end{align*}

Since $2\in S$ and $k\in \ZZ_{(S)}^{>0}$, the Kronecker symbol $\legendresymbol{\cdot}{k}$ has period $k$. Thus $\legendresymbol{(T^2\mp 4nq^\nu)/f^2}{k}$ has period $kf^2$ in $T$. By interchanging the order of summation over  $f$ and $T$, the above sum equals
\begin{equation}\label{eq:ellipticpartapproximateinsert}
\begin{split}
   2\sum_{\pm}&\sum_{\nu\in \ZZ^r}\sum_{f\in \ZZ_{(S)}^{>0}}\sum_{k\in \ZZ_{(S)}^{>0}} \frac{1}{kf}\sum_{\substack{a \bmod kf^2\\ a^2\mp 4nq^\nu\equiv 0\,(f^2)}}\legendresymbol{(a^2\mp 4nq^\nu)/f^2}{k} \sum_{\substack{T\in \ZZ^S\\T\equiv a \bmod kf^2}}\theta_\infty^\pm\legendresymbol{T}{2n^{1/2}q^{\nu/2}}\\
      &\times \prod_{i=1}^{r}\theta_{q_i}^{\pm,\nu}(T)\left[F\legendresymbol{kf^2}{|T^2\mp 4nq^\nu|_{\infty,q}'^\vartheta}+\frac{kf^2}{\sqrt{|T^2\mp 4nq^\nu|_{\infty,q}'}}V\legendresymbol{kf^2}{|T^2\mp 4nq^\nu|_{\infty,q}'^{1-\vartheta}}\right].
\end{split}
\end{equation}

We can apply the Poisson summation formula to the following function defined on the "semilocal" space $(x,y)\in \RR\times \QQ_{S_\fin}$
\[
\begin{split}
   \Psi(x,y) & =\theta_\infty^\pm\legendresymbol{x}{2n^{1/2}q^{\nu/2}}\prod_{i=1}^{r}\theta_{q_i}^{\pm,\nu}(y_i)\left[F\legendresymbol{kf^2}{|x^2\mp 4nq^\nu|_\infty^\vartheta|y^2\mp 4nq^\nu|_q'^\vartheta}\right. \\
     & \left.+\frac{kf^2}{\sqrt{|x^2\mp 4nq^\nu|_\infty|y^2\mp 4nq^\nu|_q'}}V\legendresymbol{kf^2}{|x^2\mp 4nq^\nu|_\infty^{1-\vartheta}|y^2\mp 4nq^\nu|_q'^{1-\vartheta}}\right],
\end{split}
\]
although it is not a Schwartz function on $\RR\times \QQ_{S_\fin}$. It has singularities with respect to nonarchimedean parts. However, it has rapid decay towards the singularity in the nonarchimedean parts so that the Poisson summation is also valid. We will prove the validity in \autoref{sec:poissonql} (see \autoref{prop:poissonwork}).

By using \autoref{cor:poissonrql2} for $\Psi$, \eqref{eq:ellipticpartapproximateinsert} becomes
\begin{align*}
   &2\sum_{\pm}\sum_{\nu\in \ZZ^r}\sum_{k,f\in \ZZ_{(S)}^{>0}} \frac{1}{kf}\sum_{\substack{a \bmod kf^2\\ a^2\mp 4nq^\nu\equiv 0\,(f^2)}}\legendresymbol{(a^2\mp 4nq^\nu)/f^2}{k}\frac{1}{kf^2} \sum_{\xi\in \ZZ^S}\rme\legendresymbol{a\xi}{kf^2}\rme_{q}\legendresymbol{a\xi}{kf^2}\\
\times&\int_{x\in\RR}\int_{y\in\QQ_{S_\fin}}\theta_\infty^\pm\legendresymbol{x}{2n^{1/2}q^{\nu/2}}\theta_{q}^{\pm,\nu}(y)\left[F\legendresymbol{kf^2}{|x^2\mp 4nq^\nu|_\infty^\vartheta|y^2\mp 4nq^\nu|_q'^\vartheta}+\frac{kf^2}{\sqrt{|x^2\mp 4nq^\nu|_\infty|y^2\mp 4nq^\nu|_q'}}\right.\\
     \times&\left.V\legendresymbol{kf^2}{|x^2\mp 4nq^\nu|_\infty^{1-\vartheta}|y^2\mp 4nq^\nu|_q'^{1-\vartheta}}\right]\rme\legendresymbol{-x\xi }{kf^2}\rme_{q}\legendresymbol{-y\xi}{kf^2}\rmd x\rmd y,
\end{align*}
which equals the right hand side of the formula stated in this theorem by making the change of variable $x\mapsto 2n^{1/2}q^{\nu/2}x$.
\end{proof}

\section{A Kloosterman-type series}\label{sec:kloosterman}
The main theorem of this section is \autoref{cor:dirichletglobal}, which gives an analytic continuation of a Kloosterman-type series that is similar to the result of Altu\u{g}. The idea of the proof is to factor the Kloosterman sum into local products, and then deal with the local series separated, which has been done by Altu\u{g} in \cite[Section 5]{altug2015}.

Recall that we have
\[
\Kl_{k,f}^S(\xi,m)=
  \sum_{\substack{a \bmod kf^2\\ a^2-4m\equiv 0\,(f^2)}}\legendresymbol{(a^2-4m)/f^2}{k}\rme\legendresymbol{a\xi}{kf^2}\rme_{q}\legendresymbol{a\xi}{kf^2},
\]
for $k,f\in \ZZ_{(S)}$ and $\xi,m\in \ZZ^S$. Hence
\[
\Kl_{k,f}^S(0,m)= \sum_{\substack{a \bmod kf^2\\ a^2-4m\equiv 0\,(f^2)}}\legendresymbol{(a^2-4m)/f^2}{k}.
\]

For a prime $p\neq 2$ and $m\in \ZZ_p$, the \emph{local generalized Kloosterman sum} is given by
\[
\Kl_{p^u,p^v}^{(p)}(0,m)=
\sum_{\substack{a \bmod p^{u+2v}\\ a^2-4m\equiv 0\,(p^{2v})}}\legendresymbol{(a^2-4m)/p^{2v}}{p^u}.
\]

\begin{proposition}\label{prop:localkloosterman}
Suppose that $k,f\in \ZZ_{(S)}^{>0}$ and $m\in\ZZ^S$. Then
\[
\Kl_{k,f}^S(0,m)=\prod_{p\notin S}\Kl_{p^{v_p(k)},p^{v_p(f)}}^{(p)}(0,m).
\]
\end{proposition}

\begin{proof}
Suppose that $k=\prod_{p}p^{\alpha_p}$ and $f=\prod_{p}p^{\beta_p}$. By the Chinese remainder theorem, we have an isomorphism
\[
   \ZZ/kf^2 \xrightarrow{\cong} \prod_{p\notin S}\ZZ/p^{\alpha_p+2\beta_p},\qquad
    a \mapsto (a_p)_{p\notin S},
\]
which allows us to express the global sum as a product over local terms:
\[
   \Kl_{k,f}^S(0,m) =\sum_{\substack{a \bmod kf^2\\ a^2-4m\equiv 0\,(f^2)}}\legendresymbol{(a^2-4m)/f^2}{k}=\prod_{\ell\notin S}\sum_{\substack{a_\ell \bmod \ell^{\alpha_\ell+2\beta_\ell}\\ a_\ell^2-4m\equiv 0\,(\ell^{2\beta_\ell})}}\prod_{p\notin S}\legendresymbol{(a^2-4m)/f^2}{\ell^{\alpha_p}}.
\]
Since $\legendresymbol{\cdot}{p}$ has period $p$ and $a_p^2-4m \equiv a^2-4m\pmod{p^{2\beta_p}}$, we obtain
\[
   \Kl_{k,f}^S(0,m) =\prod_{p\notin S}\sum_{\substack{a_p \bmod p^{\alpha_p+2\beta_p}\\ a_p^2-4m\equiv 0\,(p^{2\beta_p})}}\legendresymbol{(a_p^2-4m)/f^2}{p^{\alpha_p}}=\prod_{p\notin S}\Kl_{p^{v_p(k)},p^{v_p(f)}}^{(p)}(0,m).\qedhere
\]
\end{proof}

For any $m\in \ZZ^S$ and $\Re s$ sufficiently large, we define the Dirichlet series $D^S(s,m)$
\[
D^{S}(s,m)=\sum_{k,f\in \ZZ_{(S)}^{>0}}\frac{\Kl_{k,f}^S(0,m)}{k^{s+1}f^{2s+1}}.
\]
Also, for $p\notin S$, we define the local Dirichlet series
\[
D_p(s,m)=\sum_{u,v=0}^{+\infty}\frac{\Kl_{p^u,p^v}^{(p)}(0,m)}{p^{u(s+1)}p^{v(2s+1)}}.
\]
\begin{lemma}\label{lem:dirichletprod}
Suppose that $D(s,m)$ converges absolutely and $D_p(s,m)$ converges absolutely for all $p$. Then
\[
D^{S}(s,m)=\prod_{p\notin S}D_p(s,m).
\]
\end{lemma}

\begin{proof}
By substituting the factorization of $\Kl_{k,f}^S(0,m)$ from \autoref{prop:localkloosterman} into the definition of $D^{S}(s,m)$, we obtain the product formula directly.
\end{proof}

We will derive a formula of $D^{S}(s,m)$ by computing $D_p(s,m)$ for each $p\notin S$  (in particular $p\neq 2$).
\begin{proposition}\label{prop:dirichletlocal1}
Suppose that $p\notin S$, $m\in \ZZ_p^\times$ and $\Re s>0$. Then $D_p(s,m)$ converges absolutely and
  \[
  D_p(s,m)=\frac{1-1/p^{s+1}}{1-1/p^{2s}}.
  \] 
\end{proposition} 

\begin{proof}
See \cite[Lemma 5.2]{altug2015}.
\end{proof}

\begin{proposition}\label{prop:dirichletlocal2}
Suppose that $p\notin S$, $m\in p\ZZ_p$ and $\Re s>0$. Then $D_p(s,m)$ converges absolutely and
  \[
  D_p(s,m)=\frac{1-1/p^{s+1}}{1-1/p^{2s}}\cdot\frac{1-1/p^{s(v_p(m)+1)}}{1-1/p^{s}}.
  \]
\end{proposition} 

\begin{proof}
See \cite[Lemma 5.3]{altug2015}.
\end{proof}

\begin{remark}\label{rem:kloosterman}
The proofs of the above two propositions for $p=2$ are much more complicated than those of $p\neq 2$. However, since $2\in S$, we do not need to consider the case $p=2$. Hence we can prove our result without considering the $p=2$ case.
\end{remark}

\begin{corollary}\label{cor:dirichletglobal}
For $\Re s>1$ we have
\[
D^{S}(s,m)=\frac{\zeta(2s)}{\zeta(s+1)}\prod_{i=1}^{r}\frac{1-q_i^{-2s}}{1-q_i^{-s-1}}\prod_{\substack{p\mid m\\p\notin S}}\frac{1-p^{-s(v_p(m)+1)}}{1-p^{-s}}.
\]
Thus $D^{S}(s,m)$ has a meromorphic continuation to all $s\in \CC$.
\end{corollary}
\begin{proof}
Combining \autoref{lem:dirichletprod}, \autoref{prop:dirichletlocal1} and \autoref{prop:dirichletlocal2}, we obtain the Euler product decomposition
\[
   D^{S}(s,m) =\prod_{p\notin S}D_p(s,m)=\prod_{p\notin S}\frac{1-p^ {-s-1}}{1-p^{-2s}}\prod_{\substack{p\mid m\\p\notin S}}\frac{1-p^{-s(v_p(m)+1)}}{1-p^{-s}} ,
\]
which equals the right hand side of the formula in this corollary for $\Re s>1$. The meromorphic continuation follows from the analytic properties of 
$\zeta(s)$ and the local factors.
\end{proof}

\section{Isolation of specific terms}\label{sec:isolation}
In this section, we compute a formula in \autoref{thm:ellipticpoisson} and isolate specific terms in the elliptic part of the trace formula through residue analysis. The main theorem in this section is \autoref{thm:tracexi0}. The proof is a slight modification of \cite[Section 6]{altug2015}.

Recall that $\eqref{eq:ellipticpartapproximate}+\Sigma(\square)$ equals
\begin{align*}
&4\sqrt{n}\sum_{\pm}\sum_{\nu\in \ZZ^r}q^{\nu/2}\sum_{k,f\in \ZZ_{(S)}^{>0}}\frac{1}{k^2f^3}\sum_{\xi\in \ZZ^S}\Kl_{k,f}^S(\xi,\pm nq^\nu)\\
   \times&\int_{x\in\RR}\int_{y\in\QQ_{S_\fin}}\theta_\infty^\pm(x)\theta_{q}^{\pm,\nu}(y)\left[F\legendresymbol{kf^2(4nq^\nu)^{-\vartheta}}{|x^2\mp 1|_\infty^\vartheta|y^2\mp 4nq^\nu|_q'^\vartheta}+\frac{kf^2n^{-1/2}q^{-\nu/2}}{2\sqrt{|x^2\mp 1|_\infty|y^2\mp 4nq^\nu|_q'}}\right.\\
     \times&\left.V\legendresymbol{kf^2(4nq^\nu)^{\vartheta-1}}{|x^2\mp 1|_\infty^{1-\vartheta}|y^2\mp 4nq^\nu|_q'^{1-\vartheta}}\right]\rme\legendresymbol{-2x\xi n^{1/2}q^{\nu/2}}{kf^2}\rme_{q}\legendresymbol{-y\xi}{kf^2}\rmd x\rmd y.
\end{align*}
Thus the $\xi=0$ term equals
\begin{equation}\label{eq:ellipticxi0}
\begin{split}
&4\sqrt{n}\sum_{\pm}\sum_{\nu\in \ZZ^r}q^{\nu/2}\sum_{k,f\in \ZZ_{(S)}^{>0}}\frac{\Kl_{k,f}^S(0,\pm nq^\nu)}{k^2f^3} \int_{x\in\RR}\int_{y\in\QQ_{S_\fin}}\left[F\legendresymbol{kf^2(4nq^\nu)^{-\vartheta}}{|x^2\mp 1|_\infty^\vartheta|y^2\mp 4nq^\nu|_q'^\vartheta}\right.\\
+&\left.\frac{kf^2n^{-1/2}q^{-\nu/2}}{2\sqrt{|x^2\mp 1|_\infty|y^2\mp 4nq^\nu|_q'}}V\legendresymbol{kf^2(4nq^\nu)^{\vartheta-1}}{|x^2\mp 1|_\infty^{1-\vartheta}|y^2\mp 4nq^\nu|_q'^{1-\vartheta}}\right]\theta_\infty^\pm(x)\theta_{q}^{\pm,\nu}(y)\rmd x\rmd y.
\end{split}
\end{equation}

\begin{theorem}\label{thm:tracexi0}
Let $\vartheta\in \lopen 0,1\ropen$ and $v>0$ such that $\zeta(s+1)$ does not have any zeros for $|s|\leq v$ (since $\zeta(s+1)\neq 0$ at $s=0$, such $v$ exists). Let $(\sigma)$ be the contour from $\sigma-\rmi\infty$ to $\sigma+\rmi\infty$. Let $\cC_v$ be the contour from $-\rmi\infty$ to $-\rmi v$ by a straight line, then goes from $-\rmi v$ to $\rmi v$ by the left semicircle $|s|=v,\Re s\leq 0$, then from $\rmi v$ to $\rmi \infty$ by a straight line. Then 
\begin{align*}
\eqref{eq:ellipticxi0}=&4\sqrt{n}\prod_{p\mid n}\frac{1-p^{-n_p-1}}{1-p^{-1}}\sum_{\pm}\sum_{\nu\in \ZZ^r}q^{\nu/2}\int_{x\in \RR}\int_{y\in \QQ_{S_\fin}}\theta_\infty^\pm(x)\theta_{q}^{\pm,\nu}(y)\rmd x\rmd y\\
-&2^{r+1}\prod_{p\mid n}(n_p+1)\sum_{\pm}\sum_{\nu\in \ZZ^r}\int_{X_0}\int_{Y_\mathbf{1}}\frac{\theta_\infty^\pm(x)\theta_{q}^{\pm,\nu}(y)}{\sqrt{|x^2\mp 1|_\infty|y^2\mp 4nq^\nu|_q'}}\rmd x\rmd y\\
+&4\sqrt{n}\sum_{\pm}\sum_{\nu\in \ZZ^r}q^{\nu/2}\int_{x\in \RR}\int_{y\in \QQ_{S_\fin}}\theta_\infty^\pm(x)\theta_{q}^{\pm,\nu}(y)
  \left[\frac{1}{\dpii}\int_{(-1)}\widetilde{F}(s)\frac{\zeta(2s+2)}{\zeta(s+2)} \prod_{i=1}^{r}\frac{1-q_i^{-2s-2}}{1-q_i^{-s-2}}\right.\\
  \times&\left.\prod_{p\mid n}\frac{1-p^{-(s+1)(n_p+1)}}{1-p^{-s-1}}\left(\frac{(4nq^\nu)^{-\vartheta}}{|x^2\mp 1|_\infty^\vartheta|y^2\mp 4nq^\nu|_q'^\vartheta}\right)^{-s}\rmd s\right]\rmd x\rmd y\\
  +&2\sum_{\pm}\sum_{\nu\in \ZZ^r}\int_{x\in \RR}\int_{y\in \QQ_{S_\fin}}\frac{\theta_\infty^\pm(x)\theta_q^{\pm,\nu}(y)}{\sqrt{|x^2\mp 1|_\infty|y^2\mp 4nq^\nu|_q'}}\left[\frac{\sqrt{\uppi}}{\dpii}\int_{\cC_v}\widetilde{F}(s) \frac{\Gamma(\frac{\iota+s}{2})}{\Gamma(\frac{\iota+1-s}{2})}\frac{\zeta(2s)}{\zeta(s+1)} \right.\\
  \times&\left. \prod_{i=1}^{r}\frac{1-\epsilon_i q_i^{s-1}}{1-\epsilon_i q_i^{-s}}\frac{1-q_i^{-2s}}{1-q_i^{-s-1}}\prod_{p\mid n}\frac{1-p^{-s(n_p+1)}}{1-p^{-s}} \left(\frac{\uppi (4nq^\nu)^{\vartheta-1}}{|x^2\mp 1|_\infty^{1-\vartheta}|y^2\mp 4nq^\nu|_q'^{1-\vartheta}}\right)^{-s}\rmd s\right]
  \rmd x\rmd y,
\end{align*}
where $\iota$, $\epsilon_i$, $|\cdot|_{q}'$ have been defined before, and \index{x0@$X_0$} \index{y1@$Y_{\mathbf 1}$} \index{omegaiyi@$\omega_i(y_i)$}
\[
X_0=\{x\in \RR\,|\, x^2\mp 1>0\},
\quad
Y_{\mathbf{1}}=\left\{y\in \QQ_{S_\fin}\,\middle|\,\omega_i(y_i):=\legendresymbol{(y_i^2\mp 4nq^\nu)|y_i^2\mp 4nq^\nu|_{q_i}'}{q_i}=1\ \text{for all $i$}\right\}.
\]
\end{theorem}
\begin{proof}
For $\sigma=\Re s$ sufficiently large such that the series converges absolutely, by Mellin inversion formula,
\[
F(x)=\frac{1}{\dpii}\int_{(\sigma)}\widetilde{F}(s)x^{-s}\rmd s.
\]
Also, recall that
\[
V_{\iota,\epsilon}(x)=\frac{\sqrt{\uppi}}{2\uppi \rmi}\int_{(1)}\widetilde{F}(s)\prod_{i=1}^{r} \frac{1-\epsilon_i q_i^{s-1}}{1-\epsilon_i q_i^{-s}}\frac{\Gamma((\iota+s)/2)}{\Gamma((\iota+1-s)/2)}(\uppi x)^{-s}\rmd s.
\]
Substituting these into \eqref{eq:ellipticxi0}, we obtain
\begin{align*}
   \eqref{eq:ellipticxi0}=&4\sqrt{n}\sum_{\pm}\sum_{\nu\in \ZZ^r}q^{\nu/2}\sum_{k,f\in \ZZ_{(S)}^{>0}}\frac{\Kl_{k,f}^S(0,\pm nq^\nu)}{k^2f^3}\int_{x\in \RR}\int_{y\in \QQ_{S_\fin}}\theta_\infty^\pm(x)\theta_q^{\pm,\nu}(y)\\
  \times &\left[\frac{1}{\dpii}\int_{(\sigma)}\widetilde{F}(s) \left(\frac{kf^2(4nq^\nu)^{-\vartheta}}{|x^2\mp 1|_\infty^\vartheta|y^2\mp 4nq^\nu|_q'^\vartheta}\right)^{-s}\rmd s+\frac{kf^2n^{-1/2}q^{-\nu/2}}{2\sqrt{|x^2\mp 1|_\infty|y^2\mp 4nq^\nu|_q'}}\right.\\
  \times&\left.\frac{\sqrt{\uppi}}{\dpii}\int_{(\sigma)}\widetilde{F}(s)\prod_{i=1}^{r} \frac{1-\epsilon_i q_i^{s-1}}{1-\epsilon_i q_i^{-s}} \frac{\Gamma(\frac{\iota+s}{2})}{\Gamma(\frac{\iota+1-s}{2})} \left(\frac{\uppi kf^2(4nq^\nu)^{\vartheta-1}}{|x^2\mp 1|_\infty^{1-\vartheta}|y^2\mp 4nq^\nu|_q'^{1-\vartheta}}\right)^{-s}\rmd s\right]\rmd x\rmd y\\
  =& 4\sqrt{n}\sum_{\pm}\sum_{\nu\in \ZZ^r}q^{\nu/2}\int_{x\in \RR}\int_{y\in \QQ_{S_\fin}}\theta_\infty^\pm(x)\theta_q^{\pm,\nu}(y)
  \left[\frac{1}{\dpii}\int_{(\sigma)}\widetilde{F}(s) \sum_{k,f\in \ZZ_{(S)}^{>0}}\frac{\Kl_{k,f}^S(0,\pm nq^\nu)}{k^{2+s}f^{3+2s}}\right.\\
  \times&\left(\frac{(4nq^\nu)^{-\vartheta}}{|x^2\mp 1|_\infty^\vartheta|y^2\mp 4nq^\nu|_q'^\vartheta}\right)^{-s}\rmd s+\frac{n^{-1/2}q^{-\nu/2}}{2\sqrt{|x^2\mp 1|_\infty|y^2\mp 4nq^\nu|_q'}}\frac{\sqrt{\uppi}}{\dpii}\int_{(\sigma)}\widetilde{F}(s)\prod_{i=1}^{r}\frac{1-\epsilon_i q_i^{s-1}}{1-\epsilon_i q_i^{-s}}\\
  \times&\left.\frac{\Gamma(\frac{\iota+s}{2})}{\Gamma(\frac{\iota+1-s}{2})} \sum_{k,f\in \ZZ_{(S)}^{>0}}\frac{\Kl_{k,f}^S(0,\pm nq^\nu)}{k^{1+s}f^{1+2s}}\left(\frac{\uppi (4nq^\nu)^{\vartheta-1}}{|x^2\mp 1|_\infty^{1-\vartheta}|y^2\mp 4nq^\nu|_q'^{1-\vartheta}}\right)^{-s}\rmd s\right]\rmd x\rmd y.
\end{align*}
After applying \autoref{cor:dirichletglobal}, the above is transformed into
\begin{equation}\label{eq:contribute1}
\begin{split}
&4\sqrt{n}\sum_{\pm}\sum_{\nu\in \ZZ^r}q^{\nu/2}\!\int_{x\in \RR}\int_{y\in \QQ_{S_\fin}}\!\theta_\infty^\pm(x)\theta_q^{\pm,\nu}(y)\!
  \left[\!\frac{1}{\dpii}\int_{(\sigma)}\!\frac{\zeta(2s+2)}{\zeta(s+2)} \prod_{i=1}^{r}\frac{1-q_i^{-2s-2}}{1-q_i^{-s-2}}\prod_{p\mid n}\frac{1-p^{-(s+1)(n_p+1)}}{1-p^{-s-1}}\right.\\
  &\times\left(\frac{(4nq^\nu)^{-\vartheta}}{|x^2\mp 1|_\infty^\vartheta|y^2\mp 4nq^\nu|_q'^\vartheta}\right)^{-s}\widetilde{F}(s)\rmd s+\frac{n^{-1/2}q^{-\nu/2}}{2\sqrt{|x^2\mp 1|_\infty|y^2\mp 4nq^\nu|_q'}}\frac{\sqrt{\uppi}}{\dpii}\int_{(\sigma)}\frac{\Gamma(\frac{\iota+s}{2})}{\Gamma(\frac{\iota+1-s}{2})} \frac{\zeta(2s)}{\zeta(s+1)} \\
  &\times\left.\prod_{i=1}^{r}\frac{1-\epsilon_i q_i^{s-1}}{1-\epsilon_i q_i^{-s}} \frac{1-q_i^{-2s}}{1-q_i^{-s-1}}\prod_{p\mid n}\frac{1-p^{-s(n_p+1)}}{1-p^{-s}}\left(\frac{\uppi (4nq^\nu)^{\vartheta-1}}{|x^2\mp 1|_\infty^{1-\vartheta}|y^2\mp 4nq^\nu|_q'^{1-\vartheta}}\right)^{-s}\widetilde{F}(s)\rmd s\right]\rmd x\rmd y.
\end{split}
\end{equation}

We now analyze the following two contour integrals separately:
\[
\frac{1}{\dpii}\int_{(\sigma)}\widetilde{F}(s)\frac{\zeta(2s+2)}{\zeta(s+2)} \prod_{i=1}^{r}\frac{1-q_i^{-2s-2}}{1-q_i^{-s-2}}\prod_{p\mid n}\frac{1-p^{-(s+1)(n_p+1)}}{1-p^{-s-1}} \left(\frac{(4nq^\nu)^{-\vartheta}}{|x^2\mp 1|_\infty^\vartheta|y^2\mp 4nq^\nu|_q'^\vartheta}\right)^{-s}\rmd s
\]
and
\begin{align*}
&\frac{1}{\dpii}\int_{(\sigma)}\widetilde{F}(s)\frac{\Gamma(\frac{\iota+s}{2})}{\Gamma(\frac{\iota+1-s}{2})} \frac{\zeta(2s)}{\zeta(s+1)}\prod_{i=1}^{r}\frac{1-\epsilon_i q_i^{s-1}}{1-\epsilon_i q_i^{-s}} \frac{1-q_i^{-2s}}{1-q_i^{-s-1}}\\
\times&\prod_{p\mid n}\frac{1-p^{-s(n_p+1)}}{1-p^{-s}}\left(\frac{\uppi (4nq^\nu)^{\vartheta-1}}{|x^2\mp 1|_\infty^{1-\vartheta}|y^2\mp 4nq^\nu|_q'^{1-\vartheta}}\right)^{-s} \rmd s.
\end{align*}

We begin with considering the first integral.
\begin{enumerate}[itemsep=0pt,parsep=0pt,topsep=0pt,leftmargin=0pt,labelsep=2.5pt,itemindent=12pt,label=\textbullet]
  \item $\widetilde{F}(s)=K_s(2)/(sK_0(2))$ has a simple pole at $s=0$ with residue $1$ and is holomorphic otherwise.
  \item $\zeta(2s+2)/\zeta(s+2)$ has a simple pole at $s=-1/2$ with residue
  \[
  \lim_{s\to -1/2}\left(s+\frac12\right)\frac{\zeta(2s+2)}{\zeta(s+2)}=\frac{1}{2\zeta(3/2)},
  \]
  and it is holomorphic otherwise on $\Re s\geq -1$. (Since $\zeta(s)$ has no zeros on $\Re s\geq 1$.)
  \item All other terms are holomorphic on $\Re s\geq -1$.
\end{enumerate} 

If we move the contour from $(\sigma)$ to $(-1)$, by residue theorem the first integral equals
\begin{align*} 
&\frac{1}{\dpii}\int_{(-1)}\cdots\rmd s
  + (\res_{s=0}+\res_{s=-\frac12})\cdots
  = \frac{1}{\dpii}\int_{(-1)}\cdots\rmd s
  +\prod_{p\mid n}\frac{1-p^{-(n_p+1)}}{1-p^{-1}}\\
  +&\widetilde{F}\left(-\frac{1}{2}\right)\frac{1}{2\zeta(3/2)} \prod_{i=1}^{r}\frac{1-q_i^{-1}}{1-q_i^{-3/2}}\prod_{p\mid n}\frac{1-p^{-(n_p+1)/2}}{1-p^{-1/2}} \left(\frac{(4nq^\nu)^{-\vartheta}}{|x^2\mp 1|_\infty^\vartheta|y^2\mp 4nq^\nu|_q'^\vartheta}\right)^{\frac12},
\end{align*}
where $\cdots$ denotes the integrand of the first integral.

Next we analyze the second integral.
\begin{enumerate}[itemsep=0pt,parsep=0pt,topsep=0pt,leftmargin=0pt,labelsep=2.5pt,itemindent=12pt,label=\textbullet]
  \item The function
  \[
  \frac{\Gamma(\frac{\iota+s}{2})}{\Gamma(\frac{\iota+1-s}{2})} \frac{\zeta(2s)}{\zeta(s+1)}
  \]
  has a simple pole at $s=1/2$ with residue
  \[
  \lim_{s\to 1/2}\left(s-\frac12\right)\frac{\Gamma(\frac{\iota+s}{2})}{\Gamma(\frac{\iota+1-s}{2})} \frac{\zeta(2s)}{\zeta(s+1)}=\frac{1}{2\zeta(3/2)},
  \]
  and is holomorphic otherwise on $\Re s\geq 0$ other than $s=0$.
  \item The function
  \[
  \widetilde{F}(s)\prod_{i=1}^{r}\frac{1-\epsilon_i q_i^{s-1}}{1-\epsilon_i q_i^{-s}} \frac{1-q_i^{-2s}}{1-q_i^{-s-1}}
  \]
  is holomorphic on $\Re s\geq 0$ other than $s=0$. (Since $1-\epsilon_i q_i^{-s}$ divides $1-q_i^{-2s}$, the poles of $1/(1-\epsilon_i q_i^{-s})$ cancel the zeros of $1-q_i^{-2s}$.)
  \item All other terms are holomorphic on $\Re s\geq 0$.
\end{enumerate} 

Now we consider the behavior at $s=0$. $\iota\neq 0$ if and only if $\Gamma(\frac{\iota+s}{2})/\Gamma(\frac{\iota+1-s}{2})$ is regular at $s=0$, and $\epsilon_i\neq 1$ if and only if the function $\frac{1-\epsilon_i q_i^{s-1}}{1-\epsilon_i q_i^{-s}} \frac{1-q_i^{-2s}}{1-q_i^{-s-1}}$ is zero at $s=0$. Since $\widetilde{F}(s)$ has a simple pole at $s=0$ and $\zeta(2s)/\zeta(s+1)$ has a simple zero at $s=0$, we know that the function
\[
\widetilde{F}(s)\frac{\Gamma(\frac{\iota+s}{2})}{\Gamma(\frac{\iota+1-s}{2})} \frac{\zeta(2s)}{\zeta(s+1)}\prod_{i=1}^{r}\frac{1-\epsilon_i q_i^{s-1}}{1-\epsilon_i q_i^{-s}} \frac{1-q_i^{-2s}}{1-q_i^{-s-1}}
\]
is regular at $s=0$ unless $\iota=0$ and $\epsilon_i=1$ for all $i$. In this exceptional case, the function has a simple pole at $s=0$ with residue
\[
\lim_{s\to 0}s\widetilde{F}(s)\frac{\Gamma(\frac{s}{2})}{\Gamma(\frac{1-s}{2})} \frac{\zeta(2s)}{\zeta(s+1)}\prod_{i=1}^{r}\frac{1-q_i^{s-1}}{1-q_i^{-s}} \frac{1-q_i^{-2s}}{1-q_i^{-s-1}}=\frac{2^{r+1}\zeta(0)}{\Gamma(1/2)}=-\frac{2^r}{\sqrt{\uppi}}.
\]

If we move the contour from $(\sigma)$ to $\cC_v$, by residue theorem the second integral equals
\begin{align*}
& \frac{1}{\dpii}\int_{\cC_v}\cdots\rmd s+  \res_{s=\frac12}\cdots\\
     =&  \frac{1}{\dpii}\int_{\cC_v}\cdots\rmd s
    +
     \widetilde{F}\left(\frac{1}{2}\right)\frac{1}{2\zeta(3/2)}\prod_{i=1}^{r}\frac{1-q_i^{-1}}{1-q_i^{-3/2}} \prod_{p\mid n}\frac{1-p^{-(n_p+1)/2}}{1-p^{-1/2}}\left(\frac{\uppi (4nq^\nu)^{\vartheta-1}}{|x^2\mp 1|_\infty^{1-\vartheta}|y^2\mp 4nq^\nu|_q'^{1-\vartheta}}\right)^{-\frac12}
\end{align*}
for $\iota\neq 0$ or $\epsilon_i \neq 1$ for some $i$, while in the exceptional case an additional term
\[
\res_{s=0}\cdots
=-\frac{2^r}{\sqrt{\uppi}}\prod_{p\mid n}(n_p+1)
\]
should be added, where $\cdots$ denotes the integrand of the second integral.

Combining these results, we can rewrite \eqref{eq:contribute1} as
\begin{align*}
&4\sqrt{n}\sum_{\pm}\sum_{\nu\in \ZZ^r}q^{\nu/2}\int_{x\in \RR}\int_{y\in \QQ_{S_\fin}}\theta_\infty^\pm(x)\theta_q^{\pm,\nu}(y)
  \left[\frac{1}{\dpii}\int_{(-1)}\widetilde{F}(s)\frac{\zeta(2s+2)}{\zeta(s+2)} \prod_{i=1}^{r}\frac{1-q_i^{-2s-2}}{1-q_i^{-s-2}}\right.\\
  \times&\prod_{p\mid n}\frac{1-p^{-(s+1)(n_p+1)}}{1-p^{-s-1}}\left(\frac{(4nq^\nu)^{-\vartheta}}{|x^2\mp 1|_\infty^\vartheta|y^2\mp 4nq^\nu|_q'^\vartheta}\right)^{-s}\rmd s+\prod_{p\mid n}\frac{1-p^{-n_p-1}}{1-p^{-1}}\\
  +& \widetilde{F}\left(-\frac{1}{2}\right)\frac{1}{\zeta(3/2)}\prod_{i=1}^{r}\frac{1-q_i^{-1}}{1-q_i^{-3/2}}\prod_{p\mid n}\frac{1-p^{-(n_p+1)/2}}{1-p^{-1/2}}\left(\frac{(4nq^\nu)^{-\vartheta}}{|x^2\mp 1|_\infty^\vartheta|y^2\mp 4nq^\nu|_q'^\vartheta}\right)^{\frac12}
  \\
  +&\frac{n^{-1/2}q^{-\nu/2}}{\sqrt{|x^2\mp 1|_\infty|y^2\mp 4nq^\nu|_q'}}\frac{\sqrt{\uppi}}{\dpii}\int_{\cC_v}\widetilde{F}(s)  \frac{\Gamma(\frac{\iota+s}{2})}{\Gamma(\frac{\iota+1-s}{2})} \frac{\zeta(2s)}{\zeta(s+1)}\prod_{p\mid n}\frac{1-p^{-s(n_p+1)}}{1-p^{-s}}\\
  \times&\prod_{i=1}^{r}\frac{1-\epsilon_i q_i^{s-1}}{1-\epsilon_i q_i^{-s}} \frac{1-q_i^{-2s}}{1-q_i^{-s-1}}\left(\frac{\uppi (4nq^\nu)^{\vartheta-1}}{|x^2\mp 1|_\infty^{1-\vartheta}|y^2\mp 4nq^\nu|_q'^{1-\vartheta}}\right)^{-s}\rmd s
  +\frac{\sqrt{\uppi}n^{-1/2}q^{-\nu/2}}{2\sqrt{|x^2\mp 1|_\infty|y^2\mp 4nq^\nu|_q'}}\\
  \times&\widetilde{F}\left(\frac{1}{2}\right)\frac{1}{\zeta(3/2)} \prod_{i=1}^{r}\frac{1-q_i^{-1}}{1-q_i^{-3/2}}\prod_{p\mid n}\frac{1-p^{-(n_p+1)/2}}{1-p^{-1/2}}\left.\left(\frac{\uppi (4nq^\nu)^{\vartheta-1}}{|x^2\mp 1|_\infty^{1-\vartheta}|y^2\mp 4nq^\nu|_q'^{1-\vartheta}}\right)^{-\frac12}\right]\rmd x\rmd y\\
  -&2^{r+1}\sqrt{n}\sum_{\pm}\sum_{\nu\in \ZZ^r}q^{\nu/2}\int_{X_0}\int_{Y_{\mathbf{1}}}\theta_\infty^\pm(x)\theta_q^{\pm,\nu}(y)\prod_{p\mid n}(n_p+1)\frac{n^{-1/2}q^{-\nu/2}}{\sqrt{|x^2\mp 1|_\infty|y^2\mp 4nq^\nu|_q'}}\rmd x\rmd y.
\end{align*}
Since $\widetilde{F}$ is odd, the third and the last term in the bracket cancel. Thus \eqref{eq:contribute1} is precisely the right hand side of the formula stated in this theorem.
\end{proof}

We call the first term of the formula in \autoref{thm:tracexi0}, namely
\begin{equation}\label{eq:mainterm}
4\sqrt{n}\prod_{p\mid n}\frac{1-p^{-n_p-1}}{1-p^{-1}}\sum_{\pm}\sum_{\nu\in \ZZ^r}q^{\nu/2}\int_{x\in \RR}\int_{y\in \QQ_{S_\fin}}\theta_\infty^\pm(x)\theta_q^{\pm,\nu}(y)\rmd x\rmd y,
\end{equation}
the \emph{$1$-dimensional term} since this term corresponds to traces of the $1$-dimensional representations, which will be seen below. Also, we call the second term
\begin{equation}\label{eq:eisensteinterm}
2^{r+1}\prod_{p\mid n}(n_p+1)\sum_{\pm}\sum_{\nu\in \ZZ^r}\int_{X_0}\int_{Y_\mathbf{1}}\frac{\theta_\infty^\pm(x)\theta_{q}^{\pm,\nu}(y)}{\sqrt{|x^2\mp 1|_\infty|y^2\mp 4nq^\nu|_q'}}\rmd x\rmd y
\end{equation}
the \emph{Eisenstein term} since it arises from the Eisenstein series. We will compute them explicitly in the next two sections.

\section{Computing the $1$-dimensional term}\label{sec:computemain}
Let $\pi$ be a $1$-dimensional representation of $\G(\QQ)\bs\G(\AA)^1=Z_+\G(\QQ)\bs \G(\AA)$, where $Z_+$ denotes the connected component of the identity of the center of $\G(\RR)$. Then $\pi$ is of the form $\mu\circ\det$ for some character $\mu=\bigotimes_{v\in \mf{S}}'\mu_v: \RR_{>0}\QQ^\times\bs \AA^\times \to \SS^1$. From now on we identify such $\pi$ with $\mu$. For $\varphi\in C_c^\infty(\G(\AA)^1)$, the trace of $\mu(\varphi)$ is 
\[
\Tr(\mu(\varphi))=\int_{Z_+\bs\G(\AA)}\mu(\det g)\varphi(g)\rmd g.
\]

In this section, we will prove that the $1$-dimensional term is precisely the sum of the traces of $\mu(f^n)$. We will compute nonarchimedean analog of the local traces and then combine them together to get the result using the abstract Fourier inversion formula, which does not occur in \cite{altug2015}.

Suppose that $\varphi=\bigotimes_{v\in \mf{S}}'\varphi_v$, where $\varphi_v$ and $\mu_v$ are local components. Then we have
\[
\Tr(\mu(\varphi))=\prod_{v\in \mf{S}}\Tr(\mu_v(\varphi_v)),
\]
where
\[
\Tr(\mu_\ell(\varphi_p))=\int_{\G(\QQ_p)}\mu_p(\det g)\varphi_p(g)\rmd g
\]
for finite $v=\ell$ and 
\[
\Tr(\mu_\infty(\varphi_\infty))=\int_{Z_+\bs \G(\RR)}\mu_\infty(\det g)\varphi_\infty(g)\rmd g.
\]

\subsection{Local computation}
First we consider the archimedean case. Since $\mu_\infty$ is a character on $\RR_{>0}\bs \RR^\times$, we must have $\mu_\infty=\triv$ or $\mu_\infty=\sgn$.

\begin{proposition}\label{prop:archimedeantrace}
\begin{enumerate}[itemsep=0pt,parsep=0pt,topsep=0pt,leftmargin=0pt,labelsep=2.5pt,itemindent=15pt,label=\upshape{(\arabic*)}]
  \item If $\mu_\infty=\triv$, then
  \[
  \int_{Z_+\bs \G(\RR)}\mu_\infty(\det g)f_\infty(g)\rmd g=4\left(\int_{x\in \RR}\theta^+_\infty(x)\rmd x+\int_{x\in \RR}\theta^-_\infty(x)\rmd x\right).
  \]
  \item If $\mu_\infty=\sgn$, then
  \[
  \int_{Z_+\bs \G(\RR)}\mu_\infty(\det g)f_\infty(g)\rmd g=4\left(\int_{x\in \RR}\theta^+_\infty(x)\rmd x-\int_{x\in \RR}\theta^-_\infty(x)\rmd x\right).
  \]
\end{enumerate} 
\end{proposition}

\begin{proof}
Observe that the function $\psi_\pm(x)$ defined in \cite{langlands2004} is
\begin{equation}\label{eq:relate}
\frac{\theta_\infty^\pm(x)}{2\sqrt{|x^2\mp 1|}}
\end{equation}
in our notation. Indeed, in \cite{langlands2004}, Langlands defined $\psi$ to be
\[
\orb(f_\infty;\gamma)=\psi(4N,T)=\psi\left(\sgn 4N,\frac{T}{2\sqrt{|N|}}\right),
\]
where $N=\det \gamma$ and $T=\Tr \gamma$, and he defined $\psi_\pm(r)=\psi(\pm 1,r)$.

In this paper we have (cf. \eqref{eq:orbittheta})
\[
   \orb(f_\infty;\gamma)\left|\frac{T^2-4N}{N}\right|^{1/2} =\theta_\infty^{\sgn N}\legendresymbol{T}{2\sqrt{|N|}}.
\]

Finally, the different action of $Z_+$ acting on $\G(\RR)$ yields an extra factor $2$.
From these definitions and \cite[(27)]{langlands2004} we know that \eqref{eq:relate} holds.
\end{proof}

Next we consider the case $v=p\notin S$.
\begin{proposition}\label{prop:unramifiedtrace}
\begin{enumerate}[itemsep=0pt,parsep=0pt,topsep=0pt,leftmargin=0pt,labelsep=2.5pt,itemindent=15pt,label=\upshape{(\arabic*)}]
  \item If $\mu_p$ is ramified, then
  \[
  \int_{\G(\QQ_p)}\mu_p(\det g)\triv_{\cX_p^{m}}(g)\rmd g=0.
  \]
  \item If $\mu_p$ is unramified, then
  \[
  \int_{\G(\QQ_p)}\mu_p(\det g)\triv_{\cX_p^{m}}(g)\rmd g=\mu_p(p^m)p^{m}\frac{1-p^{-m-1}}{1-p^{-1}}.
  \]
\end{enumerate} 
\end{proposition}

\begin{proof}
For any $a\in \ZZ_p^\times$, consider the element $h=(\begin{smallmatrix} a & 0\\ 0 & 1 \end{smallmatrix})\in \cK_p$. Since $\cK_p$ preserves $\cX_p^{m}$ we have
\[
   \int_{\G(\QQ_p)}\mu_p(\det g)\triv_{\cX_p^{m}}(g)\rmd g =\int_{\G(\QQ_p)}\mu_p(\det (gh))\triv_{\cX_p^{m}}(gh)\rmd g=\mu_p(\det h)\int_{\G(\QQ_p)}\mu_p(\det g)\triv_{\cX_p^{m}}(g)\rmd g.
\]
If $\mu_p$ is ramified, then $\mu_p(a)\neq 1$ for some $a\in \ZZ_p^\times$. Hence from the above equation we know that 
\[
\int_{\G(\QQ_p)}\mu_p(\det g)\triv_{\cX_p^{m}}(g)\rmd g=0.
\]

Now suppose that $\mu_p$ is unramified. In this case, $\mu_p(x)$ is constant on $|x|_p=p^{-m}$. Thus $\mu_p(x)=\mu_p(p^m)$ for all $|x|_p=p^{-m}$. Therefore
\[
  \int_{\G(\QQ_p)}\mu_p(\det g)\triv_{\cX_p^{m}}(g)\rmd g=\mu_p(p^m)\int_{\G(\QQ_p)}\triv_{\cX_p^{m}}(g)\rmd g,
\]
which equals
\[
\mu_p(p^m)p^{m}\frac{1-p^{-m-1}}{1-p^{-1}}
\]
by the displayed formula in \cite{langlands2004} between (64) and (65). (There is a misprint, as explained in footnote 17 of \cite{altug2015}.)
\end{proof}

Finally, let us consider the ramified case $v=p\in S$.

\begin{proposition}\label{prop:ramifiedtrace}
For any prime $p\in S$, we have
\[
\int_{\G(\QQ_{p})}\mu_{p}(\det g)f_{p}(g)\rmd g=\sum_{\nu\in \ZZ}p^{3\nu/2}\left(1-\frac{1}{p}\right)^{-1}\int_{T\in \QQ_{p}}\int_{|N|_{p}= p^{-\nu}}\theta_{p}(T,N)\mu_{p}(N)\rmd T\rmd N.
\]
\end{proposition}

\begin{proof}
In the proof of this proposition we will write $f$ and $\mu$ for $f_{p}$ and $\mu_{p}$, respectively. All the norms below are norms on $\QQ_{p}$.

By Weyl integration formula (cf. \cite[p.241]{clozel1989}), we have
\begin{equation}\label{eq:weylintegrationsum}
\Tr(\mu(f))=\frac{1}{2}\sum_{\T(\QQ_{p})}\int_{\T(\QQ_{p})}|D(\gamma)|\orb(f\cdot (\mu\circ \det);\gamma)\rmd \gamma,
\end{equation}
where $\T(\QQ_{p})$ runs over all conjugacy classes of maximal tori in $\G(\QQ_p)$, and
\[
|D(\gamma)|=\frac{|\gamma_1-\gamma_2|^2}{|\gamma_1\gamma_2|},
\]
where $\gamma_1,\gamma_2\in E_p$ are the eigenvalues of $\gamma$.

We proceed to compute
\begin{equation}\label{eq:weylintegration}
\frac{1}{2}\int_{\T(\QQ_{p})}|D(\gamma)|\orb(f\cdot (\mu\circ \det);\gamma)\rmd \gamma.
\end{equation}
 
\underline{\emph{Case 1:}}\ \ $\T(\QQ_{p})\cong \QQ_{p}^\times\times\QQ_{p}^\times$ is a split torus. We may assume that $\T(\QQ_{p})$ is a diagonal torus.

For any $\gamma\in \T(\QQ_{p})$, we may assume that $\gamma=(\begin{smallmatrix} a & 0 \\ 0 & b \end{smallmatrix})$. Thus
\[
|D(\gamma)|=\frac{|a-b|^2}{|ab|}.
\]
Since the measure on $\T(\QQ_{p})$ is 
\[
\left(1-\frac{1}{p}\right)^{-2}\frac{\rmd a\rmd b}{\mathopen{|}\det\gamma\mathclose{|}}=\left(1-\frac{1}{p}\right)^{-2}\frac{\rmd a\rmd b}{|ab|},
\]
where $\rmd a$ and $\rmd b$ denote the standard measure on $\QQ_{p}$ (the factor $(1-1/p)^{-2}$ comes from the quotient of the volume of $\ZZ_{p}\times \ZZ_{p}$ by that of $\ZZ_{p}^\times \times \ZZ_{p}^\times$),
we know that \eqref{eq:weylintegration} equals
\[
\frac{1}{2}\left(1-\frac{1}{p}\right)^{-2}\int_{\QQ_{p}^2}\frac{|a-b|^2}{|ab|}\orb(f\cdot (\mu\circ \det);\gamma)\frac{\rmd a\rmd b}{|ab|}.
\]
Since the complement of $\QQ_{p}^\times\times \QQ_{p}^\times$ has zero measure, we can omit it.

Now we make the change of variable $T=a+b$ and $N=ab$. Since making the change of variable in the $p$-adic case satisfies a similar formula as in the real case \cite{weil1982adele}, we can compute the Jacobian as in the real case.
We have $a,b=(T\pm \sqrt{T^2-4N})/2$. By symmetry we may assume that $a=(T+\sqrt{T^2-4N})/2$ and $b=(T-\sqrt{T^2-4N})/2$. Thus
\[
\rmd a=\frac{1}{2}\left(\rmd T+\frac{T}{\sqrt{T^2-4N}}\rmd T-2\frac{\rmd N}{\sqrt{T^2-4N}}\right),
\quad
\rmd b=\frac{1}{2}\left(\rmd T-\frac{T}{\sqrt{T^2-4N}}\rmd T+2\frac{\rmd N}{\sqrt{T^2-4N}}\right).
\]
Hence
\[
\rmd a\wedge \rmd b=\frac{\rmd T\wedge \rmd N}{\sqrt{T^2-4N}}.
\]
Moreover we have $ab=N$ and $(a-b)^2=T^2-4N$. For $T,N\in \QQ_p^2$ with $T^2-4N\neq 0$, we denote $\gamma_{T,N}$ to be the regular element in $\G(\QQ_p)$ with trace $T$ and determinant $N$, up to conjugacy. 
Since the map $(a,b)\mapsto (T,N)$ is $2:1$ onto its image, \eqref{eq:weylintegration} equals
\begin{align*}
    & \left(1-\frac{1}{p}\right)^{-2}\int_{\gamma_{T,N}\in \T(\QQ_{p})}\frac{|T^2-4N|}{|N|}\orb(f\cdot (\mu\circ \det);T,N)\frac{\rmd T\rmd N}{|N|\sqrt{|T^2-4N|}} \\
    = & \left(1-\frac{1}{p}\right)^{-2}\int_{\gamma_{T,N}\in \T(\QQ_{p})}\frac{\sqrt{|T^2-4N|}}{|N|^2}\orb(f;T,N)\mu(N)\rmd T\rmd N.
\end{align*}
Since $\gamma=\gamma_{T,N}$ is split, $v_p(T^2-4N)$ is even and thus $\sqrt{|T^2-4N|}=p^{-k}$ for $p\neq 2$. For $p=2$ we have the same conclusion since $(T^2-4N)|T^2-4N|\equiv 1\pmod 4$ implies that $|T^2-4N|'=|T^2-4N|$. Since $\T(\QQ_{p})$ is split, $\chi(p)=1$ in this case. This yields
\[
\theta_p(T,N)=\frac{1}{|N|^{1/2}}\left(1-\frac{1}{p}\right)^{-1}p^{-k}\orb(f;T,N).
\]
Hence \eqref{eq:weylintegration} equals
\[
\left(1-\frac{1}{p}\right)^{-1}\int_{\gamma_{T,N}\in \T(\QQ_{p})}\frac{1}{|N|^{3/2}}\theta_{p}(T,N)\mu(N)\rmd T\rmd N.
\]

\underline{\emph{Case 2:}}\ \ $\T(\QQ_{p})\cong E_{p}^\times$ is an elliptic torus, where $E_{p}/\QQ_{p}$ is a quadratic extension.

For any $\gamma\in \T(\QQ_{p})$, recall that
\[
|D(\gamma)|=\frac{|\gamma_1-\gamma_2|^2}{|\gamma_1\gamma_2|},
\]
where $\gamma_1,\gamma_2\in E_p$ are the eigenvalues of $\gamma$. 
Now we identify $\gamma$ with $\gamma_1$ so that $\gamma_2=\overline{\gamma}$ is the conjugate of $\gamma$. Then \eqref{eq:weylintegration} equals
\[
\frac{1}{2}\int_{E_{p}^\times}\frac{|(\gamma-\overline{\gamma})^2|}{|\gamma\overline{\gamma}|}\orb(f\cdot (\mu\circ \det);\gamma)\rmd\gamma,
\]
where $\rmd\gamma$ denotes the measure on $E_{p}^\times$.

Suppose that $\cO_{E_{p}}=\ZZ_{p}\oplus \ZZ_{p}\Delta$. Then we can identify $E_{p}$ with $\QQ_{p}^2$ via the basis $\{1,\Delta\}$. Thus we can embed $E_p$ into $\G(\QQ_{p})$ by identifying $\xi\in E_p$ with the matrix of the linear transformation $\eta\mapsto \xi\eta$ with respect to $\{1,\Delta\}$. If we write $\Delta^2=\alpha+\beta\Delta$, then 
\[
1=\begin{pmatrix}
  1 & 0 \\
  0 & 1 
\end{pmatrix}\qquad \text{and}\qquad \Delta=\begin{pmatrix}
  0 & \alpha \\
  1 & \beta 
\end{pmatrix}.
\]

Under the identification $\gamma \leftrightarrow a+b\Delta$, we have
\[
\rmd \gamma=\frac{\vol(\cO_{E_{p}})}{\vol(\cO_{E_{p}}^\times)}\frac{\rmd a\rmd b}{\mathopen{|}\det(a+b\Delta)\mathclose{|}},\quad|(\gamma-\overline{\gamma})^2|=|b|^2|(\Delta-\overline{\Delta})^2|,\quad \text{and}\quad |\gamma\overline{\gamma}|=\mathopen{|}\det(a+b\Delta)\mathclose{|}.
\]
Thus \eqref{eq:weylintegration} equals
\[
\frac{1}{2}\int_{\QQ_{p}^2}\frac{|b|^2|(\Delta-\overline{\Delta})^2|} {\mathopen{|}\det(a+b\Delta)\mathclose{|}}\orb(f\cdot(\mu\circ\det);\gamma)\frac{\vol(\cO_{E_{p}})} {\vol(\cO_{E_{p}}^\times)}\frac{\rmd a\rmd b}{\mathopen{|}\det(a+b\Delta)\mathclose{|}}.
\]
Now we make the change of variable $T=\Tr(a+b\Delta)=2a+b\beta$ and $N=\det(a+b\Delta)=a^2+\beta ab-\alpha b^2$. Then we have
\[
b=\pm \sqrt{\frac{T^2-4N}{\beta^2+4\alpha}}\qquad\text{and}\qquad a=\frac{T-\beta b}{2}.
\]
Without loss of generality we take the plus sign in $b$. Then we have
\[
\rmd b=\frac{T\rmd T}{\beta^2+4\alpha}\sqrt{\frac{\beta^2+4\alpha}{T^2-4N}}-\frac{2\rmd N}{\beta^2+4\alpha}\sqrt{\frac{\beta^2+4\alpha}{T^2-4N}}\qquad \text{and}\qquad \rmd a=\frac{1}{2}\rmd T-\frac{\beta}{2}\rmd b.
\]
Hence
\[
  \rmd a\wedge \rmd b  =\frac{1}{2}\rmd T\wedge \rmd b=-\frac{1}{\beta^2+4\alpha}\sqrt{\frac{\beta^2+4\alpha}{T^2-4N}}\rmd T\wedge \rmd N=-\frac{1}{b(\beta^2+4\alpha)}\rmd T\wedge \rmd N.
\]
Therefore by making the change of variable $(a,b)\mapsto (T,N)$ and noting that this map is $2:1$ onto its image, \eqref{eq:weylintegration} becomes
\[
\frac{\vol(\cO_{E_{p}})}{\vol(\cO_{E_{p}}^\times)}\int_{\gamma_{T,N}\in \T(\QQ_{p})}\frac{|b|^2|(\Delta-\overline{\Delta})^2|}{|N|}\orb(f;T,N)\mu(N)\frac{1} {|bN(\beta^2+4\alpha)|}\rmd T\rmd N.
\]
Since $\Delta^2=\alpha+\beta\Delta$, we have 
$|(\Delta-\overline{\Delta})^2|=|\beta^2+4\alpha|$. Thus \eqref{eq:weylintegration} equals
\[
\frac{\vol(\cO_{E_{p}})}{\vol(\cO_{E_{p}}^\times)}\int_{\gamma_{T,N}\in \T(\QQ_{p})}\frac{|b|}{|N|^2}\orb(f;T,N)\mu(N)\rmd T\rmd N.
\]

If $E_{p}$ is inert, then $\vol(\cO_{E_{p}})/\vol(\cO_{E_{p}}^\times)=(1-1/p^2)^{-1}$. Since $\chi(p)=-1$ in this case, we have
\[
\theta_p(T,N)=\frac{1}{|N|^{1/2}}\left(1+\frac{1}{p}\right)^{-1}p^{-k}\orb(f;T,N).
\]
Moreover, $|b|=p^{-k_\gamma}$ by \autoref{def:defk}. Hence \eqref{eq:weylintegration} equals
\[
\left(1-\frac{1}{p}\right)^{-1}\int_{\gamma_{T,N}\in \T(\QQ_{p})}\frac{1}{|N|^{3/2}}\theta_{p}(T,N)\mu(N)\rmd T\rmd N.
\]

If $E_{p}$ is ramified, then  $\vol(\cO_{E_{p}})/\vol(\cO_{E_{p}}^\times)=(1-1/p)^{-1}$. Since $\chi(p)=0$ in this case, we have
\[
\theta_{p}(T,N)=\frac{1}{|N|^{1/2}}p^{-k}\orb(f;T,N).
\]
Hence \eqref{eq:weylintegration} equals
\[
\left(1-\frac{1}{p}\right)^{-1}\int_{\gamma_{T,N}\in \T(\QQ_{p})}\frac{1}{|N|^{3/2}}\theta_{p}(T,N)\mu(N)\rmd T\rmd N.
\]

Finally we combine the cases together. Since the regular elements in $\G(\QQ_{p})$ form a set of full measure, we have
\[
\begin{split}
   \Tr(\mu(f))
    = & \left(1-\frac{1}{p}\right)^{-1}\int_{N\in \QQ_{p}}\frac{1}{|N|^{3/2}}\int_{T\in \QQ_{p}}\theta_{p}(T,N)\mu(N)\rmd T\rmd N\\
    = & \left(1-\frac{1}{p}\right)^{-1}\sum_{\nu\in \ZZ}\int_{|N|=p^{-\nu}}p^{3\nu/2}\int_{T\in \QQ_{p}}\theta_{p}(T,N)\mu(N)\rmd T\rmd N.\qedhere
\end{split}
\]
\end{proof}

\subsection{Global computation}
Finally we consider the summation
\begin{equation}\label{eq:sum1dimensional}
\sum_{\mu}\Tr(\mu(f^n)),
\end{equation}
where $\mu$ runs over all characters on $\RR_{>0}\QQ^\times\bs \AA^\times$. If $\mu$ is ramified at some prime $p\notin S$, then $\mu_p$ is a ramified character on $\QQ_p^\times$. Thus $\Tr(\mu_p(f^n_p))=0$ and hence $\Tr(\mu(f^n))=0$. Hence we only need to consider the primitive Dirichlet characters that only ramify at the places in $S$, that is, we only need to consider characters on $\RR_{>0}\QQ^\times\prod_{p\notin S}\ZZ_p^\times \bs \AA^\times$.

\begin{lemma}\label{lem:characterfinite}
For any prime $p\in S$, there exists $M>0$ such that $\Tr(\mu_{p}(f^n_{p}))=\Tr(\mu_{p}(f_p))=0$ if $\mu_{p}$ is not trivial on $1+p^M\ZZ_{p}$.
\end{lemma}
\begin{proof}
For any $N\in \ZZ_{>0}$, we define 
\[
\cK_{p}(N)=\left\{\begin{pmatrix}
                 a & b \\
                 c & d 
               \end{pmatrix}\in \cK_{p}\ \middle|\ \begin{pmatrix}
                 a & b \\
                 c & d 
               \end{pmatrix}-\begin{pmatrix}
                 1 & 0\\
                 0 & 1 
               \end{pmatrix}\in p^N\M_2(\ZZ_{p})\right\}.
\]
Then $\cK_{p}(N)$ form a neighborhood basis of $1\in\G(\QQ_{p})$. Since $f_p\in C_c^\infty(\G(\QQ_{p}))$, there exists $M$ such that $f_p$ is biinvariant under the action of $K=\cK_{p}(M)$. Therefore for any $k\in K$ we have
\[
   \Tr(\mu_{p}(f_{p})) =\int_{\G(\QQ_p)}\mu_p(\det g)f_p(g)\rmd g=\int_{\G(\QQ_{p})}\mu_{p}(\det (gk))f_{p}(gk)\rmd g =\mu_{p}(\det(k))\Tr(\mu_{p}(f_{p})).
\]
From this we know that $\Tr(\mu_{p}(f_{p}))=0$ unless $\mu_{p}(\det(k))=1$ for all $k\in K$. For any $a\in 1+p^{M}\ZZ_{p}$ we have $(\begin{smallmatrix}  a & 0 \\ 0 & 1 \end{smallmatrix} )\in K$. Hence $\mu_{p}(a)=1$. Thus $\mu_{p}$ is trivial on $1+p^{M}\ZZ_{p}$. 
\end{proof}

\begin{corollary}\label{cor:characterfinite}
\eqref{eq:sum1dimensional} is a finite sum.
\end{corollary}
\begin{proof}
By \autoref{lem:characterfinite}, there exist $M_1,\dots,M_r\in \ZZ_{>0}$ such that $\Tr(\mu_{q_i}(f^n_{q_i}))=0$ if $\mu_{q_i}$ is not trivial on $1+q_i^{M_i}\ZZ_{q_i}$.
Suppose that $\mu$ is a character on $\RR_{>0}\QQ^\times\bs \AA^\times$ such that $\Tr(\mu(f^n))\neq 0$. Then $\Tr(\mu_p(f^n_p))\neq 0$ for each $p$. This means that $\mu_p$ is unramified for each $p\notin S$ and $\mu_{q_i}$ is trivial on $1+q_i^{M_i}\ZZ_{q_i}$ for each $i$. Hence $\mu$ can be considered as a character on
\[
\left(\RR_{>0}\QQ^\times\prod_{p\notin S}\ZZ_p
^\times \prod_{i=1}^{r}(1+q_i^{M_i}\ZZ_{q_i})\right)\bs \AA^\times,
\]
which is a finite group by class field theory. The summation is just the summation of the characters on the above group and is thus finite.
\end{proof}

It is well-known that we have a natural isomorphism
\[
\prod_{p}\ZZ_p^\times\xrightarrow{\cong}\QQ^\times\RR_{>0}\bs\AA^\times,\qquad
(x_p)_{p}\mapsto (1,(x_p)_p).
\]
The above isomorphism induces
\[
\prod_{i=1}^{r}\ZZ_{q_i}^\times\xrightarrow{\cong}\QQ^\times\RR_{>0}\prod_{p\notin S}\ZZ_p^\times\bs\AA^\times,\qquad
(x_{q_i})_{i}\mapsto (1,(1)_{p\notin S},(x_{q_i})_{q_i}).
\]
Hence we get
\begin{proposition}\label{prop:characterzq}
We have a $1-1$ correspondence
\[
\left(\QQ^\times\RR_{>0}\prod_{p\notin S}\ZZ_p^\times\bs\AA^\times\right)\sphat \to (\ZZ_{q_1}^\times\times \cdots\times \ZZ_{q_r}^\times)\sphat\ =\widehat{\ZZ_{S_\fin}^\times},\qquad
\mu\mapsto (\mu_{q_i}|_{\ZZ_{q_i}^\times}).
\]
\end{proposition}

\begin{theorem}\label{thm:maintermcompute}
The $1$-dimensional term \eqref{eq:mainterm} equals
\[
\sum_{\mu}\Tr(\mu(f^n))=\sum_{\mu\in (\QQ^\times\RR_{>0}\prod_{p\notin S}\ZZ_p^\times\bs\AA^\times)\sphat}\Tr(\mu(f^n)).
\]
\end{theorem}

\begin{proof}
By \autoref{prop:archimedeantrace}, \autoref{prop:unramifiedtrace} and \autoref{prop:ramifiedtrace} and recalling the definition of $f^n$, we have
\begin{align*}
    \sum_{\mu}\Tr(\mu(f^n)) 
    =& \sum_{\mu}4\left(\int_{x\in \RR}\theta^+_\infty(x)\rmd x+\mu_\infty(-1)\int_{x\in \RR}\theta^-_\infty(x)\rmd x\right)\sqrt{n}\prod_{p\mid n}\mu_p(n)\frac{1-p^{-n_p-1}}{1-p^{-1}}\\
    \times &\sum_{\nu\in \ZZ^r}q^{3\nu/2}\prod_{i=1}^{r}\left(1-\frac{1}{q_i}\right)^{-1}\int_{T\in \QQ_{q_i}}\int_{|N_i|_{q_i}= q_i^{-\nu_i}}\theta_{q_i}(T_i,N_i)\mu_{q_i}(N_i)\rmd T_i\rmd N_i.
\end{align*}
For a character $\mu=\bigotimes_v'\mu_v$ on $\QQ^\times\RR_{>0}\prod_{p\notin S}\ZZ_p^\times\bs\AA^\times$ we have
\[
\prod_{i=1}^{r}\mu_{q_i}(q^{-\nu})=\mu_\infty(q^\nu)\prod_{\ell\notin S}\mu_\ell(q^\nu)=1,
\quad
\mu_\infty(-1)=\prod_{p\notin S}\mu_p(-1)\prod_{i=1}^{r}\mu_{q_i}(-1)=\prod_{i=1}^{r}\mu_{q_i}(-1), 
\]
and 
\[
\prod_{p\mid n}\mu_p(n)=\mu_\infty(n^{-1})\prod_{i=1}^{r}\mu_{q_i}(n^{-1})=\prod_{i=1}^{r}\mu_{q_i}(n^{-1}).
\]
Hence
\begin{align*}
     \sum_{\mu}\Tr(\mu(f^n))
   =  &4\sqrt{n}\prod_{p\mid n}\frac{1-p^{-n_p-1}}{1-p^{-1}}
\sum_{\mu}\sum_{\pm}\sum_{\nu\in \ZZ^r}q^{3\nu/2}\int_{x\in \RR}\theta^\pm_\infty(x)\rmd x\\
\times&\prod_{i=1}^{r}\left[\left(1-\frac{1}{q_i}\right)^{-1}\int_{T\in \QQ_{q_i}}\int_{|N_i|_{q_i}= q_i^{-\nu_i}}\theta_{q_i}(T_i,N_i)\mu_{q_i}(\pm n^{-1}q^{-\nu}N_i)\rmd T_i\rmd N_i\right].
\end{align*}

By making the change of variable $N\mapsto \pm nq^\nu N$ we have
\begin{align*}
\sum_{\mu}\Tr(\mu(f^n))
=&4\sqrt{n}\prod_{p\mid n}\frac{1-p^{-n_p-1}}{1-p^{-1}}\prod_{i=1}^{r}\left(1-\frac{1}{q_i}\right)^{-1}
\sum_{\pm}\sum_{\nu\in \ZZ^r}q^{\nu/2}\int_{x\in \RR}\theta^\pm_\infty(x)\rmd x\\
\times&\int_{T\in \QQ_{S_\fin}}\left[\sum_{\mu}\prod_{i=1}^{r}\int_{N_i\in \ZZ_{q_i}^\times}\theta_{q_i}(T_i,\pm nq^{\nu}N_i)\mu_{q_i}(N_i)\rmd N_i\right]\rmd T,
\end{align*}
where the measure on $\ZZ_{q_i}^\times$ is induced from the measure on $\ZZ_{q_i}$. 
By \autoref{prop:characterzq} we have
\begin{equation}\label{eq:fourierzq}
\sum_{\mu}\prod_{i=1}^{r}\int_{N_i\in\ZZ_{q_i}^\times}\theta_{q_i}(T_i,\pm n q^\nu N_i)\mu_{q_i}(N_i)\rmd N_i
 = \sum_{\mu\in \widehat{\ZZ_{S_\fin}^\times}}\int_{N\in\ZZ_{S_\fin}^\times}\prod_{i=1}^{r}\theta_{q_i}(T_i,\pm n q^\nu N_i)(1-q_i^{-1})\mu(N)\rmd^\times N.
\end{equation}
Since \eqref{eq:fourierzq} is a finite sum by \autoref{cor:characterfinite}, by Fourier inversion formula \cite[Chapitre 2, \SSec 1, N°4]{bourbaki2019spectral}, it becomes
\[
\prod_{i=1}^{r}\theta_{q_i}(T_i,\pm n q^\nu)\int_{\ZZ_{S_\fin}^\times}\rmd^\times N=\theta_q^{\pm,\nu}(T).
\]
Hence we get
\[
\sum_{\mu}\Tr(\mu(f^n))= 4\sqrt{n}\prod_{p\mid n}\frac{1-p^{-n_p-1}}{1-p^{-1}}\sum_{\pm}\sum_{\nu\in \ZZ^r}q^{\nu/2}\int_{\RR}\theta^\pm_\infty(x)\rmd x\int_{T\in \QQ_{S_\fin}}\theta_q^{\pm,\nu}(T)\rmd T,
\]
which is precisely the $1$-dimensional term.
\end{proof}

\section{Computing the Eisenstein term}\label{sec:computeeisenstein}
Let $\xi_0$ be the representation of $\G(\AA)^1=Z_+\bs \G(\AA)$ that is induced from the trivial representation of $\B(\AA)$, where $\B$ denotes the standard Borel subgroup of $\G$. Thus $\xi_0$ consists of all smooth functions $\psi$ on $\G(\AA)$ that satisfy
\[
\psi(bg)=\left|\frac{x}{y}\right|^{1/2}\psi(g)
\] 
for any $g\in \G(\AA)$ and $b=(\begin{smallmatrix} x & z \\ 0 & y \end{smallmatrix})\in \B(\AA)$. Suppose that $\varphi=\bigotimes_{v\in \mf{S}}'\varphi_v$ is a smooth, compactly supported function on $\G(\AA)$ that is trivial on $Z_+$. We want to compute the trace of $\varphi$. It can be derived as follows:

For any $\psi\in \xi_0$, we have
\begin{align*}
   (\xi_0(\varphi)\psi)(g) & =\int_{Z_+\bs \G(\AA)}\psi(gh)\varphi(h)\rmd h =\int_{Z_+\bs \G(\AA)} \psi(h)\varphi(g^{-1}h)\rmd h \\
     & =\int_{\B(\AA)\bs \G(\AA)}\int_{Z_+\bs\B(\AA)}\psi(bh)\varphi(g^{-1}bh)\rmd_{L}b\rmd_{\B(\AA)\bs \G(\AA)}h\quad\text{(since $\G$ is unimodular)}\\
     & =\int_{\B(\AA)\bs \G(\AA)}\psi(h)\int_{Z_+\bs\B(\AA)}\left|\frac{x}{y}\right|^{1/2}\varphi(g^{-1}bh)\rmd_L b\rmd_{\B(\AA)\bs \G(\AA)}h,
\end{align*}
where $\rmd_L b$ denotes the left Haar measure on $\B(\AA)$. (See \cite[Theorem 2.58]{folland2016abstract} for the folding used in the above computation. Note that in loc. cit., the quotient is from the right, but an analogous result holds for the left quotient case.)

Thus we have 
\[
\Tr(\xi_0(\varphi))=\int_{\B(\AA)\bs \G(\AA)}\int_{Z_+\bs\B(\AA)}\left|\frac{x}{y}\right|^{1/2}\varphi(g^{-1}bg)\rmd_L b\rmd_{\B(\AA)\bs \G(\AA)}g.
\]

Suppose that $\mu=\bigotimes_{v\in \mf{S}}'\mu_v$ is a character on $\RR_{>0}\QQ^\times\bs \AA^\times$, we can define the twist $\xi_0\otimes \mu$ on $\G(\AA)$, which is also induced by the character
\[
\begin{pmatrix}
  x & z \\
  0 & y 
\end{pmatrix}\mapsto \mu(xy).
\]
A similar argument shows that 
\begin{proposition}\label{prop:eisensteinorigin}
We have
\[
\Tr((\xi_0\otimes \mu)(\varphi))=\int_{\B(\AA)\bs \G(\AA)}\int_{Z_+\bs\B(\AA)}\left|\frac{x}{y}\right|^{1/2}\mu(xy)\varphi(g^{-1}bg)\rmd_L b\rmd_{\B(\AA)\bs \G(\AA)}g.
\]
\end{proposition}

In this section, we will prove that the Eisenstein term is precisely $1/2$ times the sum of the traces of $(\xi_0\otimes\mu)(f^n)$. As in the previous section, we compute nonarchimedean analog of the local traces and combine them together to get the result using the abstract Fourier inversion formula.

\begin{proposition}\label{prop:eisensteinfolding}
We have
\[
\Tr((\xi_0\otimes \mu)(\varphi))=\int_{Z_+\bs \A(\AA)}\frac{|x-y|}{|xy|^{1/2}}\mu(xy)\int_{\A(\AA)\bs \G(\AA)}\varphi(g^{-1}tg)\rmd g \rmd t,
\]
where $\A$ is the diagonal torus in $\G$, $t=(\begin{smallmatrix} x & 0 \\  0 & y \end{smallmatrix})$, and $\rmd g$ is the measure on $\A(\AA)\bs \G(\AA)$.
\end{proposition}
\begin{proof}
By \autoref{prop:eisensteinorigin}  we have
\[
\begin{split}
\Tr((\xi_0\otimes \mu)(\varphi))&=\int_{\B(\AA)\bs \G(\AA)}\int_{Z_+\bs\B(\AA)}\left|\frac{x}{y}\right|^{1/2}\mu(xy)\varphi(g^{-1}bg)\rmd_L b\rmd g  \\
     & =\int_{\B(\AA)\bs \G(\AA)}\int_{Z_+\bs\A(\AA)}\int_{\N(\AA)}\left|\frac{x}{y}\right|^{1/2}\mu(xy) \varphi(g^{-1}tng)\rmd t\rmd n\rmd g,
\end{split}
\]
where $\N$ denotes the unipotent radical of $\B$. Now we want to make the change of variable $tn\mapsto n^{-1}tn$. Since
\[
\begin{pmatrix}
  x & 0 \\
  0 & y
\end{pmatrix}\begin{pmatrix}
  1 & z \\
  0 & 1
\end{pmatrix}=\begin{pmatrix}
  x & xz \\
  0 & y
\end{pmatrix}\quad\text{and}\quad\begin{pmatrix}
  1 & -z \\
  0 & 1
\end{pmatrix}\begin{pmatrix}
  x & 0 \\
  0 & y
\end{pmatrix}\begin{pmatrix}
  1 & z \\
  0 & 1
\end{pmatrix}=\begin{pmatrix}
  x & z(x-y) \\
  0 & y
\end{pmatrix},
\]
the change of variable is $x\mapsto x$, $y\mapsto y$ and $z\mapsto (x-y)z/x$. Since $\rmd t=\rmd x\rmd y/|xy|$ and $\rmd n=\rmd z$, we get
\begin{align*}
    & \int_{Z_+\bs\A(\AA)}\int_{\N(\AA)}\left|\frac{x}{y}\right|^{1/2}\mu(xy) \varphi(g^{-1}tng)\rmd t\rmd n \\
   =  & \int_{Z_+\bs\A(\AA)}\int_{\N(\AA)}\left|\frac{x}{y}\right|^{1/2}\mu(xy) \varphi(g^{-1}n^{-1}tng)\frac{|x-y|}{|x|}\rmd t\rmd n \\
   = &\int_{Z_+\bs\A(\AA)}\int_{\N(\AA)}\frac{|x-y|}{|xy|^{1/2}}\mu(xy) \varphi(g^{-1}n^{-1}tng)\rmd t\rmd n.
\end{align*}
Hence
\[
\begin{split}
\Tr((\xi_0\otimes \mu)(\varphi))&=\int_{Z_+\bs\A(\AA)}\int_{\N(\AA)}\int_{\B(\AA)\bs \G(\AA)}\frac{|x-y|}{|xy|^{1/2}}\mu(xy) \varphi(g^{-1}n^{-1}tng)\rmd t\rmd n\rmd g \\
     & =\int_{Z_+\bs\A(\AA)}\int_{\A(\AA)\bs \G(\AA)}\frac{|x-y|}{|xy|^{1/2}}\mu(xy) \varphi(g^{-1}tg)\rmd t\rmd g,
\end{split}
\]
where in the last step we used the following fact: If $\psi$ is a compactly supported function on $\G(\AA)$ that is invariant under the left translation of $\A(\AA)$, then
\[
\int_{\N(\AA)}\int_{\B(\AA)\bs \G(\AA)}\psi(ng)\rmd n\rmd g=\int_{\A(\AA)\bs \G(\AA)}\psi(g)\rmd g.
\]
This is because we can write $\psi(g)=\int_{\A(\AA)}\Psi(tg)\rmd t$ for some suitable $\Psi$ and then do folding and unfolding using \cite[Theorem 2.58]{folland2016abstract}.
\end{proof}

From this proposition we know that
\[
\Tr((\xi_0\otimes \mu)(\varphi))=\prod_{v\in \mf{S}}\Tr((\xi_0\otimes \mu_v)(\varphi_v)),
\]
where
\[
\Tr((\xi_0\otimes \mu_\infty)(\varphi_\infty))=\int_{Z_+\bs \A(\RR)}\frac{|x-y|_\infty}{|xy|_\infty^{1/2}}\mu_\infty(xy)\orb(\varphi_\infty;t)\rmd t
\]
and
\[
\Tr((\xi_0\otimes \mu_p)(\varphi_p))=\int_{ \A(\QQ_p)}\frac{|x-y|_p}{|xy|_p^{1/2}}\mu_p(xy)\orb(\varphi_p;t)\rmd t
\]
for a prime $p$.

\subsection{Local computation}
First we consider the archimedean case.
\begin{proposition}\label{prop:archimedeantraceeisenstein}
We have
\begin{equation}\label{eq:eisensteinarchimedeaninresult}
  \int_{Z_+\bs \A(\RR)}\frac{|x-y|_\infty}{|xy|_\infty^{1/2}}\mu_\infty(xy)\orb(f_\infty;t)\rmd t=4\left(\int_{|x|>1}\frac{\theta_\infty^+(x)}{\sqrt{x^2-1}}\rmd x+\mu_\infty(-1)\int_{\RR}\frac{\theta_\infty^-(x)}{\sqrt{x^2+1}}\rmd x\right).
\end{equation}
\end{proposition}
\begin{proof}
In the proof of this proposition we write $|\cdot|$, $f$ and $\mu$ for $|\cdot|_\infty$, $f_\infty$ and $\mu_\infty$, respectively.

We have
\begin{equation}\label{eq:eisensteinarchimedean}
\int_{Z_+\bs \A(\RR)}\frac{|x-y|}{|xy|^{1/2}}\mu(xy)\orb(f;t)\rmd t=\int_{Z_+\bs (\RR^\times \times \RR^\times)}\frac{|x-y|}{|xy|^{1/2}}\mu(xy)\orb(f;t)\frac{\rmd x\rmd y}{|xy|}.
\end{equation}
Now we make the change of variable $T=x+y$ and $N=xy$. We have $x,y=(T\pm \sqrt{T^2-4N})/2$. By symmetry we may assume that $x=(T+\sqrt{T^2-4N})/2$ and $y=(T-\sqrt{T^2-4N})/2$. A direct calculation shows that
\[
\rmd x\wedge \rmd y=\frac{\rmd T\wedge \rmd N}{\sqrt{T^2-4N}}.
\]
Moreover we have $xy=N$ and $(x-y)^2=T^2-4N$. Since the map $(x,y)\mapsto (T,N)$ is $2:1$ onto the image $T^2-4N>0$, \eqref{eq:eisensteinarchimedean} equals
\[
2\int_{Z_+\bs (T^2-4N>0)}\left|\frac{T^2-4N}{N}\right|^{1/2}\mu(N)\orb(f;t)\frac{\rmd T\rmd N}{|N|\sqrt{T^2-4N}}.
\]
Applying \eqref{eq:orbittheta} gives
\[
2\int_{Z_+\bs (T^2-4N>0)}\mu(N)\theta_\infty^{\sgn N}\legendresymbol{T}{2\sqrt{|N|}}\frac{\rmd T}{\sqrt{T^2-4N}}\frac{\rmd N}{|N|}.
\]
$Z_+\bs \{(T,N)\in \RR^2\}$ can be naturally identified with $(T,N)\in \RR\times\{\pm 1\}$. Since $\int_{1}^{\rme} \rmd a/a=1$, the integral over the quotient space $Z_+\bs (T^2-4N>0)$ reduces to an integral over
\[
\{a\cdot(T,N)=(aT,a^2N)|\ a\in [1,\rme],\ N=\pm 1,\ T^2-4N>0\}.
\]
Thus \eqref{eq:eisensteinarchimedean} equals
\[
2\sum_{\pm}\int_{1}^{\rme^2}\int_{T^2\mp 4N>0}\mu(\pm 1)\theta_\infty^{\pm}\legendresymbol{T}{2\sqrt{|N|}}\frac{\rmd T}{\sqrt{T^2\mp 4N}}\frac{\rmd N}{|N|}.
\]
By making change of variable $T\mapsto 2T\sqrt{|N|}$, this equals
\[
2\sum_{\pm}\int_{1}^{\rme^2}\int_{T^2\mp 1>0}\mu(\pm 1)\theta_\infty^{\pm}(T)\frac{\rmd T}{\sqrt{T^2\mp 1}}\frac{\rmd N}{|N|},
\]
which is precisely \eqref{eq:eisensteinarchimedeaninresult}.
\end{proof}
Now we consider the case $v=p\notin S$.
\begin{proposition}\label{prop:unramifiedtraceeisenstein}
\begin{enumerate}[itemsep=0pt,parsep=0pt,topsep=0pt,leftmargin=0pt,labelsep=2.5pt,itemindent=15pt,label=\upshape{(\arabic*)}]
  \item If $\mu_p$ is ramified, then
  \[
  \int_{ \A(\QQ_p)}\frac{|x-y|_p}{|xy|_p^{1/2}}\mu_p(xy)\orb(\triv_{\cX_p^{m}};t)\rmd t=0.
  \]
  \item If $\mu_p$ is unramified, then
  \[
  \int_{ \A(\QQ_p)}\frac{|x-y|_p}{|xy|_p^{1/2}}\mu_p(xy)\orb(\triv_{\cX_p^{m}};t)\rmd t=p^{m/2}\mu_p(p^m)(m+1).
  \]
\end{enumerate} 
\end{proposition}

\begin{proof}
By \autoref{thm:maximalcompact} we know that $\orb(\triv_{\cX_p^{m}};t)=0$ unless $x$ and $y$ are integral and $|xy|_p=p^{-m}$, and in this case we have  $\orb(\triv_{\cX_p^{m}};t)=p^{k_t}=|x-y|_p^{-1}$. Hence 
\[
\begin{split}
  \int_{ \A(\QQ_p)}\frac{|x-y|_p}{|xy|_p^{1/2}}\mu_p(xy)\orb(\triv_{\cX_p^{m}};t)\rmd t&=p^{m/2}\sum_{u=0}^{m}\int_{x\in p^u\ZZ_p^\times}\int_{a\in \ZZ_p^\times}\mu_p(p^{m}a)\rmd^\times a\rmd^\times x\\
  &=p^{m/2}(m+1)\mu_p(p^{m})\int_{a\in \ZZ_p^\times}\mu_p(a)\rmd^\times a,
\end{split}
\]
which equals $p^{m/2}(m+1)\mu_p(p^{m})$ if $\mu_p$ is unramified and equals $0$ if $\mu_p$ is ramified.
\end{proof}
Finally, we consider ramified case $v=p\in S$.
\begin{proposition}\label{prop:ramifiedtraceeisenstein}
If $p\in S$, we have
\[
\int_{\A(\QQ_{p})}\frac{|x-y|_{p}}{|xy|_{p}^{1/2}}\mu_{p}(xy)\orb(f_{p};t)\rmd t = 2\sum_{\nu\in \ZZ}p^{\nu}\left(1-\frac{1}{p}\right)^{-1}\int_{|N|_{p}= p^{-\nu}}\int_{Y_1(N)}\frac{\theta_{p}(T,N)\mu_{p}(N)}{\sqrt{|T^2-4N|'_{p}}}\rmd T\rmd N,
\]
where 
\[
Y_{1}(N)=\left\{T\in \QQ_p\,\middle|\,\omega(T):=\legendresymbol{(T^2-4N)|T^2-4N|_{p}'}{p}=1 \right\}.
\]
\end{proposition}
\begin{proof}
In the proof of this proposition we will write $|\cdot|$, $|\cdot|'$, $f$ and $\mu$ for $|\cdot|_{p}$, $|\cdot|_{p}'$, $f_{p}$ and $\mu_{p}$, respectively.

Write $T=x+y$ and $N=xy$. By the computation in \autoref{prop:ramifiedtrace} we know that 
\[
 \int_{\A(\QQ_{p})}\frac{|x-y|}{|xy|^{1/2}}\mu(xy)\orb(f;t)\rmd t = 2\left(1-\frac{1}{p}\right)^{-2}\int_{\gamma_{T,N}\in \A(\QQ_{p})}\frac{1}{|N|^{3/2}}\orb(f;T,N)\mu(N)\rmd T\rmd N.
\]
We have $\sqrt{|T^2-4N|'}=p^{-k}$ by \autoref{rem:absoluteprime}. Moreover, since $\chi(p)=1$ in this case, we have
\[
\theta_{p}(T,N)=\frac{1}{|N|^{1/2}}\left(1-\frac{1}{p}\right)^{-1}p^{-k}\orb(f;T,N).
\]
Since $t=t_{T,N}$ being split is equivalent to $\omega(T)=1$, we get
\begin{align*}
     \int_{\A(\QQ_{p})}\frac{|x-y|}{|xy|^{1/2}}\mu(xy)\orb(f;t)\rmd t=& 2\left(1-\frac{1}{p}\right)^{-1}\int_{N\in \QQ_p}\int_{Y_1(N)}\frac{1}{|N|}\frac{\theta_{p}(T,N)\mu(N)}{\sqrt{|T^2-4N|'}}\rmd T\rmd N\\
  =   & 2\sum_{\nu\in \ZZ}p^{\nu}\left(1-\frac{1}{p}\right)^{-1}\int_{|N|_{p}= p^{-\nu}}\int_{Y_1(N)}\frac{\theta_{p}(T,N)\mu(N)}{\sqrt{|T^2-4N|'}}\rmd T\rmd N.\qedhere
\end{align*}
\end{proof}

\subsection{Global computation}
We now analyze the global sum
\begin{equation}\label{eq:sumeisenstein}
\frac{1}{2}\sum_{\mu}\Tr((\xi_0\otimes\mu)(f^n)),
\end{equation}
where $\mu$ runs over all characters on $\RR_{>0}\QQ^\times\bs \AA^\times$. 

\begin{lemma}\label{lem:characterfiniteeisenstein}
For any prime $p\in S$, there exists $M>0$ such that $\Tr((\xi_0\otimes\mu_{p})(f^n_{p}))=\Tr((\xi_0\otimes\mu_{p})(f_{p}))=0$ if $\mu_p$ is not trivial on $1+p^{M}\ZZ_{p}$.
\end{lemma}
\begin{proof}
Since $\chi(p)=1$ for a split torus, by \autoref{cor:shalika}, for diagonal torus $\A(\QQ_{p})$ the behavior near the center is
\[
\orb(f_p;\gamma)=\lambda_2\left(1-\frac{1}{p}\right)p^{k_\gamma}.
\]
Since $|x-y|_{p}|xy|_{p}^{-1/2}=p^{-k_{\gamma}}$, we conclude that 
\[
\Theta_{p}(t)=\frac{|x-y|_{p}}{|xy|_{p}^{1/2}}\orb(f_{p};t)
\]
is a smooth function on $\A(\QQ_{p})$. 

Since $\Tr(\supp(f_{p}))$ and $\det(\supp(f_{p}))$ are compact, we know that for $t$ with $\orb(f_{p};t)\neq 0$, $|x+y|_p$ and $\log|xy|_p$ are bounded. Thus $\log|x|_p$ and $\log|y|_p$ are bounded and hence the support of $\orb(f_{p};t)$ is compact. Thus $\Theta_p(t)$ is smooth and compactly supported. Hence there exists $M>0$ such that for any $a\in 1+p^{M}\ZZ_{p}$ and $t\in \A(\QQ_{p})$, we have
\[
\Theta_{p}(t)=\Theta_{p}\left(\begin{pmatrix}
                      a & 0 \\
                      0 & 1 
                    \end{pmatrix}t\right).
\]
Therefore for any $a\in 1+p^{M}\ZZ_{p}$ we have
\[
\Tr((\xi_0\otimes\mu_{p})(f_{p}))=\int_{\A(\QQ_{p})}\Theta_{p}(t)\mu_{p}(xy)\rmd t=\int_{\A(\QQ_{p})}\Theta_{p}(t)\mu_{p}(axy)\rmd t=\mu_{p}(a)\Tr((\xi_0\otimes\mu_{p})(f_{p})).
\]
If $\Tr((\xi_0\otimes\mu_{p})(f_{p}))\neq 0$, then $\mu_{p}(a)=1$ for all $a\in 1+p^{M}\ZZ_{p}$. 
\end{proof}
An argument analogous to that in \autoref{cor:characterfinite} shows that 
\begin{corollary}\label{cor:characterfiniteeisenstein}
\eqref{eq:sumeisenstein} is a finite sum.
\end{corollary}

Finally, we prove the main theorem in this section.
\begin{theorem}\label{thm:eisensteincompute}
The Eisenstein term \eqref{eq:eisensteinterm} equals
\[
\frac{1}{2}\sum_{\mu}\Tr((\xi_0\otimes\mu)(f^n))=\frac{1}{2} \sum_{\mu\in (\QQ^\times\RR_{>0}\prod_{\ell\notin S}\ZZ_\ell^\times\bs\AA^\times)\sphat}\Tr((\xi_0\otimes\mu)(f^n)).
\]
\end{theorem}

\begin{proof}
By \autoref{prop:archimedeantraceeisenstein}, \autoref{prop:unramifiedtraceeisenstein} and \autoref{prop:ramifiedtraceeisenstein}, and noting that $\mu_p$ is unramified when $p\notin S$, we have
\begin{align*}
    \frac{1}{2}\sum_{\mu}\Tr((\xi_0\otimes\mu)(f^n))
    &= \frac{1}{2}\sum_{\mu}4\left(\int_{|x|>1}\frac{\theta_\infty^+(x)}{\sqrt{x^2-1}}\rmd x+\mu_{\infty}(-1)\int_{\RR}\frac{\theta_\infty^-(x)}{\sqrt{x^2+1}}\rmd x\right)\prod_{p\mid n}\mu_p(n)(n_p+1)\\
    &\times 2^r\sum_{\nu\in \ZZ^r}q^{\nu}\prod_{i=1}^{r}\left[\left(1-\frac{1}{q_i}\right)^{-1} \int_{|N_i|_{q_i}= q_i^{-\nu_i}}\int_{Y_{i,1}(N_i)} \frac{\theta_{q_i}(T_i,N_i)}{\sqrt{|T_i^2-4N_i|'_{q_i}}}\mu_{q_i}(N_i)\rmd T_i\rmd N_i\right],
\end{align*}
where 
\[
Y_{i,1}(N_i)=\left\{T_i\in \QQ_{q_i}\,\middle|\,\legendresymbol{(T_i^2-4N_i)|T_i^2-4N_i|_{q_i}'}{q_i}=1 \right\}
\]
so that $Y_{\mathbf{1}}=Y_{1,1}(\pm nq^\nu)\times\cdots\times Y_{r,1}(\pm nq^\nu)$.

Since 
\[
\prod_{i=1}^{r}\mu_{q_i}(q^{-\nu})=1,\quad \mu_\infty(-1)=\prod_{i=1}^{r}\mu_{q_i}(-1)\quad\text{and}\quad \prod_{p\mid n}\mu_p(n)=\prod_{i=1}^{r}\mu_{q_i}(n^{-1}),
\]
we get
\begin{align*}
 \frac{1}{2}\sum_{\mu}
  \Tr((\xi_0\otimes\mu)(f^n))
   =&2^{r+1}\prod_{p\mid n}(n_p+1)
\sum_{\mu}\sum_{\pm}\sum_{\nu\in \ZZ^r}q^{\nu}\int_{X_0}\frac{\theta^\pm_\infty(x)}{\sqrt{x^2\mp 1}}\rmd x\\
\times&\prod_{i=1}^{r}\left[\left(1-\frac{1}{q_i}\right)^{-1}\int_{|N_i|_{q_i}= q_i^{-\nu_i}}\int_{Y_{i,1}(N_i)}\frac{\theta_{q_i}(T_i,N_i)}{\sqrt{|T_i^2-4N_i|'_{q_i}}}\mu_{q_i}(\pm n^{-1}q^{-\nu}N_i)\rmd T_i\rmd N_i\right].
\end{align*}

By making the change of variable $N_i\mapsto \pm nq^\nu N_i$ for each $i$, we have
\[
\begin{split}
\frac{1}{2}\sum_{\mu}
  \Tr((\xi_0\otimes\mu)(f^n))= &2^{r+1}\prod_{p\mid n}(n_p+1)\prod_{i=1}^{r}\left(1-\frac{1}{q_i}\right)^{-1}
\sum_{\pm}\sum_{\nu\in \ZZ^r}\int_{X_0}\frac{\theta^\pm_\infty(x)}{\sqrt{x^2\mp 1}}\rmd x\\
\times &\sum_{\mu}\left[\int_{N\in\ZZ_{S_\fin}^\times}\prod_{i=1}^{r}\int_{Y_{1,i}(\pm nq^\nu N_i)} \frac{\theta_{q_i}(T_i,\pm n q^\nu N_i)}{\sqrt{|T_i^2\mp 4n q^\nu N_i|_{q_i}'}}\mu_{q_i}(N_i)\rmd T\rmd N\right].
\end{split}
\]

Using \autoref{prop:characterzq} and the Fourier inversion formula (since the sum over $\mu$ is  finite by \autoref{cor:characterfiniteeisenstein}), we have
\begin{align*}
   & \sum_{\mu}\left[\int_{N\in\ZZ_{S_\fin}^\times}\prod_{i=1}^{r}\int_{Y_{1,i}(\pm nq^\nu N_i)} \frac{\theta_{q_i}(T_i,\pm n q^\nu N_i)}{\sqrt{|T_i^2\mp 4n q^\nu N_i|_{q_i}'}}\mu_{q_i}(N_i)\rmd T\rmd N\right] \\
  = & \prod_{i=1}^{r}\left(1-\frac{1}{q_i}\right)\int_{Y_{1,i}(\pm nq^\nu)} \frac{\theta_{q_i}(T_i,\pm n q^\nu )}{\sqrt{|T_i^2\mp 4n q^\nu|_{q_i}'}}\rmd T_i =\prod_{i=1}^{r}\left(1-\frac{1}{q_i}\right)\int_{Y_{\mathbf{1}}}\frac{\theta_q^{\pm,\nu}(T)}{\sqrt{|T^2\mp 4n q^\nu|_q'}}\rmd T.
\end{align*}
Hence
\[
\frac{1}{2}\sum_{\mu}
  \Tr((\xi_0\otimes\mu)(f^n))=2^{r+1}\prod_{p\mid n}(n_p+1)
\sum_{\pm}\sum_{\nu\in \ZZ^r}\int_{X_0}\frac{\theta^\pm_\infty(x)}{\sqrt{x^2\mp 1}}\rmd x\int_{Y_{\mathbf{1}}}\frac{\theta_q^{\pm,\nu}(T)}{\sqrt{|T^2\mp 4n q^\nu|_q'}}\rmd T,
\]
which is precisely the Eisenstein term.
\end{proof}
\begin{remark}\label{rem:1dim}
The coefficient of the Eisenstein term in \cite{altug2015} was wrong.
Actually, the key point of such a strategy is to kill the nontempered representations in the trace formula, which are precisely $1$-dimensional representations. What we are interested in is $I_{\mathrm{cusp}}(f^{n,\rho})$ as shown in the introductory section, and nontempered representations give large contributions.
\end{remark}

\subsection{The main result}
Combining \autoref{thm:ellipticpoisson}, \autoref{thm:tracexi0}, \autoref{thm:maintermcompute}, and \autoref{thm:eisensteincompute} together, we get our main result:

\begin{theorem}\label{thm:finaltrace}
For any $\vartheta\in \lopen 0,1\ropen$, we have
\[
I_{\el}(f^n)=\sum_{\mu}\Tr(\mu(f^n))-\frac{1}{2}\sum_{\mu}\Tr((\xi_0\otimes\mu)(f^n))-\Sigma(\square)+\Sigma(0)+\Sigma(\xi\neq 0),
\]
where
\begin{align*}
\Sigma(\square)=&2\sum_{\pm}\sum_{\nu\in \ZZ^r}\sum_{\substack{T\in \ZZ^S\\ T^2\mp 4nq^\nu= \square}}\sum_{f^2\mid T^2\mp 4nq^\nu}  \sum_{k\in \ZZ_{(S)}^{>0}}\frac{1}{kf}\legendresymbol{(T^2\mp 4nq^\nu)/f^2}{k}\theta_\infty^\pm\legendresymbol{T}{2n^{1/2}q^{\nu/2}}\\
    \times &\prod_{i=1}^{r}\theta_{q_i}^{\pm,\nu}(T)\left[F\legendresymbol{kf^2}{|T^2\mp 4nq^\nu|_{\infty,q}'^\vartheta}+\frac{kf^2}{\sqrt{|T^2\mp 4nq^\nu|_{\infty,q}'}}V\legendresymbol{kf^2}{|T^2\mp 4nq^\nu|_{\infty,q}'^{1-\vartheta}}\right],
\end{align*}
\begin{align*}
\Sigma(0)=&4\sqrt{n}\sum_{\pm}\sum_{\nu\in \ZZ^r}q^{\nu/2}\int_{x\in \RR}\int_{y\in \QQ_{S_\fin}}\theta_\infty^\pm(x)\theta_{q}^{\pm,\nu}(y)
  \left[\frac{1}{\dpii}\int_{(-1)}\widetilde{F}(s)\frac{\zeta(2s+2)}{\zeta(s+2)} \prod_{i=1}^{r}\frac{1-q_i^{-2s-2}}{1-q_i^{-s-2}}\right.\\
  \times&\left.\prod_{p\mid n}\frac{1-p^{-(s+1)(n_p+1)}}{1-p^{-s-1}}\left(\frac{(4nq^\nu)^{-\vartheta}}{|x^2\mp 1|_\infty^\vartheta|y^2\mp 4nq^\nu|_q'^\vartheta}\right)^{-s}\rmd s\right]\rmd x\rmd y\\
  +&2\sum_{\pm}\sum_{\nu\in \ZZ^r}\int_{x\in \RR}\int_{y\in \QQ_{S_\fin}}\frac{\theta_\infty^\pm(x)\theta_q^{\pm,\nu}(y)}{\sqrt{|x^2\mp 1|_\infty|y^2\mp 4nq^\nu|_q'}}\left[\frac{\sqrt{\uppi}}{\dpii}\int_{\cC_v}\widetilde{F}(s) \frac{\Gamma(\frac{\iota+s}{2})}{\Gamma(\frac{\iota+1-s}{2})}\frac{\zeta(2s)}{\zeta(s+1)} \right.\\
  \times&\left. \prod_{i=1}^{r}\frac{1-\epsilon_i q_i^{s-1}}{1-\epsilon_i q_i^{-s}}\frac{1-q_i^{-2s}}{1-q_i^{-s-1}}\prod_{p\mid n}\frac{1-p^{-s(n_p+1)}}{1-p^{-s}} \left(\frac{\uppi (4nq^\nu)^{\vartheta-1}}{|x^2\mp 1|_\infty^{1-\vartheta}|y^2\mp 4nq^\nu|_q'^{1-\vartheta}}\right)^{-s}\rmd s\right]
  \rmd x\rmd y
\end{align*}
and finally,
\begin{align*}
  \Sigma(\xi\neq 0)=&4\sqrt{n}\sum_{\pm}\sum_{\nu\in \ZZ^r}q^{\nu/2}\sum_{k,f\in \ZZ_{(S)}^{>0}}\frac{1}{k^2f^3}\sum_{\xi\in \ZZ^S-\{0\}}\Kl_{k,f}^S(\xi,\pm nq^\nu)\\
   \times&\int_{x\in\RR}\int_{y\in\QQ_{S_\fin}}\theta_\infty^\pm(x)\theta_{q}^{\pm,\nu}(y)\left[F\legendresymbol{kf^2(4nq^\nu)^{-\vartheta}}{|x^2\mp 1|_\infty^\vartheta|y^2\mp 4nq^\nu|_q'^\vartheta}+\frac{kf^2n^{-1/2}q^{-\nu/2}}{2\sqrt{|x^2\mp 1|_\infty|y^2\mp 4nq^\nu|_q'}}\right.\\
     \times&\left.V\legendresymbol{kf^2(4nq^\nu)^{\vartheta-1}}{|x^2\mp 1|_\infty^{1-\vartheta}|y^2\mp 4nq^\nu|_q'^{1-\vartheta}}\right]\rme\legendresymbol{-2x\xi n^{1/2}q^{\nu/2}}{kf^2}\rme_{q}\legendresymbol{-y\xi}{kf^2}\rmd x\rmd y.
\end{align*}
\end{theorem}

Removing all ramified part in the above formulas (e.g. $q^\nu$, $y$, $\theta_{q}^{\pm,\nu}(y)$,\dots), we obtain precisely Altu\u{g}'s result \autoref{thm:maintheoremunramify}. Hence this theorem is a direct generalization of Altu\u{g}'s work.

\appendix
\section{Local orbital integral computations}\label{sec:prooflocalorbit}
We will prove \autoref{thm:maximalcompact} and \autoref{thm:iwahori} in this section. We assume that $\Delta\in \cO_{E_p}^\times$ if $E_p$ is a field in the following. We can always do this since we can replace $\Delta$ by $\Delta+a$ for some $a\in \ZZ_p$ for $p\neq 2$, or $p=2$ and $\beta\in p\ZZ_p$. If $p=2$, $\alpha\in p\ZZ_p$ and $\beta\in \ZZ_p^\times$, then $\beta^2+4\alpha\equiv 1\,(8)$. Hence the discriminant $\beta^2+4\alpha$ is a square and therefore $E_p$ is not a field.

\smallskip

\begin{proof}[Proof of \autoref{thm:maximalcompact}]
The proof of this theorem can be found in many references (e.g. \cite{langlands2004}, \cite{kottwitz2005}, and \cite{espinosa2022} for the number field case). We will prove it again for use in the next theorem.

Recall that the local orbital integral is
\[
\orb(\triv_{\cX_p^{m}};\gamma)=\int_{\G_\gamma(\QQ_p)\bs\G(\QQ_p)}\triv_{\cX_p^{m}}(g^{-1}\gamma g)\rmd g.
\]
Hence we only need to compute the measure of $g\in \G_\gamma(\QQ_p)\bs\G(\QQ_p)$ such that $g^{-1}\gamma g\in \cX_p^{m}$.

By the definition of $\cX_p^{m}$, it is easy to see that $g^{-1}\gamma g\in \cX_p^{m}$ if and only if $g^{-1}\gamma g$ fixes the lattice $L_0=\ZZ_p\oplus\ZZ_p$ (which means that $g^{-1}\gamma gL_0\subseteq L_0$) and $\mathopen{|}\det\gamma\mathclose{|}_p=\mathopen{|}\det(g^{-1}\gamma g)\mathclose{|}_p=p^{-m}$. Thus, if $\mathopen{|}\det\gamma\mathclose{|}_p\neq p^{-m}$, $g^{-1}\gamma g$ never lies in $\cX_p^{m}$ and the integral is $0$. 

From now on, we assume that $\mathopen{|}\det\gamma\mathclose{|}_p=p^{-m}$. In this case, $g^{-1}\gamma g\in \cX_p^{m}$ if and only if $g^{-1}\gamma g$ fixes the lattice $L_0=\ZZ_p\oplus\ZZ_p$, if and only if $\gamma$ fixes the lattice $gL_0$.

Since $\G(\QQ_p)$ acts transitively on the set of all lattices in $\QQ_p^2$ and the stabilizer of $L_0$ is precisely $\G(\ZZ_p)$, we can identify the space of all lattices in $E$ with $\G(\QQ_p)/\G(\ZZ_p)$ by identifying $E$ with $\QQ_p^2$ via a basis of $E$. Therefore, we only need to consider the elements in
\[
\G_\gamma(\QQ_p)\bs\G(\QQ_p)/\G(\ZZ_p)
\]
and find the contributing measure for each element satisfying $g^{-1}\gamma g\in \cX_p^{m}$.

Suppose that $g\in \G_\gamma(\QQ_p)\bs\G(\QQ_p)/\G(\ZZ_p)$ such that $g^{-1}\gamma g\in \cX_p^{m}$. Then the contributing measure is
\[
\begin{split}
   \vol(\G_\gamma(\QQ_p)\bs \G_\gamma(\QQ_p)g\G(\ZZ_p)) & =\vol(\G_\gamma(\QQ_p)\bs \G_\gamma(\QQ_p)g\G(\ZZ_p)g^{-1}) \\
     & =\vol((\G_\gamma(\QQ_p)\cap g\G(\ZZ_p)g^{-1})\bs g\G(\ZZ_p)g^{-1}).
\end{split}
\]
By the normalization of the measures, we have $\vol(g\G(\ZZ_p)g^{-1})=\vol(\G(\ZZ_p))=1$. Hence the contributing measure is
\[
\frac{1}{\vol_{\G_\gamma(\QQ_p)}(\G_\gamma(\QQ_p)\cap g\G(\ZZ_p)g^{-1})}.
\]
We will see in the proof that $\G_\gamma(\QQ_p)\cap g\G(\ZZ_p)g^{-1}\subseteq \G_\gamma(\ZZ_p)$. Hence the contributing measure for $g^{-1}\gamma g\in \cX_p^{m}$ is
\begin{equation}\label{eq:contributemeasure1}
[\G_\gamma(\ZZ_p):\G_\gamma(\QQ_p)\cap g\G(\ZZ_p)g^{-1}].
\end{equation}

Now we consider the split case, the inert case, and the ramified case separately.

\underline{\emph{Case 1:}}\ \  $p$ is split in $E$, that is, $E_p\cong \QQ_p^2$. Under this identification, $\G_\gamma$ is the diagonal torus. 

\begin{lemma}\label{lem:splitk}
A full representative of $\G_\gamma(\QQ_p)\bs\G(\QQ_p)/\G(\ZZ_p)$ can be
\[
\left\{\begin{pmatrix}
  1 & p^{-j} \\
  0 & 1 
\end{pmatrix}\middle|\, j\in \ZZ_{\geq 0}\right\}.
\]
\end{lemma}

\begin{insertproof}
It can be shown by direct computation, the details are left to the reader (or see \cite{langlands2004} or \cite{kottwitz2005}). 
\end{insertproof}

Suppose that $g=(\begin{smallmatrix}  1 & p^{-j}  \\ 0 & 1 \end{smallmatrix})$ and $\gamma=(\begin{smallmatrix}  \gamma_1 & 0  \\ 0 & \gamma_2 \end{smallmatrix})$ (we identify $\gamma_i\in E$ with the roots of the characteristic polynomial of $\gamma$ in $\QQ_p$). We have
\[
gL=\ZZ_p\begin{pmatrix}
   1 \\
   0 
\end{pmatrix}+\ZZ_p\begin{pmatrix}
   p^{-j} \\
   1 
\end{pmatrix}.
\]
Hence $\gamma$ fixes $gL$ if and only if
\[
\begin{pmatrix}
   \gamma_1 \\
   0 
\end{pmatrix}=\gamma\begin{pmatrix}
   1 \\
   0 
\end{pmatrix}\in \ZZ_p\begin{pmatrix}
   1 \\
   0 
\end{pmatrix}+\ZZ_p\begin{pmatrix}
   p^{-j} \\
   1 
\end{pmatrix}\quad\text{and}\quad\begin{pmatrix}
   \gamma_1p^{-j} \\
   \gamma_2
\end{pmatrix}=\gamma\begin{pmatrix}
   p^{-j} \\
   1 
\end{pmatrix}\in \ZZ_p\begin{pmatrix}
   1 \\
   0 
\end{pmatrix}+\ZZ_p\begin{pmatrix}
   p^{-j} \\
   1 
\end{pmatrix}.
\]
From this we know that $\gamma_2$ is integral, and then $\gamma$,$\gamma_1$ are integral. Also, $\gamma_1p^{-j}\in \ZZ_p+\gamma_2p^{-j}$ and hence $(\gamma_1-\gamma_2)p^{-j}\in \ZZ_p$. Conversely, if $\gamma_1$ and $\gamma_2$ are integral and $(\gamma_1-\gamma_2)p^{-j}\in \ZZ_p$, then
\[
\begin{pmatrix}
   \gamma_1 \\
   0 
\end{pmatrix}= \gamma_1\begin{pmatrix}
   1 \\
   0 
\end{pmatrix}+0\begin{pmatrix}
   p^{-j} \\
   1 
\end{pmatrix}\quad\text{and}\quad\begin{pmatrix}
   \gamma_1p^{-j} \\
   \gamma_2
\end{pmatrix}=(\gamma_1-\gamma_2)p^{-j}\begin{pmatrix}
   1 \\
   0 
\end{pmatrix}+\gamma_2\begin{pmatrix}
   p^{-j} \\
   1 
\end{pmatrix}.
\]
Hence $\gamma$ fixes $gL_0$.

Write $\Delta=(\Delta_1,\Delta_2)\in E$ and we identify $\Delta$ with $\Delta_1$. Then $\Delta_2=\overline{\Delta}$. Since $\ZZ_p+\ZZ_p\Delta=\cO_{E_p}$, and
\[
\ZZ_p\begin{pmatrix}
   1 \\
   0 
\end{pmatrix}+\ZZ_p\begin{pmatrix}
   0 \\
   1 
\end{pmatrix}
\]
is also $\cO_{E_p}$, we have
\[
\begin{pmatrix}
  1 & \Delta_1 \\
  1 & \Delta_2 
\end{pmatrix}\in \G(\ZZ_p).
\]
Thus $|\Delta-\overline{\Delta}|_p=|\Delta_1-\Delta_2|_p=1$. Hence
\[
p^{-k}=|\kappa|_p=\frac{|\gamma_1-\gamma_2|_p}{|\Delta-\overline{\Delta}|_p}=|\gamma_1-\gamma_2|_p.
\]
Thus $\gamma$ fixes $gL_0$ if and only if $k\geq j$.

Now we find the contributing measure for each $j$. Suppose that
\[
\begin{pmatrix}
  a & 0 \\
  0 & d 
\end{pmatrix}\in \G_\gamma(\QQ_p)\cap g\G(\ZZ_p)g^{-1}.
\]
Then
\[
\begin{pmatrix}
  a & p^{-j}(a-d) \\
  0 & d 
\end{pmatrix}=g^{-1}\begin{pmatrix}
  a & 0 \\
  0 & d 
\end{pmatrix}g\in \G(\ZZ_p).
\]
Thus $a\equiv d\pmod {p^j}$. The subgroup
\[
\left\{\begin{pmatrix}
  a & 0 \\
  0 & d 
\end{pmatrix}\in \G(\ZZ_p)\,\middle|\,a\equiv d\pmod {p^j}\right\}
\]
has index $1$ if $j=0$ and has index $\#((\ZZ_p/p^j\ZZ_p)^\times)=p^j-p^{j-1}$ if $j\geq 1$. Thus the orbital integral is
\[
1+\sum_{j=1}^{k}(p^j-p^{j-1}).
\]
Since $\chi(p)=1$ in this case, we get the formula \eqref{eq:maximallocal} in the split case.

\underline{\emph{Case 2:}}\ \  $p$ is inert or ramified in $E_p$. In this case, $E_p$ is a field and we can identify $E_p$ with $\QQ_p^2$ via the basis $\{1,\Delta\}$. Thus we can embed $E_p$ into $\G(\QQ_p)$ by identifying $\xi\in E_p$ with the matrix of the linear transformation $\eta\mapsto \xi\eta$ with respect to $\{1,\Delta\}$. Suppose that $\Delta^2=\alpha+\beta\Delta$. We then have
\[
1=\begin{pmatrix}
  1 & 0 \\
  0 & 1 
\end{pmatrix}\qquad \text{and}\qquad \Delta=\begin{pmatrix}
  0 & \alpha \\
  1 & \beta 
\end{pmatrix}.
\]
Since $\Delta$ is integral, we have $\alpha,\beta\in \ZZ_p$. Moreover, since we assumed that $\Delta\in \cO_{E_p}^\times$, we have $\alpha\in \ZZ_p^\times$.

\begin{lemma}\label{lem:elliptick}
A full representative of $\G_\gamma(\QQ_p)\bs\G(\QQ_p)/\G(\ZZ_p)$ can be chosen to be
\[
\left\{\begin{pmatrix}
  1 & 0 \\
  0 & p^{j} 
\end{pmatrix}\middle|\, j\in \ZZ_{\geq 0}\right\}.
\]
\end{lemma}

\begin{insertproof}
By direct computation, or see \cite{langlands2004} or \cite{kottwitz2005}. 
\end{insertproof}
Now we come back to the proof of the inert case and the ramified case. By the above lemma we know that it suffices to consider the lattices $L=gL_0$ for $g=(\begin{smallmatrix} 1 & 0 \\ 0 & p^{j} \end{smallmatrix})$ with $j\geq 0$. In this case, we have $L=\ZZ_p+p^j\ZZ_p\Delta$. Hence $\gamma L\subseteq L$ if and only if $\gamma\in \ZZ_p+p^j\ZZ_p\Delta$ and $p^j\Delta \gamma\in \ZZ_p+p^j\ZZ_p\Delta$. The first condition implies that $\gamma$ is integral. Now we analyze the second condition under the assumption that $\gamma$ is integral. Suppose that $\gamma=a+p^j b\Delta$ with $a,b\in \ZZ_p$. Then $\gamma-\overline{\gamma}=p^j b(\Delta-\overline{\Delta})$. This means that $\kappa=p^j b$. Hence we must have $k\geq j$. Conversely, if $k\geq j$ and $\gamma$ is integral, then $\gamma=a+\kappa\Delta\in \ZZ_p+p^j\ZZ_p\Delta$ and
\[
p^j\Delta\gamma\in p^j\Delta\ZZ_p+p^{2j}\Delta^2\ZZ_p=p^j\Delta\ZZ_p+p^{2j}(\alpha+\beta\Delta)\ZZ_p\subseteq \ZZ_p+p^j\ZZ_p\Delta.
\]
Thus $\gamma$ fixes $L=gL_0$ if and only if $k\geq j$.

The last work is to compute the contributing measure for each $j$. Suppose that 
\[
a+c\Delta\in \G_\gamma(\QQ_p)\cap \begin{pmatrix}
  1 & 0 \\
  0 & p^{j} 
\end{pmatrix}\G(\ZZ_p)\begin{pmatrix}
  1 & 0 \\
  0 & p^{-j} 
\end{pmatrix}.
\]
Then
\[
\begin{pmatrix}
  a & cp^{j}\alpha \\
  cp^{-j} & a+c\beta
\end{pmatrix}=\begin{pmatrix}
  1 & 0 \\
  0 & p^{-j} 
\end{pmatrix}(a+c\Delta)\begin{pmatrix}
  1 & 0 \\
  0 & p^{j} 
\end{pmatrix}\in \G(\ZZ_p).
\]

If $j=0$, this condition is just $a+c\Delta\in \G(\ZZ_p)$, that is, $a+c\Delta\in \cO_{E_p}^\times=\G_\gamma(\ZZ_p)$. Hence the contributing measure is $1$.
If $j\geq 1$, this condition is equivalent to $c\in p^j\ZZ_p$ and $a\in \ZZ_p^\times$. Hence the index in $\G_\gamma(\ZZ_p)$ is
\[
[\cO_{E_p}^\times:\ZZ_p^\times +p^j\ZZ_p\Delta]=\frac{[\cO_{E_p}^\times:1+p^j\cO_{E_p}]}{[\ZZ_p^\times +p^j\ZZ_p\Delta:1+p^j\cO_{E_p}]}.
\]

If we choose any representative $\{h_1,\dots,h_n\}$ of $\ZZ_p^\times/(1+p^j\ZZ_p)$, it is easy to see that it is also a representative of $(\ZZ_p^\times +p^j\ZZ_p\Delta)/(1+p^j\cO_{E_p})$. Hence 
\[
[\cO_{E_p}^\times:\ZZ_p^\times +p^j\ZZ_p\Delta]=\frac{[\cO_{E_p}^\times:1+p^j\cO_{E_p}]}{[\ZZ_p^\times :1+p^j\ZZ_p]}.
\]

If $E_p/\QQ_p$ is inert, then $p$ is a prime in $E_p$ with inertia degree $2$. Thus
\[
[\cO_{E_p}^\times:\ZZ_p^\times +p^j\ZZ_p\Delta]=\frac{[\cO_{E_p}^\times:1+p^j\cO_{E_p}]}{[\ZZ_p^\times :1+p^j\ZZ_p]}=\frac{p^{2j-2}(p^2-1)}{(p-1)p^{j-1}}=p^{j-1}(p+1).
\]
Hence the orbital integral is
\[
1+\sum_{j=1}^{k}p^{j-1}(p+1).
\]
Since $\chi(p)=-1$ in this case, we get the formula \eqref{eq:maximallocal} in the inert case. 

If $E_p/\QQ_p$ is ramified, then $p=\mf{p}^2$, where $\mf{p}$ is a prime in $E_p$ with inertia degree $1$. Thus
\[
[\cO_{E_p}^\times:\ZZ_p^\times +p^j\ZZ_p\Delta]=\frac{[\cO_{E_p}^\times:1+\mf{p}^{2j}\cO_{E_p}]}{[\ZZ_p^\times :1+p^j\ZZ_p]}=\frac{p^{2j-1}(p-1)}{(p-1)p^{j-1}}=p^j.
\]
Hence the orbital integral is
\[
1+\sum_{j=1}^{k}p^j.
\]
Since $\chi(p)=0$ in this case, we get the formula \eqref{eq:maximallocal} in the ramified case. 
\end{proof}

\smallskip

\begin{proof}[Proof of \autoref{thm:iwahori}]
The idea is similar. By the definition of the Iwahori subgroup, it is easy to see that $g^{-1}\gamma g\in \cI_p$ if and only if $g^{-1}\gamma g$ fixes the lattices $L_0=\ZZ_p\oplus\ZZ_p$ and $L_1=\ZZ_p\oplus p\ZZ_p$, and satisfies $\mathopen{|}\det\gamma\mathclose{|}_p=\mathopen{|}\det(g^{-1}\gamma g)\mathclose{|}_p=1$. Hence for $\mathopen{|}\det\gamma\mathclose{|}_p\neq 1$, $g^{-1}\gamma g$ never lies in $\cI_p$ and hence the integral is $0$. 

From now on, we assume that $\mathopen{|}\det\gamma\mathclose{|}_p=1$. In this case, $g^{-1}\gamma g\in \cI_p$ if and only if $g^{-1}\gamma g$ fixes $L_0$ and $L_1$, if and only if $\gamma$ fixes the lattices $gL_0$ and $gL_1$.

Observe that if $h\in \cI_p$, then $ghL_0=gL_0$ and $ghL_1=gL_1$. Hence we only need to consider representatives in the double coset
\[
\G_\gamma(\QQ_p)\bs \G(\QQ_p)/ \cI_p.
\]
If $g\in \G_\gamma(\QQ_p)\bs \G(\QQ_p)/ \cI_p$ such that $\gamma$ fixes the lattice $gL_0$ and $gL_1$, the contributing measure is
\[
\vol(\G_\gamma(\QQ_p)\bs \G_\gamma(\QQ_p)g\cI_p)=\vol(\G_\gamma(\QQ_p)\bs \G_\gamma(\QQ_p)g\cI_p g^{-1})=\vol((\G_\gamma(\QQ_p)\cap g\cI_p g^{-1})\bs g\cI_p g^{-1}).
\]
We will show that $\G_\gamma(\QQ_p)\cap g\cI_p g^{-1}\subseteq \G_\gamma(\ZZ_p)$. Hence the contributing measure is
\[
\vol(g\cI_p g^{-1})[\G_\gamma(\ZZ_p):\G_\gamma(\QQ_p)\cap g\cI_p g^{-1}].
\]

\begin{lemma}\label{lem:iwahoriindex}
We have $[\cK_p:\cI_p]=p+1$. Moreover, a full representative of $\cK_p/\cI_p$ can be chosen to be
\[
\left\{\begin{pmatrix}
    1 & 0 \\
   0 & 1 
  \end{pmatrix}\right\}\sqcup \left\{\begin{pmatrix}
    \lambda & 1 \\
     1 & 0
  \end{pmatrix}\middle| \, \lambda \in \FF_p\right\},
\]
where we use the identification $\FF_p\cong \ZZ_p/p\ZZ_p$.
\end{lemma}

\begin{insertproof}
We have a surjective group homomorphism $\cK_p=\GL_2(\ZZ_p)\twoheadrightarrow \GL_2(\FF_p)$ defined by modulo $p$, and $\cI_p$ is the preimage of the Borel subgroup $B=\B(\FF_p)$.
By Bruhat decomposition we have
\[
\GL_2(\FF_p)/B=B\sqcup\bigsqcup_{\lambda\in \FF_p}\begin{pmatrix}
    \lambda & 1 \\
     1 & 0
  \end{pmatrix}B.
\]
Hence the conclusion holds by lifting the decomposition back to $\cK_p$.
\end{insertproof}
From this lemma we know that $\vol(\cI_p)=1/(p+1)$. Hence the contributing volume for $g\cI_p$ is
\[
\frac{[\G_\gamma(\ZZ_p):\G_\gamma(\QQ_p)\cap g\cI_p g^{-1}]}{p+1}.
\]
Moreover, if $h_1,h_2,\dots,h_n$ form a set of representatives of the double coset $\G_\gamma(\QQ_p)\bs \G(\QQ_p)/\cK_p$, then by this lemma we know that for any $g\in \G(\QQ_p)$, there exists some $h_i$ or $h_i(\begin{smallmatrix}  \lambda & 1 \\  1 & 0 \end{smallmatrix})$ such that this element is equivalent to $g$ in $\G_\gamma(\QQ_p)\bs \G(\QQ_p)/\cK_p$.

In the previous theorem we have obtained a full set of representatives $h_1,\dots,h_n$ of $\G_\gamma(\QQ_p)\bs \G(\QQ_p)/\cK_p$ in each case. Hence to find a full representative of $\G_\gamma(\QQ_p)\bs \G(\QQ_p)/\cI_p$, we only need to consider the equivalence classes of the elements of $h_i$ and $h_i(\begin{smallmatrix}  \lambda & 1 \\  1 & 0 \end{smallmatrix})$ in $\G_\gamma(\QQ_p)\bs \G(\QQ_p)/\cI_p$.

\underline{\emph{Case 1:}}\ \  $p$ is split in $E$, that is, $E_p\cong \QQ_p^2$. Under this identification, $\G_\gamma$ is the diagonal torus. 
\begin{lemma}
A full representative of $\G_\gamma(\QQ_p)\bs\G(\QQ_p)/\cI_p$ can be
\[
\left\{\begin{pmatrix}
  1 & p^{-j} \\
  0 & 1 
\end{pmatrix}\middle|\, j\in \ZZ_{\geq 0}\right\}\sqcup \left\{\begin{pmatrix}
  1 & p^{-j} \\
  0 & 1 
\end{pmatrix}\begin{pmatrix}
  0 & 1 \\
  1 & 0 
\end{pmatrix}\middle|\, j\in \ZZ_{\geq 0}\right\}\sqcup\left\{\begin{pmatrix}
  0 & 1 \\
  1 & 0 
\end{pmatrix}\right\}.
\]
\end{lemma}

\begin{insertproof}
By \autoref{lem:splitk} and \autoref{lem:iwahoriindex} we only need to consider the equivalence classes of
\[
\begin{pmatrix}
  1 & p^{-j} \\
  0 & 1 
\end{pmatrix}\quad\text{and}\quad\begin{pmatrix}
  1 & p^{-j} \\
  0 & 1 
\end{pmatrix}\begin{pmatrix}
  \lambda & 1 \\
  1 & 0 
\end{pmatrix}\ (\lambda\in \FF_p)
\]
with $j\geq 0$. If two elements of such type are equivalent in $\G_\gamma(\QQ_p)\bs\G(\QQ_p)/\cI_p$, they are also equivalent in $\G_\gamma(\QQ_p)\bs\G(\QQ_p)/\cK_p$. Hence by \autoref{lem:splitk} we know that their $j$'s must equal. Hence we only need to determine whether two given elements in
\[
\left\{\begin{pmatrix}
  1 & p^{-j} \\
  0 & 1 
\end{pmatrix}\right\}\cup \left\{\begin{pmatrix}
  1 & p^{-j} \\
  0 & 1 
\end{pmatrix}\begin{pmatrix}
    \lambda & 1 \\
     1 & 0
  \end{pmatrix}\middle|\, \lambda\in \FF_p\right\}
\]
are equivalent for fixed $j$.

Observe that $g_1$ and $g_2$ are equivalent in the double coset $H\bs G/K$ if and only if there exists $h\in H$ such that $g_1^{-1}h g_2\in K$.
If $(\begin{smallmatrix} 1 & p^{-j} \\ 0  & 1  \end{smallmatrix})$ and $(\begin{smallmatrix} 1 & p^{-j} \\ 0  & 1  \end{smallmatrix})(\begin{smallmatrix} \lambda & 1 \\ 1  & 0  \end{smallmatrix})$ are equivalent, then there exist $a,d\in \QQ_p^\times$ such that
\[
\begin{pmatrix}
  ap^{-j}-dp^{-j}+a\lambda & a \\
  d & 0
\end{pmatrix}=\begin{pmatrix}
  1 & p^{-j} \\
  0 & 1 
\end{pmatrix}^{-1}
\begin{pmatrix}
  a & 0 \\
  0 & d 
\end{pmatrix}
\begin{pmatrix}
  1 & p^{-j} \\
  0 & 1 
\end{pmatrix}\begin{pmatrix}
    \lambda & 1 \\
     1 & 0
  \end{pmatrix}\in \cI_p,
\]
which is impossible. 

Now given $\lambda\neq \mu\in \FF_p$, we will determine whether
\[
\begin{pmatrix}
  1 & p^{-j} \\
  0 & 1 
\end{pmatrix}\begin{pmatrix}
    \lambda & 1 \\
     1 & 0
  \end{pmatrix}\qquad\text{and}\qquad\begin{pmatrix}
  1 & p^{-j} \\
  0 & 1 
\end{pmatrix}\begin{pmatrix}
    \mu & 1 \\
     1 & 0
  \end{pmatrix}
\]
are equivalent. It suffices to determine whether
\[
\begin{pmatrix}
  d & 0 \\
  ap^{-j}-dp^{-j}+a\lambda-d\mu & a
\end{pmatrix}=\begin{pmatrix}
    \mu & 1 \\
     1 & 0
  \end{pmatrix}^{-1}\begin{pmatrix}
  1 & p^{-j} \\
  0 & 1 
\end{pmatrix}^{-1}
\begin{pmatrix}
  a & 0 \\
  0 & d 
\end{pmatrix}
\begin{pmatrix}
  1 & p^{-j} \\
  0 & 1 
\end{pmatrix}\begin{pmatrix}
    \lambda & 1 \\
     1 & 0
  \end{pmatrix}\in \cI_p.
\]

Suppose that $j>0$. In this case, we can write $s=\mu-\lambda$ and $a=sp^j+1$, $d=1$. Then $a,d\in \ZZ_p^\times$ and
\[
ap^{-j}-dp^{-j}+a\lambda-d\mu\equiv s+\lambda-\mu\equiv 0\pmod p.
\]
Hence they are equivalent. Thus for $j>0$ there are two equivalent classes:
\[
\left\{\begin{pmatrix}
  1 & p^{-j} \\
  0 & 1 
\end{pmatrix}\right\}\quad\left\{\begin{pmatrix}
  1 & p^{-j} \\
  0 & 1 
\end{pmatrix}\begin{pmatrix}
  \lambda & 1 \\
  1 & 0 
\end{pmatrix}\middle|\,\lambda\in \FF_p\right\}.
\]

Now we suppose that $j=0$. In this case, the two elements are equivalent if and only if there exist $a,d\in \ZZ_p^\times$ such that $(\lambda+1)a-(\mu+1)d\in p\ZZ_p$. If $\lambda,\mu\neq -1$ we can always find such $a$ and $d$, while if $\lambda=-1$ or $\mu=-1$ we cannot find such $a$ and $d$. Hence for $j=0$ there are three equivalent classes:
\[
\left\{\begin{pmatrix}
  1 & p^{-j} \\
  0 & 1 
\end{pmatrix}\right\},\quad\left\{\begin{pmatrix}
  1 & p^{-j} \\
  0 & 1 
\end{pmatrix}\begin{pmatrix}
  \lambda & 1 \\
  1 & 0 
\end{pmatrix}\middle|\,\lambda\in \FF_p-\{-1\}\right\},\quad \left\{\begin{pmatrix}
  1 & p^{-j} \\
  0 & 1 
\end{pmatrix}\begin{pmatrix}
  -1 & 1 \\
  1 & 0 
\end{pmatrix}=\begin{pmatrix}
  0 & 1 \\
  1 & 0 
\end{pmatrix}\right\}.\qedhere
\]
\end{insertproof}
By the previous lemma we only need to consider the following $g$'s:

Suppose that $g=(\begin{smallmatrix}  1 & p^{-j}  \\ 0 & 1 \end{smallmatrix})$ with $j\geq 0$. Observe that
\[
\begin{pmatrix}
  1 & p^{-j} \\
  0 & 1 
\end{pmatrix}L_1=\begin{pmatrix}
  1 & p^{-j+1} \\
  0 & p 
\end{pmatrix}L_0,
\]
and 
\[
\begin{pmatrix}
  1 & p^{-j+1} \\
  0 & p 
\end{pmatrix}\sim \begin{pmatrix}
  1 & p^{-j+1} \\
  0 & 1 
\end{pmatrix}
\]
in $\G_\gamma(\QQ_p)\bs \G(\QQ_p)/\cK_p$. Hence by the computation of \autoref{thm:maximalcompact} in the split case we know that $\gamma$ fixes $gL_0$ if and only if $\gamma$ is integral and $j\leq k$, and fixes $gL_1$ if and only if $\gamma$ is integral and $j-1\leq k$. Thus $\gamma$ fixes both $gL_0$ and $gL_1$ if and only if $\gamma$ is integral and $j\leq k$.

Now we find the contributing measure. Suppose that 
\[
\begin{pmatrix}
  a & 0 \\
  0 & d 
\end{pmatrix}\in \G_\gamma(\QQ_p)\cap g\cI_p g^{-1}.
\]
Then
\[
\begin{pmatrix}
  a & p^{-j}(a-d) \\
  0 & d 
\end{pmatrix}=g^{-1}\begin{pmatrix}
  a & 0 \\
  0 & d 
\end{pmatrix}g\in \cI_p,
\]
and thus $a\equiv d\pmod {p^j}$. The subgroup
\[
\left\{\begin{pmatrix}
  a & 0 \\
  0 & d 
\end{pmatrix}\in \G(\ZZ_p)\,\middle|\,a\equiv d\pmod {p^j}\right\}
\]
has index $1$ if $j=0$ and has index $=p^j-p^{j-1}$ if $j\geq 1$. Hence the contributing measure is $1/(p+1)$ for $j=0$ and is $(p^j-p^{j-1})/(p+1)$ for $j\geq 1$.

If $g=(\begin{smallmatrix}  1 & p^{-j}  \\ 0 & 1 \end{smallmatrix})(\begin{smallmatrix}  \lambda & 1 \\ 1 & 0 \end{smallmatrix})$ with $j\geq 0$, or $\lambda\neq -1$ and $j=0$. Observe that
\[
\begin{pmatrix}
  1 & p^{-j} \\
  0 & 1 
\end{pmatrix}\begin{pmatrix}
  \lambda & 1 \\
 1 & 0 
\end{pmatrix}L_0=\begin{pmatrix}
  1 & p^{-j} \\
  0 & 1 
\end{pmatrix}L_0
\]
and
\[
\begin{pmatrix}
  1 & p^{-j} \\
  0 & 1 
\end{pmatrix}\begin{pmatrix}
  \lambda & 1 \\
 1 & 0 
\end{pmatrix}L_1=\begin{pmatrix}
  p & \lambda+p^{-j} \\
  0 & 1 
\end{pmatrix}L_0
\]
with
\[
\begin{pmatrix}
  p & \lambda+p^{-j} \\
  0 & 1 
\end{pmatrix}\sim \begin{pmatrix}
  1 & p^{-1}(\lambda+p^{-j}) \\
  0 & 1 
\end{pmatrix}
\]
in $\G_\gamma(\QQ_p)\bs \G(\QQ_p)/\cK_p$. Hence $\gamma$ fixes $gL_0$ if and only if $\gamma$ is integral and $j\leq k$, and fixes $gL_1$ if and only if $\gamma$ is integral and $j+1\leq k$ (since $\lambda\neq -1$ in the $j=0$ case). Thus $\gamma$ fixes both $gL_0$ and $gL_1$ if and only if $\gamma$ is integral and $j\leq k-1$.

Suppose that 
$
(\begin{smallmatrix}
  a & 0 \\
  0 & d 
\end{smallmatrix})\in \G_\gamma(\QQ_p)\cap g\cI_p g^{-1}
$. Then
\[
\begin{pmatrix}
  a & 0 \\
  (a-d)(p^{-j}+\lambda) & d 
\end{pmatrix}=g^{-1}\begin{pmatrix}
  a & 0 \\
  0 & d 
\end{pmatrix}g\in \cI_p,
\]
and thus $a\equiv d\pmod {p^{j+1}}$. Thus the contributing measure is $(p^{j+1}-p^j)/(p+1)$.

Finally, suppose that $g=(\begin{smallmatrix}  1 & 1  \\ 0 & 1 \end{smallmatrix})(\begin{smallmatrix}  -1 & 1 \\ 1 & 0 \end{smallmatrix})=(\begin{smallmatrix}  0 & 1  \\ 1 & 0 \end{smallmatrix})$. We have $gL_0=L_0$ and
\[
gL_1=\begin{pmatrix}
       0 & p \\
       1 & 0 
     \end{pmatrix}L_0=\begin{pmatrix}
       p & 0 \\
       0 & 1 
     \end{pmatrix}L_0
\]
with 
\[
\begin{pmatrix}
       p & 0 \\
       0 & 1 
     \end{pmatrix}\sim \begin{pmatrix}
       1 & 0 \\
       0 & 1 
     \end{pmatrix}
\]
in $\G_\gamma(\QQ_p)\bs \G(\QQ_p)/\cK_p$. Hence $\gamma$ fixes both $gL_0$ and $gL_1$ if and only if $\gamma$ is integral and $0=j\leq k$. 

Suppose that 
$
(\begin{smallmatrix}
  a & 0 \\
  0 & d 
\end{smallmatrix})\in \G_\gamma(\QQ_p)\cap g\cI_p g^{-1}
$. Then
\[
\begin{pmatrix}
 d & 0 \\
  0 & a 
\end{pmatrix}=g^{-1}\begin{pmatrix}
  a & 0 \\
  0 & d 
\end{pmatrix}g\in \cI_p,
\]
and thus $a,d\in \ZZ_p^\times$. Hence the contributing measure is $1/(p+1)$.

To sum up, we get
\[
\orb(\triv_{\cI_p};\gamma)= \frac{1}{p+1}+\sum_{j=1}^{k}\frac{p^{j}-p^{j-1}}{p+1}+
\sum_{j=0}^{k-1}\frac{p^{j+1}-p^{j}}{p+1}+\frac{1}{p+1}, 
\]
which simplifies to \eqref{eq:iwahori} since $\chi(p)=1$ in this case.

\underline{\emph{Case 2:}}\ \  $p$ is inert or ramified in $E$. In this case, $E_p$ is a field and we can identify $E$ with $\QQ_p^2$ via the basis $\{1,\Delta\}$. Recall that $\Delta$ is identified with $(\begin{smallmatrix}  0 & \alpha  \\ 1 & \beta \end{smallmatrix})$ under this identification.
\begin{lemma}
\begin{enumerate}[itemsep=0pt,parsep=0pt,topsep=0pt,leftmargin=0pt,labelsep=2.5pt,itemindent=15pt,label=\upshape{(\arabic*)}]
  \item If $p$ is inert, a full representative of $\G_\gamma(\QQ_p)\bs\G(\QQ_p)/\cI_p$ can be
\[
\left\{\begin{pmatrix}
  1 & 0 \\
  0 & p^{j} 
\end{pmatrix}\middle|\, j\in \ZZ_{\geq 0}\right\}\sqcup \left\{\begin{pmatrix}
  1 & 0 \\
  0 & p^{j}  
\end{pmatrix}\begin{pmatrix}
  0 & 1 \\
  1 & 0 
\end{pmatrix}\middle|\, j\in \ZZ_{\geq 1}\right\}.
\]
\item If $p$ is ramified, a full representative of $\G_\gamma(\QQ_p)\bs\G(\QQ_p)/\cI_p$ can be
\[
\left\{\begin{pmatrix}
  1 & 0 \\
  0 & p^{j} 
\end{pmatrix}\middle|\, j\in \ZZ_{\geq 0}\right\}\sqcup \left\{\begin{pmatrix}
  1 & 0 \\
  0 & p^{j}  
\end{pmatrix}\begin{pmatrix}
  0 & 1 \\
  1 & 0 
\end{pmatrix}\middle|\, j\in \ZZ_{\geq 1}\right\}\sqcup\left\{\begin{pmatrix}
  \lambda_0 & 1 \\
  1 & 0 
\end{pmatrix}\right\},
\]
where $\lambda_0$ is the unique solution in $\FF_p$ such that $\lambda^2+\beta\lambda-\alpha=0$ (since $p$ is ramified, the solution is unique).
\end{enumerate} 
\end{lemma}

\begin{insertproof}
By \autoref{lem:elliptick} and \autoref{lem:iwahoriindex} we only need to consider the equivalence classes of
\[
\begin{pmatrix}
  1 & 0 \\
  0 & p^{j} 
\end{pmatrix}\quad\text{and}\quad\begin{pmatrix}
  1 & 0 \\
  0 & p^{j}  
\end{pmatrix}\begin{pmatrix}
  \lambda & 1 \\
  1 & 0 
\end{pmatrix}\ (\lambda\in \FF_p)
\]
with $j\geq 0$. As in the proof of the split case, we only need to determine whether given two elements in
\[
\left\{\begin{pmatrix}
  1 & 0 \\
  0 & p^{j}  
\end{pmatrix}\right\}\cup \left\{\begin{pmatrix}
  1 & 0 \\
  0 & p^{j}  
\end{pmatrix}\begin{pmatrix}
    \lambda & 1 \\
     1 & 0
  \end{pmatrix}\middle|\, \lambda\in \FF_p\right\}
\]
are equivalent for fixed $j$.

First we consider
\[
\begin{pmatrix}
  1 & 0 \\
  0 & p^{j}  
\end{pmatrix}\quad\text{and}\quad \begin{pmatrix}
  1 & 0 \\
  0 & p^{j}  
\end{pmatrix}\begin{pmatrix}
    \lambda & 1 \\
     1 & 0
  \end{pmatrix}.
\]
We have
\[
\begin{pmatrix}
  1 & 0 \\
  0 & p^{j}  
\end{pmatrix}^{-1}(a+c\Delta) \begin{pmatrix}
  1 & 0 \\
  0 & p^{j}  
\end{pmatrix}\begin{pmatrix}
    \lambda & 1 \\
     1 & 0
  \end{pmatrix}=\begin{pmatrix}
    cp^j\alpha+a\lambda & a \\
     a+c\beta+cp^{-j}\lambda & cp^{-j}
  \end{pmatrix}.
\]
First assume that $j\geq 1$. If $\lambda\neq 0$, we can choose $c=-p^{j}$ and $a=\lambda$. In this case, all the entries in the above matrix are integers. Moreover, $cp^j\alpha+a\lambda=-p^{2j}\alpha+\lambda^2\in \ZZ_p^\times$, $cp^{-j}=-1\in \ZZ_p^\times$, and
\[
a+c\beta+cp^{-j}\lambda\equiv \lambda-\lambda\equiv 0\pmod p.
\]

Hence the above matrix belongs to $\cI_p$ and thus $(\begin{smallmatrix}  1 & 0  \\ 0 & p^j \end{smallmatrix})$ is equivalent to $(\begin{smallmatrix}  1 & 0  \\ 0 & p^j \end{smallmatrix})(\begin{smallmatrix}  \lambda & 1  \\ 1 & 0 \end{smallmatrix})$ for $\lambda\neq 0$. 

If $\lambda=0$ and the above matrix belongs to $\cI_p$, then we must have $cp^j\alpha,cp^{-j}\in \ZZ_p^\times$, which is  impossible since $\alpha\in \ZZ_p^\times$ (recall that $\Delta\in \cO_{E_p}^\times$). Thus  $(\begin{smallmatrix}  1 & 0  \\ 0 & p^j \end{smallmatrix})$ is not equivalent to $(\begin{smallmatrix}  1 & 0  \\ 0 & p^j \end{smallmatrix})(\begin{smallmatrix}  0 & 1  \\ 1 & 0 \end{smallmatrix})$.

Now we consider the case $j=0$. In this case, the matrix belongs to $\cI_p$ if and only if $c,c\alpha+a\lambda\in \ZZ_p^\times$, $a\in \ZZ_p$, and $p\mid a+c(\beta+\lambda)$.
Choose $c=1$ and $a=-c(\beta+\lambda)$. If $c\alpha+a\lambda\in \ZZ_p^\times$, then  $(\begin{smallmatrix}  1 & 0  \\ 0 & p^j \end{smallmatrix})$ is equivalent to $(\begin{smallmatrix}  1 & 0  \\ 0 & p^j \end{smallmatrix})(\begin{smallmatrix}  \lambda & 1  \\ 1 & 0 \end{smallmatrix})$ in this case. The last condition is equivalent to
\[
-(\lambda^2+\beta\lambda-\alpha)=\alpha-(\beta+\lambda)\lambda\in \ZZ_p^\times.
\]
Thus $(\begin{smallmatrix}  1 & 0  \\ 0 & p^j \end{smallmatrix})$ is equivalent to $(\begin{smallmatrix}  1 & 0  \\ 0 & p^j \end{smallmatrix})(\begin{smallmatrix}  \lambda & 1  \\ 1 & 0 \end{smallmatrix})$ unless $p$ is ramified and $\lambda=\lambda_0$. (For the inert case, the above equation has no solutions in $\FF_p$.)

If $\lambda=\lambda_0$ such that $c,c\alpha+a\lambda\in \ZZ_p^\times$, $a\in \ZZ_p$, and $a+c(\beta+\lambda)\in p\ZZ_p$, then we have
\[
c\alpha+a\lambda\equiv c\alpha-c(\beta+a\lambda)\lambda\equiv -c(\lambda^2+\beta\lambda-\alpha)\equiv 0\pmod p,
\]
which is a contradiction. Hence $(\begin{smallmatrix}  1 & 0  \\ 0 & p^j \end{smallmatrix})$ is not equivalent to $(\begin{smallmatrix}  1 & 0  \\ 0 & p^j \end{smallmatrix})(\begin{smallmatrix}  \lambda_0 & 1  \\ 1 & 0 \end{smallmatrix})$.
\end{insertproof}
By the previous lemma we only need to consider the following $g$'s:

If $g=(\begin{smallmatrix}  1 & 0  \\ 0 & p^j \end{smallmatrix})$ for $j\geq 0$. Note that
\[
gL_1=\begin{pmatrix}
       1 & 0 \\
       0 & p^{j+1} 
     \end{pmatrix}L_0,
\]
hence by the computation of \autoref{thm:maximalcompact} in the inert and ramified cases we know that $\gamma$ fixes $gL_0$ if and only if $\gamma$ is integral and $j\leq k$, and fixes $gL_1$ if and only if $\gamma$ is integral and $j+1\leq k$. Thus $\gamma$ fixes both $gL_0$ and $gL_1$ if and only if $\gamma$ is integral and $j\leq k-1$.

Suppose that 
\[
a+c\Delta=\begin{pmatrix}
  a & c\alpha \\
  c & a+c\beta 
\end{pmatrix}\in \G_\gamma(\QQ_p)\cap g\cI_p g^{-1}.
\]
Then
\[
\begin{pmatrix}
  a & cp^j\alpha \\
  cp^{-j} & a+c\beta 
\end{pmatrix}=g^{-1}(a+c\Delta)g\in \cI_p,
\]
and thus $c\in p^{j+1}\ZZ_p$. Hence the contributing measure is $p^j(p+1)/(p+1)$ if $E_p$ is inert and is $p^{j+1}/(p+1)$ if $E_p$ is ramified.

If $g=(\begin{smallmatrix}  1 & 0  \\ 0 & p^j \end{smallmatrix})(\begin{smallmatrix}  0 & 1 \\ 1 & 0 \end{smallmatrix})$ for $j\geq 1$. Note that
\[
gL_0=\begin{pmatrix}
       1 & 0 \\
       0 & p^j 
     \end{pmatrix}L_0\quad \text{and}\quad
gL_1=\begin{pmatrix}
       p & 0 \\
       0 & p^j 
     \end{pmatrix}L_0
\]
and
\[
\begin{pmatrix}
       p & 0 \\
       0 & p^j 
     \end{pmatrix}\sim p^{-1}\begin{pmatrix}
       p & 0 \\
       0 & p^j 
     \end{pmatrix}= \begin{pmatrix}
       1 & 0 \\
       0 & p^{j-1} 
     \end{pmatrix}
\]
in $\G_\gamma(\QQ_p)\bs \G(\QQ_p)/\cK_p$. Hence $\gamma$ fixes $gL_0$ if and only if $\gamma$ is integral and $j\leq k$, and fixes $gL_1$ if and only if $\gamma$ is integral and $j-1\leq k$. Thus $\gamma$ fixes both $gL_0$ and $gL_1$ if and only if $\gamma$ is integral and $j\leq k$.

Now we find the contributing measure. Suppose that $a+c\Delta\in \G_\gamma(\QQ_p)\cap g\cI_p g^{-1}$. Then
\[
\begin{pmatrix}
  a+c\beta & cp^{-j} \\
  cp^{j}\alpha & a
\end{pmatrix}=g^{-1}(a+c\Delta)g\in \cI_p.
\]
Thus $c\in p^{j}\ZZ_p$. Thus the contributing measure is $p^{j-1}(p+1)/(p+1)$ if $E_p$ is inert and is $p^{j}/(p+1)$ if $E_p$ is ramified, since $j\geq 1$.

Finally we consider the case $g=(\begin{smallmatrix}  \lambda_0 & 1 \\ 1 & 0 \end{smallmatrix})$ for $p$ ramified. Note that $gL_0=L_0$ and
\[
gL_1=\begin{pmatrix}
       p & \lambda_0 \\
       0 & 1
     \end{pmatrix}L_0.
\]
Also, in $\G_\gamma(\QQ_p)\bs \G(\QQ_p)/\cK_p$ we have
\[
\begin{pmatrix}
       p & \lambda_0 \\
       0 & 1
     \end{pmatrix}\sim p^{-1}(\Delta-(\lambda_0+\beta))\begin{pmatrix}
       p & \lambda_0 \\
       0 & 1
     \end{pmatrix}=\begin{pmatrix}
       -\beta-\lambda_0 & -p^{-1}(\lambda_0^2+\beta\lambda_0-\alpha) \\
       1 & 0 
     \end{pmatrix}.
\]
We have $v_p(-p^{-1}(\lambda_0^2+\beta\lambda_0-\alpha))=0$. Thus the matrix on the right hand side belongs to $\cK_p$. Hence this matrix is equivalent to $(\begin{smallmatrix}  1 & 0 \\ 0 & 1 \end{smallmatrix})$ in $\G_\gamma(\QQ_p)\bs \G(\QQ_p)/\cK_p$. Hence $\gamma$ fixes both $gL_0$ and $gL_1$ if and only if $\gamma$ is integral and $k\geq 0$.

Suppose that $a+c\Delta\in \G_\gamma(\QQ_p)\cap g\cI_p g^{-1}$. Then
\[
\begin{pmatrix}
  a+c(\beta+\lambda_0) & c \\
  c(-\lambda_0^2-\beta\lambda_0+\alpha) & a-c\lambda_0
\end{pmatrix}=g^{-1}(a+c\Delta)g\in \cI_p,
\]
which is equivalent to $a+c\Delta\in \cO_{E_p}^\times$ in this case. Hence the contributing measure is $1/(p+1)$.

To summarize, we obtain
\[
\orb(\triv_{\cI_p};\gamma)= \sum_{j=0}^{k-1}\frac{p^{j}(p+1)}{p+1}+
\sum_{j=1}^{k}\frac{p^{j-1}(p+1)}{p+1}
\]
if $p$ is inert and
\[
\orb(\triv_{\cI_p};\gamma)= \sum_{j=0}^{k-1}\frac{p^{j+1}}{p+1}+
\sum_{j=1}^{k}\frac{p^{j}}{p+1}+\frac{1}{p+1}
\]
if $p$ is ramified. These expressions simplify to \eqref{eq:iwahori} since $\chi(p)=-1$ if $p$ is inert and $\chi(p)=0$ if $p$ is ramified.
\end{proof}

\section{Semilocal analysis}\label{sec:poissonql}
\subsection{The Poisson summation formula on $\QQ_S$}
In this subsection we will derive the Poisson summation formula on $\QQ_S=\RR\times\QQ_{q_1}\times  \dots\times \QQ_{q_r}$. \index{qs@$\QQ_S$}

The measure on $\QQ_S$ is defined as follows: We equip $\RR$ with the Lebesgue measure and equip $\QQ_{q_i}$ with the Haar measure such that the volume of $\ZZ_{q_i}$ equals $1$. Finally, we equip $\QQ_S$ with the product measure of these measures.
The ring $\ZZ^S$ \index{zuppers@$\ZZ^S$} of $S$-integers can be embedded into $\QQ_S$ diagonally. We will establish an analogue of the Poisson summation formula for $\QQ_S$ using $\ZZ^S$ as the discrete subgroup.

\begin{proposition}\label{prop:cocompact}
$\ZZ^S$ is discrete in $\QQ_S$, and the quotient $\QQ_S/\ZZ^S$ is compact with volume $1$.
\end{proposition}
\begin{proof}
We define the norm on $\QQ_S$ by $\|(x,y_1,\dots,y_r)\|=\max\{|x|_\infty,|y_1|_{q_1},\dots,|y_r|_{q_r}\}$.
Suppose that $\alpha,\beta\in \ZZ^S$ with $\alpha\neq \beta$. We claim that $\|(\alpha-\beta,\alpha-\beta,\dots,\alpha-\beta)\|\geq 1$. Indeed, if $\|(\alpha-\beta,\alpha-\beta,\dots,\alpha-\beta)\|<1$, then $|\alpha-\beta|_\infty<1$ and $|\alpha-\beta|_{q_i}<1$ for all $i$. From the second inequality, $v_{q_i}(\alpha-\beta)\geq 1$. Since $\alpha-\beta\in \ZZ^S$ we have $v_\ell(\alpha-\beta)\geq 0$ for all $\ell\notin S$. Thus $\alpha-\beta\in \ZZ$. Since $|\alpha-\beta|_\infty<1$, we must have $\alpha-\beta=0$, which is a contradiction. Hence $\|(\alpha-\beta,\alpha-\beta,\dots,\alpha-\beta)\|\geq 1$ and thus $\ZZ^S$ is discrete in $\QQ_S$.

Now we prove the second assertion. We claim that $[0,1\ropen\times \ZZ_{q_1}\times \dots\times \ZZ_{q_r}$ is a fundamental domain of $\QQ_S/\ZZ^S$.  
For any $(x,y_1,\dots,y_r)\in \QQ_S=\RR\times\QQ_{q_1}\times\dots\times\QQ_{q_r}$, we let $t_i=\langle y_i\rangle_{q_i}\in \ZZ[1/q_i]$ be the fractional part of $y_i$. Then $y_i-t_i\in \ZZ_{q_i}$.
Since $t_i\in \ZZ[1/q_i]$ we know that $t_i\in \ZZ_{q_j}$ if $i\neq j$. Thus $y_i-t\in \ZZ_{q_i}$ for all $i$, where $t=t_1+\dots+t_r\in \ZZ^S$. Thus $y_i-t-\lfloor x-t\rfloor\in \ZZ_{q_i}$ for all $i$ and $x-t-\lfloor x-t\rfloor\in [0,1\ropen$. Let $m=t+\lfloor x-t\rfloor\in \ZZ^S$. Then
\[
(x,y)-(m,m)=(x-t-\lfloor x-t\rfloor,y_1-t-\lfloor x-t\rfloor,\dots,y_r-t-\lfloor x-t\rfloor)\in [0,1\ropen\times \ZZ_{q_1}\times \dots\times\ZZ_{q_r}.
\]
Suppose that $m\in \ZZ^S$ and $(x,y_1,\dots,y_r)\in [0,1\ropen\times \ZZ_{q_1}\times \dots\times\ZZ_{q_r}$ such that $(x+m,y_1+m,\dots,y_r+m)\in [0,1\ropen\times \ZZ_{q_1}\times \dots\times\ZZ_{q_r}$. Since $y_i,y_i+m\in \ZZ_{q_i}$ for all $i$, $m\in \ZZ_{q_i}$ for all $i$. Hence $v_{q_i}(m)\geq 0$ for all $i$. Since $m\in \ZZ^S$ we have $v_p(m)\geq 0$ for all $p\notin S$. Thus $v_\ell(m)\geq 0$ for all prime $\ell$ and hence $m\in \ZZ$. In this case, $x$ and $x+m$ cannot both lie in $[0,1\ropen$ unless $m=0$. Thus $[0,1\ropen\times \ZZ_{q_1}\times \dots\times\ZZ_{q_r}$ is a fundamental domain of $\QQ_S/\ZZ^S$. 
Since $[0,1\ropen\times \ZZ_{q_1}\times \dots\times \ZZ_{q_r}$ has volume $1$, the second assertion follows.
\end{proof}
Recall that we have $\widehat{\RR}=\RR$ and $\widehat{\QQ_p}=\QQ_p$ for any prime $p$ via the identifications 
\[
(\psi_a:x\mapsto\rme(ax))\mapsfrom a \qquad\text{and}\qquad (\psi_a:x\mapsto\rme_\ell(ax))\mapsfrom a,
\]
respectively (see \cite[Chapitre 2]{bourbaki2019spectral} for example). As shown in \cite{tate1950}, we have the following Fourier inversion formulas. If $f\in L^1(\RR)$, its Fourier transform is
\[
\widehat{f}(\xi)=\int_{\RR}f(x)\rme(-x\xi)\rmd x
\]
for $\xi\in \RR$, and if $\widehat{f}\in L^1(\RR)$, the Fourier inversion formula can be written as
\[
f(x)=\int_{\RR}\widehat{f}(\xi)\rme(x\xi)\rmd \xi
\]
for a.e. $x\in \RR$. If $f\in L^1(\QQ_p)$, its Fourier transform is
\[
\widehat{f}(\eta)=\int_{\QQ_p}f(y)\rme_p(-y\eta)\rmd y
\]
for $\eta\in \QQ_p$, and if $\widehat{f}\in L^1(\QQ_p)$, the Fourier inversion formula is
\[
f(y)=\int_{\QQ_p}\widehat{f}(\eta)\rme_p(y\eta)\rmd \eta
\]
for a.e. $y\in \QQ_p$. (See the introductory section for the definition of $\rme(x)$ and $\rme_p(y)$.)\index{ex@$\rme(x)$}\index{epy@$\rme_p(y)$}

Now we consider the semilocal space \index{qs@$\QQ_S$} $\QQ_S=\RR\times \QQ_{S_\fin}$, where \index{qsfin@$\QQ_{S_\fin}$} $\QQ_{S_\fin}=\QQ_{q_1}\times\dots\times \QQ_{q_r}$. We have $\widehat{\QQ_S}=\QQ_S$ via the identification
\[
\left(\psi_{a}:x\mapsto\rme_{S}(ax)\right)\mapsfrom a,
\]
where $\rme_S(y)=\rme(y)\prod_{i=1}^{r}\rme_{q_i}(y_i)$ if $y=(y_0,y_1,\dots,y_r)$ and the product of two elements in $\QQ_{S}$ is componentwise.
Analogously, the Fourier transform of $\varphi\in L^1(\QQ_{S})$ is
\begin{equation}\label{eq:deffourier}
\widehat{\varphi}(\xi)=\int_{\QQ_{S}}\varphi(x)\rme_S(-x\xi)\rmd x
\end{equation}
for $\xi\in \QQ_{S}$, where $x=(x_0,x_1,\dots,x_r)$ and $\xi=(\xi_0,\xi_1,\dots,\xi_r)$ with $x_i,\xi_i\in \QQ_{q_i}$ for $i\geq 1$ and $x_0,\xi_0\in \RR$.  If $\widehat{\varphi}\in L^1(\QQ_S)$, the Fourier inversion formula is
\begin{equation}\label{eq:fourierinversion}
\varphi(x)=\int_{\QQ_{S}}\widehat{\varphi}(\xi)\rme_S(x\xi)\rmd \xi
\end{equation}
for a.e. $x\in \QQ_S$.

For any $\alpha\in \ZZ^S$, we can associate it with a character $\psi_{\alpha}$ on $\QQ_S$, where $\alpha$ is considered to be $(\alpha,\dots,\alpha)\in \RR\times\QQ_{q_1}\times\dots\times \QQ_{q_r}$. For any $\beta\in \ZZ^S$, we claim that $\psi_{\alpha}(\beta)=\rme(\alpha\beta) \prod_{i=1}^{r}\rme_{q_i}(\alpha\beta)=1$. Indeed, we have 
\[
\alpha\beta-\sum_{i=1}^{r}\langle\alpha\beta\rangle_{q_i}\in \ZZ_{q_i}
\]
for all $q_i$. Since $\alpha\beta\in \ZZ^S$ and $\langle\alpha\beta \rangle_{q_i}\in \ZZ[1/q_i]\subseteq \ZZ^S$, we find that $\alpha\beta-\sum_{i=1}^{r}\langle\alpha\beta\rangle_{q_i}\in \ZZ$.
Thus 
\[
\rme(\alpha\beta)\rme_{q}(\alpha\beta)=\rme(\alpha\beta)\prod_{i=1}^{r}\rme_{q_i}(\alpha\beta) =\rme\left(\alpha\beta-\sum_{i=1}^{r} \langle\alpha\beta\rangle_{q_i}\right)=1.
\]
Hence $\psi_{\alpha}$ can be considered as a character on $\QQ_S/\ZZ^S$ and we obtain a homomorphism $\Phi:\ZZ^S\to (\QQ_S/\ZZ^S)\sphat\ $ by sending $\alpha$ to $\psi_\alpha$.

\begin{proposition}
The homomorphism $\Phi$ defined above is actually an isomorphism.
\end{proposition}
\begin{proof}
We begin by proving that $\Phi$ is injective. Suppose that $\Phi(\alpha)$ is trivial for some $\alpha$. Then $\psi_{\alpha}$ is trivial on $\QQ_S$ and hence $\alpha=0$. Thus $\Phi$ is injective.

Next we prove surjectivity. Let
\[
\cH=\{\alpha\in\QQ_S:\rme_S(\alpha\beta)=1\text{ for every }\beta\in\ZZ^S\}
\]
be the annihilator of $\ZZ^S$ in $\widehat{\QQ_S}\cong\QQ_S$. The proposition is
equivalent to proving $\cH=\ZZ^S$ diagonally.

By the construction of $\alpha\mapsto\psi_\alpha$ it is clear that we have $\ZZ^S\subseteq \cH$. Conversely, let
\[
\alpha=(\alpha_0,\alpha_1,\ldots,\alpha_r)\in \cH,
\qquad \alpha_0\in \RR,\quad\alpha_i\in\QQ_{q_i}\ (i\geq 1).
\]
Taking $\beta=1$ gives
\[
\rme(\alpha_0)\prod_i \rme_{q_i}(\alpha_i)=1.
\]
This means that
\[
\alpha_0-\sum_i\langle \alpha_i\rangle_{q_i}\in\ZZ.
\]
Since each $\langle \alpha_i\rangle_{q_i}$ lies in $\ZZ[1/q_i]\subset\ZZ^S$,
we get $\alpha_0\in\ZZ^S$. Put $\alpha=\alpha_0$ and subtract the diagonal
element $\alpha\in H$. It remains to prove that, if
\[
\eta=(0,\eta_1,\ldots,\eta_r)\in \cH,
\]
then every $\eta_i=0$.

Fix $j$. For every $i\neq j$, choose $M_i\geq 0$ such that
$q_i^{M_i}\eta_i\in\ZZ_{q_i}$, and set
\[
C_j=\prod_{i\neq j}q_i^{M_i}.
\]
For every $n\geq 0$, the element $\beta=C_jq_j^{-n}$ lies in $\ZZ^S$. Since $\eta\in \cH$,
\[
1=\prod_i \rme_{q_i}(\eta_i\beta).
\]
For $i\ne j$, the factor $\eta_i\beta$ lies in $\ZZ_{q_i}$, so $\rme_{q_i}(\eta_i\beta)=1$. Hence
\[
\mathrm e_{q_j}(\eta_j C_jq_j^{-n})=1
\quad\text{for all }n\geq 0.
\]
The kernel of $\rme_{q_j}$ is $\mathbb Z_{q_j}$, so
\[
\eta_j C_jq_j^{-n}\in\ZZ_{q_j}
\quad\text{for all }n\geq 0.
\]
Because $C_j$ is a $q_j$-adic unit, this is possible only if $\eta_j=0$.
Since $j$ was arbitrary, all $\eta_j=0$, and therefore $\alpha=\alpha$ is diagonal
with $\alpha\in\ZZ^S$. Thus $\cH=\ZZ^S$, so $\Phi$ is surjective and the proposition is proved. 
\end{proof}

This establishes the Pontryagin duality for $\ZZ^S$ and $\QQ_S/\ZZ^S$ \cite[Chapitre 2, \SSec 1, N°5]{bourbaki2019spectral}. Suppose that $\varphi$ is a continuous function on $\QQ_S/\ZZ^S$, its Fourier transform is
\[
\widehat{\varphi}(\alpha)=\int_{\QQ_S/\ZZ^S}\varphi(\xi)\rme_S(-\alpha\xi)\rmd \xi.
\]
If $\widehat{\varphi}\in L^1(\ZZ^S)$, by Fourier inversion formula \cite[Chapitre 2, \SSec 1, N°4]{bourbaki2019spectral}, there exists an absolute constant $c$ such that 
\[
\varphi(x)=c\sum_{\alpha\in \ZZ^S}\widehat{\varphi}(\alpha) \rme_S(\alpha x)
\]
for any $x\in \QQ_S$. If we choose $\varphi$ to be the characteristic function of $\QQ_S/\ZZ^S$, then $\widehat{\varphi}(\alpha)=0$ if $\alpha\neq 0$ and is the volume of $(\QQ_S)/\ZZ^S$ if $\alpha=0$. Since the volume of $\QQ_S/\ZZ^S$ is $1$ by \autoref{prop:cocompact}, the constant $c=1$. We thus obtain
\begin{equation}\label{eq:fourierexpansion}
\varphi(x)=\sum_{\alpha\in \ZZ^S}\widehat{\varphi}(\alpha) \rme_S(\alpha x).
\end{equation}

\begin{theorem}[Semilocal Poisson summation formula]\label{thm:poisson}
Suppose that $\varphi$ is a function on $\QQ_S$ such that
\begin{enumerate}[itemsep=0pt,parsep=0pt,topsep=0pt,leftmargin=0pt,labelsep=2.5pt,itemindent=15pt,label=\upshape{(\roman*)}]
  \item For any $x\in \QQ_S$,
  \[
  \sum_{\alpha\in \ZZ^S}\varphi(x+\alpha)
  \]
  converges absolutely and it defines a continuous function on $\QQ_S$.
  \item The sum
  \[
  \sum_{\xi\in \ZZ^S}\widehat{\varphi}(\xi)
  \]
  converges absolutely.
\end{enumerate} 
Then
\begin{equation}\label{eq:poissonrql}
\sum_{\alpha\in \ZZ^S}\varphi(\alpha)=\sum_{\xi\in \ZZ^S}\widehat{\varphi}(\xi).
\end{equation}
\end{theorem}

\begin{proof}
Let 
\[
f(x)=\sum_{\alpha\in \ZZ^S}\varphi(x+\alpha).
\]
then $f$ is a continuous function on $\QQ_S/\ZZ^S$. We have
\begin{align*}
   \widehat{f}(\alpha) & =\int_{\QQ_S/\ZZ^S}f(x)\rme_S(-\alpha x)\rmd x=\int_{\QQ_S/\ZZ^S}\sum_{\beta\in \ZZ^S}\varphi(x+\beta)\rme_S(-\alpha x)\rmd x\\
      & =\int_{\QQ_S/\ZZ^S}\sum_{\beta\in \ZZ^S}\varphi(x+\beta)\rme_S(-\alpha (x+\beta))\rmd x=\int_{\QQ_S}\varphi(x)\rme_S(-\alpha x)\rmd x=\widehat{\varphi}(\alpha).
\end{align*}
Since 
\[
\sum_{\alpha\in \ZZ^S}\widehat{\varphi}(\alpha)=\sum_{\alpha\in \ZZ^S}\widehat{f}(\alpha),
\]
which converges absolutely, by Fourier inversion formula we have
\[
\sum_{\alpha\in \ZZ^S}\varphi(x+\alpha)=\sum_{\alpha\in \ZZ^S}\widehat{\varphi}(\alpha)\rme_S(\alpha x).
\]
By letting $x=0$ we get the desired formula.
\end{proof}

\begin{corollary}\label{cor:poissonrql2}
Suppose that $\varphi$ is a function on $\QQ_{S}$, $a\in \ZZ$, and $k\in \ZZ_{>0}$ such that $\gcd(k,S)=1$, and
\begin{enumerate}[itemsep=0pt,parsep=0pt,topsep=0pt,leftmargin=0pt,labelsep=2.5pt,itemindent=15pt,label=\upshape{(\roman*)}]
  \item For any $x\in \QQ_{S}$,
  \[
  \sum_{\alpha\in \ZZ^S}\varphi(x+k\alpha)
  \]
  converges absolutely and it defines a continuous function on $\QQ_S$.
  \item The sum
  \[
  \sum_{\xi\in \ZZ^S}\widehat{\varphi}\left(\frac{\xi}{k}\right)
  \]
  converges absolutely.
\end{enumerate} 
Then
\begin{equation}\label{eq:poissonrql2}
\sum_{\substack{\alpha\in \ZZ^S\\ \alpha-a\in k\ZZ^S}}\varphi(\alpha)=\frac{1}{k}\sum_{\xi\in \ZZ^S}\rme_S\legendresymbol{a\xi}{k} \widehat{\varphi}\left(\frac{\xi}{k}\right).
\end{equation}
\end{corollary}

\begin{proof}
We have
\[
\sum_{\substack{\alpha\in \ZZ^S\\ \alpha-a\in k\ZZ^S}}\varphi(\alpha)=\sum_{\beta\in \ZZ^S}\varphi(a+k\beta).
\]
Let $\psi(x)=\varphi(a+kx)$. We first compute the Fourier transform of $\psi$. By definition
\[
\widehat{\psi}(\xi)=\int_{\QQ_S}\varphi(a+kx)\rme_S(-x\xi)\rmd x.
\]
Making change of variable $x\mapsto x/k-a$, we have $\rmd(x/k-a)=\rmd x/(|k|_\infty\prod_{i=1}^{r}|k|_{q_i})=\rmd x/k$. Hence we obtain
\[
   \widehat{\psi}(\xi)=\frac{1}{k}\int_{\QQ_S}\varphi(x) \expcharacter[S]{-\frac{x\xi}{k}+\frac{a\xi}{k}} \rmd x=\frac{1}{k}\rme_S\legendresymbol{a\xi}{k}\widehat{\varphi}\left(\frac{\xi}{k}\right).
\]
The corollary now follows from the Poisson summation formula \eqref{eq:poissonrql} for $\psi$.
\end{proof}

\subsection{Application to the function $\Psi$ in \autoref{thm:ellipticpoisson}}
Let $N=\pm nq^\nu$, $\Psi(x,y)$ can be simplified to
\begin{equation}\label{eq:defpsi}
\begin{split}
   \Psi(x,y) &=\theta_\infty^\pm\legendresymbol{x}{2\sqrt{|N|}}\prod_{i=1}^{r}\theta_{q_i}(y_i,N)\left[F\legendresymbol{kf^2}{|x^2-4N|_\infty^\vartheta|y^2-4N|_q'^\vartheta} \right.\\
    &+\left.\frac{kf^2}{\sqrt{|x^2-4N|_\infty|y^2-4N|_q'}}V\legendresymbol{kf^2}{|x^2-4N|_\infty ^{1-\vartheta}|y^2-4N|_q'^{1-\vartheta}}\right].
\end{split}
\end{equation}
The main problem is that this function is not smooth for $y$ near $2\sqrt{N}$ in nonarchimedean places. However, since $F$ and $V$ are of rapid decay we could expect that the Poisson summation formula still holds. 
\begin{proposition}\label{prop:estimatefv}
For any $a\in \ZZ_{\geq 0}$, $F^{(a)}(x)$ and $V_{\iota,\epsilon}^{(a)}(x)$ have rapid decay as $x\to +\infty$.
\end{proposition}
\begin{proof}
We first prove that $F$ has rapid decay. If $a=0$, the conclusion holds by \eqref{eq:rapidf}. Now assume that $a\geq 1$. Since
\[
F'(x)=-\frac{1}{2K_0(2)}\frac{1}{x}\rme^{-x-1/x},
\]
by induction we know that $F^{(a)}(x)$ is a rational function of $x$ times $\rme^{-x-1/x}$. Hence $F^{(a)}(x)$ has rapid decay when $x\to +\infty$. 

Now we consider the functions $V_{\iota,\epsilon}$. We have
\[
V_{\iota,\epsilon}(x)=\frac{\uppi^{1/2}}{\dpii}\int_{(\sigma)}\widetilde{F}(s)\prod_{i=1}^{r} \frac{1-\epsilon_iq_i^{s-1}} {1-\epsilon_iq_i^{-s}}\frac{\Gamma(\frac{\iota+s}{2})}{\Gamma(\frac{\iota+1-s}{2})}(\uppi x)^{-s}\rmd s
\]
for any $\sigma\geq 1$. Moreover, for any $a\in \ZZ_{\geq 0}$ we can take the derivative in the integrand and obtain
\[
V_{\iota,\epsilon}^{(a)}(x)=\frac{\uppi^{1/2}}{\dpii}\int_{(\sigma)}\widetilde{F}(s)\prod_{i=1}^{r} \frac{1-\epsilon_iq_i^{s-1}} {1-\epsilon_iq_i^{-s}}\frac{\Gamma(\frac{\iota+s}{2})}{\Gamma(\frac{\iota+1-s}{2})}(-1)^as(s+1)\cdots(s+a-1)\uppi^{-s} x^{-s-a}\rmd s.
\]

For any $A>1$, consider the region $1\leq \sigma\leq A$ and $|t|\gg 1$ (where $s=\sigma+\rmi t$). Since the region $1\leq \sigma\leq A$ and $|t|\ll 1$ is bounded, we have
\[
V_{\iota,\epsilon}^{(a)}(x)\ll_{A,a} \int_{(A),|t|\gg 1}|\widetilde{F}(s)|\prod_{i=1}^{r} \frac{|1-\epsilon_iq_i^{s-1}|} {|1-\epsilon_iq_i^{-s}|}\frac{|\Gamma(\frac{\iota+s}{2})|}{|\Gamma(\frac{\iota+1-s}{2})|}|s||s+1|\cdots|s+a-1|x^{-\sigma-a}\rmd s.
\]

By \autoref{prop:estimatemf} we have
\[
\widetilde{F}(s)\ll_A |s|^{|\sigma|-1}\rme^{-\uppi|t|/2},
\]
and by \autoref{prop:stirling} (2) we have
\[
\frac{|\Gamma(\frac{\iota+s}{2})|}{|\Gamma(\frac{\iota+1-s}{2})|}\asymp_A \frac{|t/2|^{(\iota+\sigma-1)/2}\rme^{-\uppi|t|/4}}{|t/2|^{(\iota-\sigma)/2}\rme^{-\uppi|t|/4}}=\left|\frac{t}{2}\right|^{\sigma-1/2}.
\]
Also, we have
\[
\frac{|1-\epsilon_i q_i^{s-1}|}{|1-\epsilon_i q_i^{-s}|}\leq \frac{1+q_i^{A-1}}{1-q_i^{-A}}\ll_A q_i^A\ll_{A,q_i} 1
\]
for all $\Re s=A$ and $i=1,\dots,r$. Therefore
\begin{align*}
   V_{\iota,\epsilon}^{(a)}(x) & \ll_{A,a,q}\int_{(A),|t|\gg 1}\left|\frac{t}{2}\right|^{A-1/2}|t|^ax^{-A-a}|t|^{A-1}\rme^{-\uppi|t|/2}\rmd s \\
     & =x^{-A-a}\int_{(A),|t|\gg 1}\left|\frac{t}{2}\right|^{A-1/2}|t|^a\rme^{-\uppi|t|/2}\rmd s\ll_{A,a,q} x^{-A-a}. \qedhere
\end{align*}
\end{proof}
Now we consider the following two functions for $t\in \lopen 0,+\infty\ropen$:
\[
F(kf^2t)\qquad\text{and}\qquad kft^{1/(2(1-\vartheta))} V(kf^2t).
\]
Since $F^{(a)}$ and $V^{(a)}$ have  rapid decay as $x\to +\infty$ for each $a\in \ZZ_{\geq 0}$ by \autoref{prop:estimatefv}, these two functions also have rapid decay. Clearly, the conditions in \autoref{cor:poissonrql2} hold for linear combinations of $\varphi_1,\varphi_2$ if they hold for $\varphi_1$ and $\varphi_2$. Hence we only need to consider a more general case
\[
\Omega(x,y)=\theta_\infty^\pm\legendresymbol{x}{2\sqrt{|N|}}\prod_{i=1}^{r}\theta_{q_i}(y_i,N)\Phi\legendresymbol{1}{|x^2-4N|_\infty^\vartheta|y^2-4N|_q'^\vartheta},
\]
where $\vartheta>0$ and $\Phi$ is a smooth function on $\lopen 0,+\infty\ropen$ such that $\Phi^{(a)}(x)$ has rapid decay as $x\to +\infty$ for each $a\in \ZZ_{\geq 0}$.

\begin{proposition}\label{prop:poissonwork}
$\Omega$ satisfies the conditions in \autoref{cor:poissonrql2}.
\end{proposition}

\begin{proof}
(i) follows from \autoref{lem:nufinite}. Now we prove (ii).

We have the following lemma proved in \cite[Proposition 4.1]{altug2015}:
\begin{lemma}\label{lem:smooth}
Let $\vartheta>0$ and $\Phi(x)\in C^\infty(\lopen 0,+\infty\ropen)$ such that for any $a\in \ZZ_{\geq 0}$, $\Phi^{(a)}(x)$ has rapid decay as $x\to +\infty$. Then $\theta_\infty^\pm(x)\Phi(|1-x^2|^{-\vartheta})$ is smooth.
\end{lemma}

Now we consider the nonarchimedean places. For any $p\in S$, we can express $\theta_{p}(y,N)$ as
\[
\theta_{p}(y,N)=\theta_{p}^{0}(y,N)+\theta_{p}^{1}(y,N)+\theta_{p}^{-1}(y,N),
\]
where $\theta_{p}^{0}(y,N)$ is zero near $y=\pm 2\sqrt{N}$ (thus is smooth), $\theta_{p}^1(y,N)$ is only supported near $2\sqrt{N}$ and $\theta_{p}^{-1}(y,N)$ is only supported near $-2\sqrt{N}$. If $N$ has no square roots in $\QQ_{p}$, we define $\theta_{p}^{\pm 1}(y,N)$ to be $0$. Correspondingly we write
\[
\Omega(x,y)=\sum_{v\in \{0,\pm 1\}^r}\Omega^v(x,y),
\]
where
\[
\Omega^{(v_1,\dots,v_r)}(x,y)=\theta_\infty^\pm\legendresymbol{x}{2\sqrt{|N|}} \prod_{i=1}^{r}\theta_{q_i}^{v_i}(y_i,N)\Phi\legendresymbol{1}{|x^2-4N|_\infty^\vartheta|y^2-4N|_q'^\vartheta}
\]
for $i=1,\dots,r$ and $v_i\in \{0,\pm 1\}$. Clearly it suffices to verify (ii) for $\Omega^{(v_1,\dots,v_r)}$. 
By symmetry it suffices to consider $\Omega^{v}(x,y)$ with $v=(v_1,\dots,v_r)$ such that $v_i=0$ or $1$. We first make the support of $\theta_{q_i}^1(y_i,N)$ more precise. We choose $\theta_{q_i}^1(y_i,N)$ that is supported in $2\sqrt{N}+q_i^{m_i}\ZZ_{q_i}$, where $m_i$ is chosen so that $m_i>v_{q_i}(4\sqrt{N})$ and \autoref{cor:shalika} holds in the neighborhood $2\sqrt{N}+q_i^{m_i}\ZZ_{q_i}$ of $2\sqrt{N}$. 
\begin{lemma}\label{lem:rapiddecay}
We have
\[
\widehat{\Omega^v}(\xi,\eta)\ll (1+|\xi|_\infty)^{-2}\prod_{i=1}^{r}(1+|\eta_i|_{q_i})^{-3}, 
\]
where the implied constant depends only on $\Omega^v(x,y)$.
\end{lemma}
\begin{insertproof}
For simplicity, we assume that $v_i=1$ for $i=1,\dots,k$ and $v_i=0$ for $i=k+1,\dots,r$.
The Fourier transform of $\Omega^v$ is
\[
   \widehat{\Omega^v}(\xi,\eta)=\int_{\RR}\int_{ \QQ_{S_\fin}}\theta_\infty^\pm\legendresymbol{x}{2\sqrt{|N|}}\prod_{i=1}^{r}\theta_{q_i}^{v_i}(y_i,N) \Phi\legendresymbol{1}{|x^2-4N|_\infty^\vartheta|y^2-4N|_q'^\vartheta}\rme(-x\xi) \rme_{q}(-y\eta)\rmd x\rmd y,
\]
where $x\in\RR$, $y\in\QQ_{S_\fin}$.

\underline{\emph{Case 1:}}\ \ $|\xi|_\infty\gg 1$.

We first consider the integral over $x$, that is
\[
\int_{\RR}\theta_\infty^\pm\legendresymbol{x}{2\sqrt{|N|}}\Phi\legendresymbol{1}{|x^2-4N|_\infty^\vartheta|y^2-4N|_q'^\vartheta} \rme^{-\dpii x\xi}\rmd x.
\]
Since the integrand is smooth (by \autoref{lem:smooth}) and compactly supported, we can use integration by parts twice. Thus the above equals
\[
-\frac{1}{4\uppi^2\xi^2}\int_{\RR}\frac{\rmd^2}{\rmd x^2}\left(\theta_\infty^\pm\legendresymbol{x}{2\sqrt{|N|}}\Phi\legendresymbol{1}{|x^2-4N|_\infty^\vartheta|y^2-4N|_q'^\vartheta}\right)\rme^{-\dpii x\xi}\rmd x.
\]

We have
\begin{equation}\label{eq:estimatederivativey}
\frac{\rmd^2}{\rmd x^2}\left(\theta_\infty^\pm\legendresymbol{x}{2\sqrt{|N|}}\Phi\legendresymbol{1}{|x^2-4N|_\infty^\vartheta|y^2-4N|_q'^\vartheta}\right)\ll_M |y^2-4N|_q'^{-M}
\end{equation}
for any $M>0$. Indeed, for $x\neq \pm 2\sqrt{|N|}$ we have
\[
\begin{split}
    & \frac{\rmd^2}{\rmd x^2}\left(\theta_\infty^\pm\legendresymbol{x}{2\sqrt{|N|}}\Phi\legendresymbol{1}{|x^2-4N|_\infty^\vartheta|y^2-4N|_q'^\vartheta}\right) =\varphi_1(x)\Phi\legendresymbol{1}{|x^2-4N|_\infty^\vartheta|y^2-4N|_q'^\vartheta} \\
     +&\frac{\varphi_2(x)}{|y^2-4N|_q'^\vartheta}\Phi'\legendresymbol{1}{|x^2-4N|_\infty^\vartheta|y^2-4N|_q'^\vartheta}+ \frac{\varphi_3(x)}{|y^2-4N|_q'^{2\vartheta}}\Phi''\legendresymbol{1}{|x^2-4N|_\infty^\vartheta|y^2-4N|_q'^\vartheta},
\end{split}
\]
where $\varphi_i(x)$ are compactly supported functions such that $\varphi_i(x)$ has polynomial growth as $x\to \pm 2\sqrt{|N|}$. Since $\Phi,\Phi'$ and $\Phi''$ all have rapid decay and the range of $x$ is compact, we conclude that \eqref{eq:estimatederivativey} holds.

Next we consider the nonarchimedean integral
\begin{equation}\label{eq:nonarchimedeansingular}
\int_{\QQ_{S_\fin}}\prod_{i=1}^{r}\theta_{q_i}^{v_i}(y_i,N)\frac{\rmd^2}{\rmd x^2}\left(\theta_\infty^\pm\legendresymbol{x}{2\sqrt{|N|}} \Phi\legendresymbol{1}{|x^2-4N|_\infty^\vartheta|y^2-4N|_q'^\vartheta}\right) \rme_{q}(-y\eta)\rmd y.
\end{equation} 
By making the change of variable $y_i\mapsto y_i+2\sqrt{N}$ for $i=1,\dots,k$, \eqref{eq:nonarchimedeansingular} becomes
\begin{align*}
    & \prod_{i=1}^{k}\rme_{q_i}(-2\sqrt{N}\eta_i)\int_{\QQ_{S_\fin}}\prod_{i=1}^{k} \theta_{q_i}^{1}(y_i,+2\sqrt{N},N) \prod_{i=k+1}^{r}\theta_{q_i}^{0}(y_i,N) \\
  \times   & \frac{\rmd^2}{\rmd x^2}\left[\theta_\infty^\pm\legendresymbol{x}{2\sqrt{|N|}} \Phi\legendresymbol{1}{|x^2-4N|_\infty^\vartheta\prod_{i=1}^{k}|y_i^2+4\sqrt{N}y_i|_{q_i}'^\vartheta \prod_{i=k+1}^{r}|y_i^2-4N|_{q_i}'^\vartheta}\right] \rme_{q}(-y\eta)\rmd y.
\end{align*}
$\prod_{i=1}^{k}\rme_{q_i}(-2\sqrt{N}\eta)$ times the above equation can be written as
\begin{align*}
I(\eta_1,\dots,\eta_r)=&\int_{\QQ_{q_{k+1}}}\dots\int_{\QQ_{q_r}} \sum_{u_1=m_1}^{+\infty}\cdots\sum_{u_k=m_k}^{+\infty}\int_{q_1^{u_1}\ZZ_{q_1}^\times}\dots \int_{q_k^{u_k}\ZZ_{q_k}^\times}\prod_{i=1}^{k}\theta_{q_i}^{1}(y_i+2\sqrt{N},N)\prod_{i=k+1}^{r} \theta_{q_i}^{0}(y_i,N) \\
\times&\frac{\rmd^2}{\rmd x^2}\left[\theta_\infty^\pm\legendresymbol{x}{2\sqrt{|N|}} \Phi\legendresymbol{1}{|x^2-4N|_\infty^\vartheta\prod_{i=1}^{k}|y_i^2+4\sqrt{N}y_i|_{q_i}'^\vartheta \prod_{i=k+1}^{r}|y_i^2-4N|_{q_i}'^\vartheta}\right] \\
\times&\prod_{i=1}^{r}\rme_{q_i}(-y_i\eta_i)\rmd y_1\cdots \rmd y_r.
\end{align*}
Now we define
\begin{align*}
I_{u_1,\dots,u_k}(\eta_1,\dots,\eta_r)=&\int_{\QQ_{q_{k+1}}}\dots\int_{\QQ_{q_r}} \int_{q_1^{u_1}\ZZ_{q_1}^\times}\dots \int_{q_k^{u_k}\ZZ_{q_k}^\times}\prod_{i=1}^{k}\theta_{q_i}^{1}(y_i+2\sqrt{N},N)\prod_{i=k+1}^{r} \theta_{q_i}^{0}(y_i,N) \\
\times&\frac{\rmd^2}{\rmd x^2}\left[\theta_\infty^\pm\legendresymbol{x}{2\sqrt{|N|}} \Phi\legendresymbol{1}{|x^2-4N|_\infty^\vartheta\prod_{i=1}^{k}|y_i^2+4\sqrt{N}y_i|_{q_i}'^\vartheta \prod_{i=k+1}^{r}|y_i^2-4N|_{q_i}'^\vartheta}\right] \\
\times&\prod_{i=1}^{r}\rme_{q_i}(-y_i\eta_i)\rmd y_1\cdots \rmd y_r
\end{align*}
so that 
\[
I(\eta_1,\dots,\eta_r)=\sum_{u_1=m_1}^{+\infty}\cdots\sum_{u_k=m_k}^{+\infty}
I_{u_1,\dots,u_k}(\eta_1,\dots,\eta_r).
\]

Consider the integrals
\[
I_{u_i}(\eta_i)=\int_{q_i^{u_i}\ZZ_{q_i}^\times}F_i(y_i)\rme_{q_i}(-y_i\eta_i)\rmd y_i
\]
for $i=1,\dots,k$ and
\[
I(\eta_i)=\int_{\QQ_{q_i}}F_i(y_i)\rme_{q_i}(-y_i\eta_i)\rmd y_i
\]
for $i=k+1,\dots,r$, where
\[
F_i(y_i)=\theta_{q_i}^{1}(y_i+2\sqrt{N},N)\frac{\rmd^2}{\rmd x^2}\left[\theta_\infty^\pm\legendresymbol{x}{2\sqrt{|N|}}\Phi\legendresymbol{1}{|x^2-4N|_\infty^\vartheta\prod_{i=1}^{k}|y_i^2+4\sqrt{N}y_i|_{q_i}'^\vartheta \prod_{i=k+1}^{r}|y_i^2-4N|_{q_i}'^\vartheta}\right]
\]
for $i=1,\dots,k$ and
\[
F_i(y_i)=\theta_{q_i}^{0}(y_i,N)\frac{\rmd^2}{\rmd x^2}\left[\theta_\infty^\pm\legendresymbol{x}{2\sqrt{|N|}}\Phi\legendresymbol{1}{|x^2-4N|_\infty^\vartheta\prod_{i=1}^{k}|y_i^2+4\sqrt{N}y_i|_{q_i}'^\vartheta \prod_{i=k+1}^{r}|y_i^2-4N|_{q_i}'^\vartheta}\right]
\]
for $i=k+1,\dots,r$.
We will prove that this integral is $0$ if $|\eta_i|_{q_i}$ is sufficiently large, and this quantity can be computed explicitly. 

We first consider the case that $i\in\{1,\dots,k\}$. For $v_{q_i}(y_i)=u_i\geq m_i$, by the assumption of $\theta_{q_i}^{v_i}(y_i)$ and \eqref{eq:shalikalocal} we have
\begin{equation}\label{eq:shalikalocal2}
\theta_{q_i}^1(y_i+2\sqrt{N},N)=\lambda_1\left(1-\frac{\chi(q_i)}{q_i}\right)^{-1}\sqrt{|y_i^2+4\sqrt{N}y_i|_{q_i}'}\frac{1-\chi(q_i)}{1-q_i}+\lambda_2,
\end{equation}
and
\[
\chi(q_i)=\legendresymbol{(y_i^2+4\sqrt{N}y_i)|y_i^2+4\sqrt{N}y_i|_{q_i}'}{q_i}.
\]

Since $v_{q_i}(y_i)>v_{q_i}(4\sqrt{N})$, we have
\[
v_{q_i}(y_i^2+4\sqrt{N}y_i)=v_{q_i}(y_i)+v_{q_i}(y_i+4\sqrt{N})=v_{q_i}(y_i)+v_{q_i}(4\sqrt{N}).
\]
Suppose that $a\in 1+q_i^2\ZZ_{q_i}$. Since $v_{q_i}((a+1)y_i+4\sqrt{N})=v_{q_i}(4\sqrt{N})=v_{q_i}(y_i+4\sqrt{N})$ and $a-1\in q_i^2\ZZ_{q_i}$, we have
\[
\frac{(ay_i)^2+4\sqrt{N}(ay_i)}{y_i^2+4\sqrt{N}y_i}-1=\frac{(a-1)((a+1)y_i+4\sqrt{N})}{y_i+4\sqrt{N}}\in q_i^2\ZZ_{q_i}.
\]
Therefore $|y_i^2+4\sqrt{N}y_i|_{q_i}'=|(ay_i)^2+4\sqrt{N}(ay_i)|_{q_i}'$ (see \autoref{rem:absoluteprime}). Moreover we get
\[
\legendresymbol{(y_i^2+4\sqrt{N}y_i)|y_i^2+4\sqrt{N}y_i|_{q_i}'}{q_i}= \legendresymbol{((ay_i)^2+4\sqrt{N}(ay_i))|(ay_i)^2+4\sqrt{N}(ay_i)|_{q_i}'}{q_i}.
\]
Hence $\theta_{q_i}(y_i+2\sqrt{N},N)=\theta_{q_i}(ay_i+2\sqrt{N},N)$.
Therefore $F_i(y_i)$ is invariant under the map $y_i\mapsto ay_i$, where $a\in 1+q_i^2\ZZ_{q_i}$.

Suppose that $b\in q_i^{u_i+2}\ZZ_{q_i}$. Then $y_i+b=(1+by_i^{-1})y_i$ with $1+by_i^{-1}\in 1+q_i^2\ZZ_{q_i}$ (since $y_i\in q_i^{u_i}\ZZ_{q_i}^\times$). Hence $F_i(y_i)$ invariant under the map $y_i\mapsto y_i+b$. Making such change of variable in $I_{u_i}(\eta_i)$, we obtain
\[
   I_{u_i}(\eta_i)=\int_{q_i^{u_i}\ZZ_{q_i}^\times}F_i(y_i)\rme_{q_i}(-(y_i+b)\eta_i)\rmd y_i =I_{u_i}(\eta_i)\rme_{q_i}(-b\eta_i).
\]
From this we know that $I_{u_i}(\eta_i)=0$ unless $\rme_{q_i}(-b\eta_i)=1$ for all $b\in q_i^{u_i+2}\ZZ_{q_i}$. If $v_{q_i}(\eta_i)<-u_i-2$, then $b\eta_i\notin \ZZ_{q_i}$ for $b=q_i^{u_i+2}$. Thus $\rme_{q_i}(-b\eta_i)\neq 1$. Hence $I_{u_i}(\eta_i)=0$. 

For $i=k+1,\dots,n$, $\theta_{q_i}^0(y_i,N)$ is smooth and compactly supported. Hence there exists $L_i\geq 2$ (independent of $u_i$) such that $F_i(y_i)$ is invariant under $y_i\mapsto y_i+b$ for any $b\in q_i^{L_i}\ZZ_{q_i}$. Hence $I_{u_i}(\eta_i)=0$ for $v_{q_i}(\eta_i)<-L_i$.

We have thus proved that $I_{u_1,\dots,u_k}(\eta_1,\dots,\eta_r)=0$ unless $v_{q_i}(\eta_i)\geq -u_i-L_i$ for $i=1,\dots,k$ and $v_{q_i}(\eta_i)\geq -L_i$ for $i=k+1,\dots,r$. In this case, we have the trivial estimate $\theta_{q_i}^{v_i}(y_i,N)\ll 1$ by \autoref{lem:ctscpt} (uniform in $u_i$) and hence 
\begin{align*}
I_{u_1,\dots,u_k}(\eta_1,\dots,\eta_r)\ll_M & \int_{K_{q_{k+1}}}\dots\int_{K_{q_r}}\int_{q_1^{u_1}\ZZ_{q_1}^\times}\dots \int_{q_k^{u_k}\ZZ_{q_k}^\times}\prod_{i=1}^{k}|y_i^2+4\sqrt{N}y_i|_{q_i}'^{M\vartheta} \prod_{i=k+1}^{r}|y_i^2-4N|_{q_i}'^{M\vartheta}\rmd y\\
\ll_M &\prod_{i=1}^{k}q_i^{u_i} q_i^{-u_i M\vartheta}\ll \prod_{i=1}^{k} q_i^{-3u_i}
\end{align*}
for appropriate $M$ (since $|\cdot|_\ell'\asymp_\ell |\cdot|_\ell$), where $K_i$ denotes the support of $\theta_i^0(y_i,N)$ for $i=k+1,\dots,r$. 

Suppose that $v_{q_i}(\eta_i)=-w_i$. Then $|\eta_i|_{q_i}=q_i^{w_i}$. Hence $I_{u_1,\dots,u_r}(\eta_1,\dots,\eta_r)=0$ if $u_i<-w_i-2$ for some $i\in \{1,\dots,k\}$. Hence
\begin{align*}
    & \sum_{u_1=m_1}^{+\infty}\cdots\sum_{u_k=m_k}^{+\infty}I_{u_1,\dots,u_k}(\eta_1,\dots,\eta_r)\ll  \left(\sum_{u_1=\max\{-w_1-2,m_1\}}^{+\infty}q_1^{-3u_1}\right)\cdots \left(\sum_{u_k=\max\{-w_k-2,m_k\}}^{+\infty}q_k^{-3u_k}\right) \\
      &\ll q_1^{-3\max\{-w_1-2,m_1\}}\cdots q_r^{-3\max\{-w_k-2,m_k\}}\ll (1+|\eta_1|_{q_1})^{-3}\cdots (1+|\eta_k|_{q_k})^{-3}.
\end{align*}
Therefore
\[
\eqref{eq:nonarchimedeansingular}
   \ll  (1+|\eta_1|_{q_1})^{-3}\cdots (1+|\eta_k|_{q_k})^{-3}
\]
if $|\eta_i|_{q_i}\leq q_i^{L_i}$ for all $i=k+1,\dots,r$.
Since the integral is nonzero only if $|\eta_i|_{q_i}\leq q_i^{L_i}$ for all $i=k+1,\dots,n$, we conclude that for all $\eta\in \QQ_{S_\fin}$,
\[
 \eqref{eq:nonarchimedeansingular}
   \ll   (1+|\eta_1|_{q_1})^{-3}\cdots (1+|\eta_r|_{q_r})^{-3}.
\]

Finally,
\begin{align*}
   \widehat{\Omega^v}(\xi,\eta) &=-\frac{1}{4\uppi^2\xi^2}\int_{\RR}\int_{\QQ_{S_\fin}}\prod_{i=1}^{r}\theta_{q_i}^{v_i}(y_i,N)\\
   &\times\frac{\rmd^2}{\rmd x^2}\left(\theta_\infty^\pm\legendresymbol{x}{2\sqrt{|N|}}\Phi\legendresymbol{1}{|x^2-4N|_\infty^\vartheta|y^2-4N|_q'^\vartheta}\right) \rme_{q}(-y\eta)\rmd y\rme(-x\xi)\rmd x \\
    &\ll |\xi|_\infty^{-2}\int_{K}(1+|\eta_1|_{q_1})^{-3}\cdots (1+|\eta_r|_{q_r})^{-3}\rmd x\ll |\xi|_\infty^{-2}(1+|\eta_1|_{q_1})^{-3}\cdots (1+|\eta_r|_{q_r})^{-3},
\end{align*}
where $K$ denotes the support of $\theta_\infty^\pm(x/2\sqrt{|N|})$.

\underline{\emph{Case 2:}}\ \ $|\xi|_\infty\ll 1$.

In this case, we consider directly the nonarchimedean integral
\[
\int_{\QQ_{S_\fin}}\prod_{i=1}^{r}\theta_{q_i}^{v_i}(y_i,N)\theta_\infty^\pm \legendresymbol{x}{2\sqrt{|N|}}\Phi\legendresymbol{1}{|x^2-4N|_\infty^\vartheta|y^2-4N|_q'^\vartheta} \rme_{q}(-y\eta)\rmd y.
\]
An analogous argument shows that the integral is $\ll (1+|\eta_1|_{q_1})^{-3}\cdots (1+|\eta_r|_{q_r})^{-3}$.
Finally 
\begin{align*}
   \widehat{\Omega^v}(\xi,\eta) =&\int_{\RR}\int_{\QQ_{S_\fin}}\prod_{i=1}^{r}\theta_{q_i}^{v_i}(y_i,N)\theta_\infty^\pm\legendresymbol{x}{2\sqrt{|N|}} \Phi\legendresymbol{1}{|x^2-4N|_\infty^\vartheta|y^2-4N|_q'^\vartheta}\rme_{q}(-y\eta)\rmd y\rme(-x\xi)\rmd x \\
    \ll & \int_{K}(1+|\eta_1|_{q_1})^{-3}\cdots (1+|\eta_r|_{q_r})^{-3}\rmd x\ll(1+|\eta_1|_{q_1})^{-3}\cdots (1+|\eta_r|_{q_r})^{-3}.
\end{align*}
The lemma is now proved by combining these cases together.
\end{insertproof}
\begin{remark}
By putting in more effort, one can prove that
\[
\widehat{\Omega^v}(\xi,\eta)\ll_{M_0,M_1,\dots,M_r} (1+|\xi|_\infty)^{-M_0}(1+|\eta_1|_{q_1})^{-M_1}\cdots (1+|\eta_r|_{q_r})^{-M_r}
\]
for any $M_0,M_1,\dots,M_r>0$, but we do not need this in this paper.
\end{remark}
Finally we prove that
\[
\sum_{\xi\in \ZZ^S}\widehat{\Omega^v}\left(\frac{\xi}{k},\frac{\xi}{k}\right)
\]
converges absolutely. We have
\begin{align*}
   &\sum_{\xi\in \ZZ^S}\left|\widehat{\Omega^v}\left(\frac{\xi}{k},\frac{\xi}{k}\right)\right| \ll  \sum_{\xi\in \ZZ^S}\left(1+\left|\frac{\xi}{k}\right|_\infty\right)^{-2}\prod_{i=1}^{r} (1+|\xi|_{q_i})^{-3} \\
   \ll&
   \sum_{I\subseteq \{1,\dots,r\}}\sum_{\substack{c_i\in \ZZ_{<0}\ i\in I\\ c_j\in \ZZ_{\geq 0}\ j\notin I}}\sum_{\substack{\xi\in \ZZ^S\\ v_{q_i}(\xi)=c_i \ i\in I}}\left(1+\left|\frac{\xi}{k}\right|_\infty\right)^{-2}\prod_{i=1}^{r} (1+|\xi|_{q_i})^{-3}\\
     \ll &1+\sum_{I\subseteq \{1,\dots,r\}}\prod_{i\in I}\sum_{\substack{c_i=-\infty\\ i\in I}}^{-1}\sum_{\xi\in \ZZ-\{0\}}\left(|\xi| \prod_{i\in I}q_i^{c_i}\right)^{-2}\prod_{i\in I}(1+q_i^{-c_i})^{-3}
     \\
     \ll &1+\sum_{\xi\in \ZZ-\{0\}}|\xi|^{-2}\sum_{I\subseteq \{1,\dots,r\}}\prod_{i\in I}\sum_{\substack{c_i=-\infty\\ i\in I}}^{-1}\frac{q_i^{2c_i}}{(1+q_i^{c_i})^{3}},
\end{align*}
which converges absolutely. Thus $\Omega$ satisfies (ii) and finally the theorem is proved.
\end{proof}

\section*{Acknowledgements}
I would like to thank \emph{Bin Xu} for giving me this project, the idea of Shalika germ, the idea of Poisson summation on the semilocal space and for telling me that the isolation in the ramified case is expected to be the sum of all $1$-dimensional representations rather than only the trivial representation as in the unramified case. I would also like to thank \emph{Taiwang Deng} for the discussion of orbital integrals and \emph{Yueke Hu} for discussions on Kloosterman sums and the approximate functional equation.

\bibliography{ref.bib}

\providecommand{\bysame}{\leavevmode\hbox to3em{\hrulefill}\thinspace}
\providecommand{\MR}{\relax\ifhmode\unskip\space\fi MR }
\providecommand{\MRhref}[2]{%
  \href{http://www.ams.org/mathscinet-getitem?mr=#1}{#2}
}
\providecommand{\href}[2]{#2}
\begin{thebibliography}{EELKW24}

\bibitem[Alt15]{altug2015}
Salim~Ali Altu\u{g}, \emph{Beyond endoscopy via the trace formula: 1. {P}oisson
  summation and isolation of special representations}, Compos. Math.
  \textbf{151} (2015), no.~10, 1791--1820.

\bibitem[Alt17]{altug2017}
\bysame, \emph{Beyond endoscopy via the trace formula, {II}: {A}symptotic
  expansions of {F}ourier transforms and bounds towards the {R}amanujan
  conjecture}, Amer. J. Math. \textbf{139} (2017), no.~4, 863--913.

\bibitem[Alt20]{altug2020}
\bysame, \emph{Beyond endoscopy via the trace formula---{III} {T}he standard
  representation}, J. Inst. Math. Jussieu \textbf{19} (2020), no.~4,
  1349--1387.

\bibitem[Alt24]{altug2024}
\bysame, \emph{Beyond endoscopy and the trace formula}, On the {L}anglands
  program---endoscopy and beyond, Lect. Notes Ser. Inst. Math. Sci. Natl. Univ.
  Singap., vol.~43, World Sci. Publ., Hackensack, NJ, 2024, pp.~315--377.

\bibitem[Art17]{arthur2017}
James Arthur, \emph{Problems beyond endoscopy}, Representation theory, number
  theory, and invariant theory, Progr. Math., vol. 323, Birkh\"auser/Springer,
  Cham, 2017, pp.~23--45.

\bibitem[Bou19]{bourbaki2019spectral}
Nicolas Bourbaki, \emph{{\'E}l\'ements de math\'ematique. {T}h\'eories
  spectrales. {C}hapitres 1 et 2}, second ed., Springer, Cham, 2019.

\bibitem[Clo89]{clozel1989}
Laurent Clozel, \emph{Orbital integrals on {$p$-adic} groups: A proof of the
  {Howe} conjecture}, Annals of Mathematics \textbf{129} (1989), no.~2,
  237--251.

\bibitem[DE26]{deng2026}
Taiwang Deng and Malors Espinosa, \emph{Beyond endoscopy for $\mathrm{GL}(3,
  \mathbb{Q})$: Poisson summation}, preprint (2026), \texttt{ArXiv:2603.21506}.

\bibitem[{\relax DLMF}]{DLMF}
\emph{{\it NIST Digital Library of Mathematical Functions}},
  \url{https://dlmf.nist.gov/}, 
  A.~B. {Olde Daalhuis}, D.~W. Lozier, B.~I. Schneider, R.~F. Boisvert, C.~W.
  Clark, B.~R. Miller, B.~V. Saunders, H.~S. Cohl, and M.~A. McClain, eds.

\bibitem[EELKW24]{emory2024}
Melissa Emory, Malors Espinosa-Lara, Debanjana Kundu, and Tian~An Wong,
  \emph{Beyond endoscopy via {Poisson} summation for {$GL(2,K)$}}, preprint
  (2024), \texttt{ArXiv:2404.10139}.

\bibitem[EL22]{espinosa2022}
Malors~Emilio Espinosa~Lara, \emph{Explorations on beyond endoscopy}, Ph.D.
  thesis, University of Toronto (Canada), 2022.

\bibitem[FL11]{finis2011}
Tobias Finis and Erez Lapid, \emph{On the {Arthur-Selberg} trace formula for
  {GL}(2)}, Groups, Geometry, and Dynamics \textbf{5} (2011), no.~2, 367--391.

\bibitem[FLN10]{langlands2010}
Edward Frenkel, Robert Langlands, and B\'ao~Ch\^au Ng\^o, \emph{Formule des
  traces et fonctorialit\'e: le d\'ebut d'un programme}, Ann. Sci. Math.
  Qu\'ebec \textbf{34} (2010), no.~2, 199--243.

\bibitem[Fol16]{folland2016abstract}
Gerald~B. Folland, \emph{A course in abstract harmonic analysis}, second ed.,
  Textbooks in Mathematics, CRC Press, Boca Raton, FL, 2016.

\bibitem[GH24]{getz2024}
Jayce~R. Getz and Heekyoung Hahn, \emph{An introduction to automorphic
  representations: with a view toward trace formulae}, Graduate Texts in
  Mathematics, vol. 300, Springer, 2024.

\bibitem[IK04]{iwaniec2004analytic}
Henryk Iwaniec and Emmanuel Kowalski, \emph{Analytic number theory}, American
  Mathematical Society Colloquium Publications, vol.~53, American Mathematical
  Society, Providence, RI, 2004.

\bibitem[Kot05]{kottwitz2005}
Robert~E. Kottwitz, \emph{Harmonic analysis on reductive {$p$}-adic groups and
  {L}ie algebras}, Harmonic analysis, the trace formula, and {S}himura
  varieties, Clay Math. Proc., vol.~4, Amer. Math. Soc., Providence, RI, 2005,
  pp.~393--522.

\bibitem[KT23]{kivinen2023}
Oscar Kivinen and Cheng-Chiang Tsai, \emph{Shalika germs for tamely ramified
  elements in {$GL_n$}}, preprint (2023), \texttt{ArXiv:2209.02509}.

\bibitem[Lan80]{langlands1980}
Robert~P. Langlands, \emph{Base change for {${\mathrm {GL}}(2)$}}, Annals of
  Mathematics Studies, vol. No. 96, Princeton University Press, Princeton, NJ;
  University of Tokyo Press, Tokyo, 1980.

\bibitem[Lan04]{langlands2004}
\bysame, \emph{Beyond endoscopy}, Contributions to automorphic forms, geometry,
  and number theory, Johns Hopkins Univ. Press, Baltimore, MD, 2004,
  pp.~611--697.

\bibitem[Lee26]{lee2026}
Yuchan Lee, \emph{Beyond endoscopy for $\mathrm{GL}_3(\mathbb{Q})$: Functional
  equation for the {$L$}-function of a cubic order}, 2026,
  \texttt{ArXiv:2607.03083}.

\bibitem[Ove14]{marius2014}
Marius Overholt, \emph{A course in analytic number theory}, Graduate Studies in
  Mathematics, vol. 160, American Mathematical Society, Providence, RI, 2014.

\bibitem[Sar01]{sarnak2001}
Peter Sarnak, \emph{Comments on {R}obert {L}anglands' lecture: "{E}ndoscopy and
  {B}eyond"}, 2001, available at
  \url{http://publications.ias.edu/sites/default/files/SarnakLectureNotes-1.pdf}.

\bibitem[Sha72]{shalika1972}
J.~A. Shalika, \emph{A theorem on semi-simple {$p$}-adic groups}, Annals of
  Mathematics. Second Series \textbf{95} (1972), 226--242.

\bibitem[She79]{shelstad1979}
D.~Shelstad, \emph{Orbital integrals for {${\mathrm {GL}}_{2}({\mathbf R})$}},
  Automorphic forms, representations and {$L$}-functions {P}art 1, Proc.
  Sympos. Pure Math., vol. XXXIII, Amer. Math. Soc., Providence, RI, 1979,
  pp.~107--110.

\bibitem[Tat67]{tate1950}
John~T. Tate, \emph{Fourier analysis in number fields and {Hecke}'s
  zeta-functions}, Algebraic Number Theory (J.W.S. Cassels and A.~Fr{\"o}lich,
  eds.), Academic Press, 1967, pp.~305--347.

\bibitem[Ten15]{tenenbaum2015analytic}
G{\'e}rald Tenenbaum, \emph{Introduction to analytic and probabilistic number
  theory}, Graduate Studies in Mathematics, vol. 163, American Mathematical
  Soc., 2015.

\bibitem[Ven04]{venkatesh2004}
Akshay Venkatesh, \emph{Beyond endoscopy and special forms on {GL(2)}}, Journal
  für die reine und angewandte Mathematik (2004), no.~577, 23--80.

\bibitem[Wei82]{weil1982adele}
Andr\'e Weil, \emph{Adeles and algebraic groups}, Progress in Mathematics,
  vol.~23, Birkh\"auser, Boston, MA, 1982.

\end{thebibliography}
\bibliographystyle{amsalpha}

\printindex

\end{document}